\documentclass[10pt, a4paper, UKenglish]{article}

\usepackage[UKenglish]{babel}

\usepackage{amssymb,amsmath,amsfonts,amsthm,nomencl,mathrsfs } 
\usepackage{bbm}
\usepackage[arrow, matrix, curve]{xy}
\usepackage[latin1]{inputenc}
\usepackage{a4wide}
\usepackage{color}
\usepackage{hyperref}
\usepackage{upref}
\usepackage{dirtytalk}
\usepackage{physics}
\usepackage{enumerate}
\usepackage{shuffle}
\usepackage{slashed}

\newcommand{\IC}{\mathbb{C}}
\newcommand{\IR}{\mathbb{R}}

\newcommand{\question}[1]{\leavevmode{\marginpar{\tiny%
			$\hbox to 0mm{\hspace*{-0.5mm}$\leftarrow$\hss}%
			\vcenter{\vrule depth 0.1mm height 0.1mm width \the\marginparwidth}%
			\hbox to 0mm{\hss$\rightarrow$\hspace*{-0.5mm}}$\\\relax\raggedright #1}}}

\newcommand{\rg}{\mathrm{rg}}

\renewcommand{\tr}{\mathrm{tr}}

\newcommand{\dom}{\mathrm{Dom}}

\newcommand{\IN}{\mathbb{N}}

\newcommand{\Ii}{\mathbbm{i}}

\newcommand{\IF}{\mathcal{F}}
\newcommand{\Id}{{\rm d}}

\newcommand{\p}{\partial}
\newcommand{\di}{\slashed{\p}}

\newcommand{\vol}{\mathrm{vol}}

\newcommand{\one}{\mathbbm{1}}

\DeclareMathOperator*{\supp}{supp}
\DeclareMathOperator{\coker}{coker}
\DeclareMathOperator{\ind}{ind}

\theoremstyle{plain}            
\newtheorem{theorem}{theorem}[section]
\newtheorem{Lemma}[theorem]{Lemma}
\newtheorem{Corollary}[theorem]{Corollary}
\newtheorem{Theorem}[theorem]{Theorem}
\newtheorem{Proposition}[theorem]{Proposition}
\newtheorem{Propandef}[theorem]{Proposition and Definition}

\theoremstyle{definition}
\newtheorem{Definition}[theorem]{Definition}

\newtheorem{Hypothesis}[theorem]{Hypothesis}
\newtheorem{Remark}[theorem]{Remark}
\newtheorem{Example}[theorem]{Example}

\title{Trace and Index of Dirac-Schr\"odinger Operators on Open Space with Operator Potentials}
\author{Oliver F\"urst\footnote{\textsc{O. F\"urst, Mathematisches Institut, Rheinische Friedrich-Wilhelms-Universit\"at Bonn, Endenicher Allee 60, 53115 Bonn}, \url{ofuerst@math.uni-bonn.de}}}
\date{13th December 2024}

\begin{document}
	\maketitle
	
	\begin{abstract}
		We develop a principal trace and generalized index formula for a Dirac-Schr\"odinger operator $D$ on open space of odd dimension $d\geq 3$ with a potential given by a family of self-adjoint unbounded operators acting on a infinite dimensional Hilbert space $H$. The presented results generalize formulas surrounding the Callias index theorem to the case of unbounded operator potentials, for which the operator $D$ is not necessarily Fredholm. This is the principal novelty of this paper. As application, we include examples where the trace formula is used to calculate the Witten index of non-Fredholm massless $(d+1)$-Dirac-Schr\"odinger operators acting in $L^2\left(\IR^{d+1},H\right)$.
	\end{abstract}
	
	\section{Introduction}
	The central goal of this paper is to calculate the trace (called in the following the "principal trace") $\tr_{L^2\left(\IR^{d},H\right)}\tr_{\IC^{r}}\left(e^{-tD^{\ast}D}-e^{-tDD^{\ast}}\right)$, $t>0$, of the Dirac-Schr\"odinger operator $D=\Ii\di+A$, and the associated index formula, under suitable conditions on the operator-valued potential $A$.
	
	Let us start with the setup, explaining the notation above. We fix $\Ii$ as the imaginary unit in $\IC$. Let $H$ be a separable complex Hilbert space, and let $A_{0}$ be a self-adjoint operator in $H$ with domain $\dom\ A_{0}$, which plays the role of a "model operator". We equip domains of (Borel) functions $f$ of $A_{0}$ with the corresponding graph norm, which turns $\dom\ f\left(A_{0}\right)$ into a separable Hilbert space, embedding continuously into $H$ with norm $1$. Fix $d\in\IN$, $d\geq 3$ to be the odd dimension of the euclidean space $\IR^{d}$, and let $A$ be a function on $\IR^{d}$ with values in the self-adjoint operators in $H$ with constant domain $\dom\ A\left(x\right)=\dom\ A_{0}$, $x\in\IR^{d}$. Denote by $M_{0}$ the constant operator function $\IR^{d}\ni x\mapsto A_{0}$. Fix $r\in\IN$ as the rank (the minimal choice is $r=2^{\frac{d-1}{2}}$) of Clifford matrices $\left(c^{i}\right)_{i=1}^{d}\subset\IC^{r\times r}$, i.e. for $i,j\in\left\{1,\ldots,d\right\}$,
	\begin{align}\label{clifdefeq1}
		c^{i}c^{j}+c^{j}c^{i}=-2\delta_{ij}\one_{\IC^{r}}.
	\end{align}
	For $p\in\left[1,\infty\right]$ denote by $L^p\left(\IR^{d},\IC^{r}\otimes H\right)$ the $\IC^{r}\otimes H$-valued $L^{p}$ elements, and with slight abuse of notation, denote by $A$ and $M_{0}$ the multiplication operators in the Hilbert space $L^2\left(\IR^{d},\IC^{r}\otimes H\right)$, given by
	\begin{align}
    \left(Af\right)\left(x\right):=&A\left(x\right)f\left(x\right),\ x\in\IR^{d},\ f\in\dom\ A,\nonumber\\
		\left(M_{0}f\right)\left(x\right):=&A_{0}f\left(x\right),\ x\in\IR^{d},\ f\in\dom\ M_{0},\nonumber\\
		\dom\ A:=&\dom\ M_{0}:=L^2\left(\IR^{d},\IC^{r}\otimes\dom\ A_{0}\right).
	\end{align}
	Denote by $H^k\left(\IR^{d},\IC^{r}\otimes H\right)$ the $L^2$-Sobolev space of order $k\in\IN_{0}$ with values in $\IC^{r}\otimes H$. Let $\di$ be the Dirac operator in $\IC^{r}\otimes H$ given by
	\begin{align}
		\left(\di f\right)\left(x\right):=&\sum_{i=1}^{d}c^{i}\left(\p_{i}f\right)\left(x\right),\ x\in\IR^{d},\ f\in\dom\ \di,\nonumber\\
		\dom\ \di:=& H^1\left(\IR^{d},\IC^{r}\otimes H\right).
	\end{align}
	It is well-known that $\di$ is self-adjoint, with $\Delta=\di^{2}=\Delta_{\IR^{d}}\otimes\one_{\IC^{r}\otimes H}$, where $\Delta_{\IR^{d}}$ is the self-adjoint euclidean Laplace operator $\Delta_{\IR^{d}}:=-\sum_{i=1}^{d}\p_{i}^{2}$, $\dom \Delta_{\IR^{d}}:= H^2\left(\IR^{d}\right)$. In $L^2\left(\IR^{d},\IC^{r}\otimes H\right)$ introduce the operators $D,D_{0}$ given by
	\begin{align}
		D:=&\Ii\di+A,\ D_{0}:=\Ii\di+M_{0},\\
		\dom\ D=&\dom\ D_{0}=\dom\ \di\cap\dom\ M_{0}.
	\end{align}
	$D_{0}$ is normal with $D_{0}^{\ast}=-\Ii\di+M_{0}$, and $\dom\ D_{0}^{\ast}=\dom\ D_{0}$ (cf. \cite{GLMST}[Lemma 4.2]). We introduce the non-negative operator $H_{0}$ given by
	\begin{align}
		H_{0}:=&D_{0}D_{0}^{\ast}=D_{0}^{\ast}D_{0}=\Delta+M_{0}^{2},\nonumber\\
		\dom\ H_{0} =&H^{2}\left(\IR^{d},\IC^{r}\otimes H\right)\cap\dom\ M_{0}^{2}.
	\end{align}
	In the following preparatory section we will give conditions on $A$, such that the domain conditions $\dom\ DD^{\ast}=\dom\ D^{\ast}D=\dom\ H_{0}$, $\dom\ D=\dom\ D^{\ast}=\dom\ D_{0}$ are satisfied. We will also show that under further assumptions on the function $A$, akin to the asymptotic homogeneity condition imposed classically in Callias original paper (\cite{Cal}[Theorem 2]), the partial trace $\tr_{\IC^{r}}\left(e^{-tD^{\ast}D}-e^{-tDD^{\ast}}\right)$ is trace-class in $L^2\left(\IR^{d},H\right)$, for $t>0$.
    After introducing a suitable approximation of the operator function $A$, the bulk of this paper then deals with the calculation of the principal trace.

    Let us now present the conditions we impose on the operator function $A=B+M_{0}$, and discuss some of the details. We use the following notation: $\langle x\rangle_{z}:=\left(\left|x\right|^{2}+z\right)^{\frac{1}{2}}$, and $\langle x\rangle:=\langle x\rangle_{1}$.
	
	\begin{Hypothesis}[Hypothesis \ref{hyp1}]
    Let $\alpha\geq 1$, and $N\geq\lfloor\frac{\alpha-1}{2}\left(d+1\right)\rfloor+1$. Let $B$ be a family of symmetric operators in $H$, with domains $\dom\ A_{0}$. Assume that $B\langle A_{0}\rangle^{-1}$ is continuously differentiable in the strong operator topology, and that for all $\beta\in\left[-2N-1,2N+1\right]$ the operators $\langle A_{0}\rangle^{\beta}B\left(x\right)\langle A_{0}\rangle^{-1-\beta}$ are closable for all $x\in\IR^{d}$ with bounded closures and
    \begin{align}\label{introdiffeq}
        \left\|\overline{\langle M_{0}\rangle^{\beta}B\langle M_{0}\rangle^{-1-\beta}}\right\|_{L^{\infty}\left(\IR^{d},B\left(H\right)\right)}<\infty.
    \end{align}
    Assume that for $\beta=\pm2N$,
    \begin{align}\label{introdomeq}
        &\left\|\langle M_{0}\rangle^{\beta}B\langle M_{0}\rangle^{-\beta}\langle M_{0}\rangle^{-1}_{z}\right\|_{L^{\infty}\left(\IR^{d},B\left(H\right)\right)}+\left\|\langle M_{0}\rangle^{1+\beta}B\langle M_{0}\rangle^{-1-\beta}\langle M_{0}\rangle^{-1}_{z}\right\|_{L^{\infty}\left(\IR^{d},B\left(H\right)\right)}\nonumber\\
		&+\left\|\langle M_{0}\rangle^{\beta}\nabla B\langle M_{0}\rangle^{-1-\beta}\langle M_{0}\rangle^{-1}_{z}\right\|_{L^{\infty}\left(\IR^{d},B\left(H\right)^{d}\right)}\xrightarrow{w\to\infty}0,\ w:=\operatorname{dist}\left(z,\left(-\infty,0\right]\right).
    \end{align}
		Further assume that for all $\beta\in\left[-2N+\alpha,2N+1\right]$, there exists $\epsilon>0$, such that
		\begin{align}\label{introtraceeq}
            &\left\|x\mapsto\langle x\rangle\langle A_{0}\rangle^{-\alpha+\beta}\left(\nabla B\right)\left(x\right)\langle A_{0}\rangle^{-\beta}\right\|_{L^{\infty}\left(\IR^{d},S^{d}\left(H\right)^{d}\right)}\nonumber\\
            &+\left\|x\mapsto\langle x\rangle^{1+\epsilon}\langle A_{0}\rangle^{-\alpha+\beta}\left(\p_{R}B\right)\left(x\right)\langle A_{0}\rangle^{-\beta}\right\|_{L^{\infty}\left(\IR^{d},S^{d}\left(H\right)\right)}<\infty,
		\end{align}
		where $\p_{R}:=\sum_{i=1}^{d}\frac{x^{i}}{\left|x\right|}\p_{i}$ denotes the radial vector field, and $S^{d}\left(H\right)$ the $d$th Schatten-von Neumann class of $H$.
        
        Finally, let $S_{1}\left(0\right)$ denote the $\left(d-1\right)$-dimensional unit sphere, and assume that for $\phi\in\dom\ A_{0}$, a.e. $y\in S_{1}\left(0\right)$, a.e. $x\in\IR^{d}$,
		\begin{align}\label{introradlimeq}
			\lim_{R\to\infty}R\left\|\left(\nabla B\left(Ry+x\right)-\nabla B\left(Ry\right)\right)\phi\right\|_{H^{d}}=\lim_{R\to\infty}\left\|\left(B\left(Ry+x\right)-B\left(Ry\right)\right)\phi\right\|_{H}=0.
		\end{align}
	\end{Hypothesis}

	Let us give an overview over the above conditions. We first note that the parameter $\alpha\geq 1$ is central. It allows flexibility in the most important condition (\ref{introtraceeq}) where a larger $\alpha$ means a weaker requirement on $B$. It should be noted however that this comes at a price, namely that $N$ has to be increased. The parameter $N$ determines how regular $B$ needs to be on the Hilbert scale generated by $A_{0}$ (cf. conditions (\ref{introdiffeq}) and (\ref{introdomeq}). It should however be noted that the regularity of $B$ on $\IR^{d}$ needs to be in any case just once continuously differentiable, independent of the choice of $\alpha$, and thus $N$.
    The condition (\ref{introdomeq}) encodes Kato-Rellich bounds on the family $B$, which in particular imply that all $A\left(x\right)=B\left(x\right)+A_{0}$ have constant domain $\dom\ A_{0}$. In rough terms condition (\ref{introdomeq}) ensures that the domain properties of the involved operators are well-behaved.
    The condition (\ref{introtraceeq}) stipulates the (relative) Schatten-von Neumann membership of the derivatives of the operator function $B$ and their decay for $\left|x\right|\to\infty$. The radial derivative is singled out, since we assume that it decays faster. These decay conditions correspond to the assumption of asymptotic homogeneity of order $0$ in Callias' classical setup \cite{Cal}[Theorem 2]. The final condition (\ref{introradlimeq}) also refers to the behaviour of $B$ at infinity. Roughly speaking, it implies that it should not matter on which ray one approaches a fixed direction at infinity of the functions $B$ and $\nabla B$.

    It should be remarked that while the above conditions are quite general, one should also emphasize that the possibility of arbitrarily choosing $\alpha\geq 1$ is not only interesting for generality but also because of the requirements from applications. If one wants to incorporate the choice of a proper differential operator of order $1$ on $\IR^{n}$ for $A_{0}$ itself ($H$ is then $L^2\left(\IR^{n}\right)$), one runs into problems with condition (\ref{introtraceeq}) if one is just allowed to take $\alpha=1$. Due to well-known Schatten-von Neumann estimates (cf. Chapter 4 in \cite{Simon}) this issue boils basically down to the fact that $x^{-n}$ is not integrable on $\IR^{n}$ at infinity. We note that at the end of this paper we include as an example results dealing with exactly this sketched setup, which stem from \cite{F2}, which gives the detailed proofs in this particular setup, which we omit in this paper.

    After explaining the conditions imposed on the operator potential $A$, let us present the main result of this paper, which is the desired principal trace formula in terms of $A$ and $\Id A$ in a neighbourhood of infinity.
	\begin{Theorem}[Theorem \ref{newmainthm}]\label{introthm}
		Assume Hypothesis \ref{hyp1}. For $\phi\in C^{\infty}_{c}\left(\IR^{d}\right)$ denote $A_{\phi}:=A_{0}+\left(1-\phi\right)\left(A-A_{0}\right)$. Then for $t>0$, and any $\phi\in C^{\infty}_{c}\left(\IR^{d}\right)$,
		\begin{align}\label{introeq1}
			&\tr_{L^{2}\left(\IR^{d},H\right)}\tr_{\IC^{r}}\left(e^{-tD^{\ast}D}-e^{-tDD^{\ast}}\right)\nonumber\\
			=&\frac{2}{d}\left(4\pi\right)^{-\frac{d}{2}}\Ii^{d}\kappa_{c}t^{\frac{d}{2}}\int_{s\in\Delta_{d-1}}\int_{\IR^{d}}\tr_{H}\bigwedge_{j=0}^{d-1}\left(\Id A_{\phi}e^{-ts_{j}A_{\phi}^2}\right)\Id s,
		\end{align}
		where (\ref{introeq1}) is independent of the choice of $\phi$, and
		\begin{align}
			\kappa_{c}:=\tr_{\IC^{r}}\left(c^{1}\cdot\ldots\cdot c^{d}\right),
		\end{align}
		$\Delta_{d-1}=\left\{u\in\left[0,1\right]^{d},\sum_{j=0}^{d-1}u_{j}=1\right\}$ denotes the $\left(d-1\right)$-simplex, and $\bigwedge$ the exterior product.
	\end{Theorem}
	
	We note that formula (\ref{introeq1}) only depends on the values of $A$ and $\Id A$ in a neighbourhood of infinity due to the cut-off $\phi$. If $A$ and $\Id A$ additionally assume radial limits in a particular sense, we show that only these limits appear in the trace formula.
	
	\begin{Corollary}[Corollary \ref{newcor}]
		Assume Hypothesis \ref{hyp1}, and denote $A^{\circ}\left(y\right)\psi:=\lim_{R\to\infty}A\left(Ry\right)\psi$, for $\psi\in\dom\ A_{0}$, and $y\in S_{1}\left(0\right)$. Assume that $A^{\circ}\psi$ is differentiable on $S_{1}\left(0\right)$ for $\psi\in\dom\ A_{0}$, such that for a.e. $y\in S_{1}\left(0\right)$,
		\begin{align}
			\lim_{R\to\infty}R\left\|\p_{i}A^{\circ}\left(y\right)\psi-\p_{i}A\left(Ry\right)\psi\right\|_{H}=0.
		\end{align}
		and that for $\beta\in\left[-2N+\alpha,2N+1\right]$, and all unit vector fields $\nu$ in $TS_{1}\left(0\right)$,
		\begin{align}
			\langle A_{0}\rangle^{-\alpha+\beta}\left(\nu A^{\circ}\right)\langle A_{0}\rangle^{-\beta}\in L^{\infty}\left(S_{1}\left(0\right),S^{d}\left(H\right)\right).
		\end{align}
		Denote for $\rho\in C^{\infty}_{c}\left(\IR\right)$,
		\begin{align}
			\widetilde{A^{\circ}}_{\rho}\left(x\right)&:=A_{0}+\left(1-\rho\left(\left|x\right|\right)\right)\left(A^{\circ}\left(\frac{x}{\left|x\right|}\right)-A_{0}\right).
		\end{align}
		Then for $t>0$, and any $\rho\in C^{\infty}_{c}\left(\IR\right)$, with $\rho\equiv 1$ near $0$,
		\begin{align}
			&\tr_{L^{2}\left(\IR^{d},H\right)}\tr_{\IC^{r}}\left(e^{-tD^{\ast}D}-e^{-tDD^{\ast}}\right)\nonumber\\
			=&\frac{2}{d}\left(4\pi\right)^{-\frac{d}{2}}\Ii^{d}\kappa_{c}t^{\frac{d}{2}}\int_{s\in\Delta_{d-1}}\int_{\IR^{d}}\tr_{H}\bigwedge_{j=0}^{d-1}\left(\left(\Id\widetilde{A^{\circ}}_{\rho}\right)e^{-ts_{j}\widetilde{A^{\circ}}_{\rho}^2}\right)\Id u,
		\end{align}
	independent of the choice of $\rho$.
	\end{Corollary}
	
	As a direct application of the principal trace formula of Theorem \ref{introthm} we may calculate the Fredholm index of $D$, whenever $D$ is Fredholm.
	
	\begin{Theorem}[Theorem \ref{fredindexthm}]\label{introfredthm}
		Assume Hypothesis \ref{hyp1}. Assume that there exists $R_{0}\geq 0$, such that for $\left|x\right|\geq R_{0}$, we have $0\in\rho\left(A\left(x\right)\right)$. Then $D$ is Fredholm, and the Fredholm index $\ind D$ is given by
		\begin{align}\label{introfredthmeq1}
			\ind D=\frac{2}{d}\left(4\pi\right)^{-\frac{d}{2}}\Ii^{d}\kappa_{c}\lim_{t\to\infty}t^{\frac{d}{2}}\int_{s\in\Delta_{d-1}}\int_{\IR^{d}}\tr_{H}\bigwedge_{j=0}^{d-1}\left(\left(\Id A_{\phi}\right)e^{-ts_{j}A_{\phi}^2}\right)\Id s,
		\end{align}
		where $A_{\phi}:=A_{0}+\left(1-\phi\right)\left(A-A_{0}\right)$ for any $\phi\in C^{\infty}_{c}\left(\IR^{d}\right)$. Formula (\ref{introfredthmeq1}) is independent of the choice of $\phi$.
	\end{Theorem}
	
	The principal trace formula of Theorem \ref{introthm}, can also be applied to determine some cases a generalized version of index, the so called (partial) Witten index, of the operator $D$. The notion was introduced in \cite{GesSim}. We give here the definition of the semi-group regularized partial Witten index.
	
	\begin{Definition}
		Let $T$ be a linear, densely defined operator $\IC^{r}\otimes X$ for some separable complex Hilbert space $X$. We say the partial (semi-group regularized) Witten index of $T$ exists, if for some $t_{0}>0$, we have
		\begin{align}
			\tr_{\IC^{r}}\left(e^{-t T^{\ast}T}-e^{-tTT^{\ast}}\right)\in S^1\left(X\right),\ t\geq t_{0},
		\end{align}
		and if the limit
		\begin{align}
			\ind_{W}T:=\lim_{t\to\infty}\tr_{X}\left(\tr_{\IC^{r}}\left(e^{-t T^{\ast}T}-e^{-tTT^{\ast}}\right)\right),
		\end{align}
		exists, then $\ind_{W}T$ is called the partial Witten index of $T$.
	\end{Definition}
	
	The (partial) Witten index may assume any real number (cf. \cite{BolGesGroSchSim}, \cite{CGGLPSZ},\cite{CGLPSZ} in dimension $d=1$, and \cite{F2} for $d\geq 3$), however if $T$ is additionally Fredholm, it coincides with the usual Fredholm index (see Remark \ref{wittenrem}).
    
	We conclude this paper, as we mentioned before, with examples of partial Witten index formulas taken from \cite{F2} without their proofs, dealing with massless $\left(d+1\right)$-Dirac-Schr\"odinger operators. These operators are non-Fredholm, but possess partial Witten index formulas given in Example \ref{theexample}, and for a special choice of potential in Example \ref{exampleminor}. The point of these examples is that they extend the known Witten index formula, given directly in terms of the underlying potential, in the case $d=1$ (cf. \cite{BolGesGroSchSim},\cite{CGGLPSZ},\cite{CGLPSZ}), to $d\geq 3$ odd. In this presented setup the model operator $A_{0}=\Ii\p$ is just the derivative in one direction, and $B$ is given by a (matrix) potential over $\IR^{d+1}$.
	
	\begin{Definition}
		For a potential $V:\IR^{d}\times\IR\rightarrow B_{sa}\left(H\right)$, define the associated \textit{massless $\left(d+1\right)$-Dirac-Schr\"odinger operator} $D_{V}$,
		\begin{align}
			\left(D_{V}f\right)\left(x,y\right):=&i\sum_{j=1}^{d}c^{j}\p_{x^{j}}f\left(x,y\right)+i\p_{y}f\left(x,y\right)+V\left(x,y\right)f\left(x,y\right),\ x\in\IR^{d},\ y\in\IR,\nonumber\\
			f\in& W^{1,2}\left(\IR^{d+1},\IC^{r}\otimes H\right),
		\end{align}
	\end{Definition}
	
	Without going into the details of the assumptions on the potential $V$, we then have the index formula:
	
	\begin{Example}[Example \ref{theexample}]
		Assume that $V$ satisfies the requirements given in Example \ref{theexample}. Then
		\begin{align}
			\ind_{W}D_{V}=\frac{1}{2\pi}\left(4\pi\right)^{-\frac{d-1}{2}}\frac{\left(\frac{d-1}{2}\right)!}{d!}\kappa_{c}\int_{\IR^{d}}\tr_{H}\left(\left(U^{V}\right)^{-1}\Id U^{V}\right)^{\wedge d},
		\end{align}
		where $\wedge d$ is the $d$-fold exterior power,
		\begin{align}
			\kappa_{c}:=\tr_{\IC^{r}}\left(c^{1}\cdot\ldots\cdot c^{d}\right),
		\end{align}
		and the unitary $U^{V}$ is given by $U^{V}\left(x\right):=\lim_{n\to\infty}U^{V\left(x,\cdot\right)}\left(n,-n\right)$, $x\in\IR^{d}$, where $U^{A}\left(y_{1},y_{2}\right)$, $y_{1},y_{2}\in\IR$, for a given $A:\IR\rightarrow B_{sa}\left(H\right)$ is the (unique) evolutions system of the equation
		\begin{align}
			u'\left(y\right)=\Ii A\left(y\right) u\left(y\right),\ y\in\IR.
		\end{align}
	\end{Example}
	
	Before we proceed, let us outline the historical context of the problem. A differential operator of the type $D=\Ii\di+A$ was discussed prominently by Atiyah, Patodi and Singer in their seminal series of papers \cite{APS1}\cite{APS2}\cite{APS3}. Therein $A$ is a family over $\IR$ of first order, elliptic, differential operators on a compact, odd dimensional manifold with invertible asymptotic endpoints. They showed that $D$ is Fredholm with Fredholm index $\ind D=\mathrm{sflow}\left(A,0\right)$, where the right hand side is the spectral flow of $A$ through $0$. The one-dimensional case, in which $A$ is parametrized over $\IR$, and $\Ii\di$ is just the derivative along $\IR$, was further discussed by Robbin and Salamon \cite{RobSal} for more general types of unbounded operator families $A$, however still with $A\left(x\right)$, $x\in\IR$, having discrete spectra. One should also note \cite{BolGesGroSchSim}, where some of the notions about generalized index, and spectral flow, appeared. In 2008, Pushnitski showed in \cite{Push} that discreteness of the spectrum of $A$ can, in essence, be replaced by the following assumption: Assume that $A=A_{-}+B$, where $A_{-}$ is a generic self-adjoint operator in a separable Hilbert space, and the derivative $B'$ of the family $B$ of bounded, symmetric operators is an integrable family of trace-class operators in $H$, i.e.
	\begin{align}\label{introeq4}
		\int_{\IR}\left\|B'\left(x\right)\right\|_{S^1\left(H\right)}<\infty.
	\end{align}
	Under these conditions $D$ is Fredholm with
	\begin{align}
		\ind D=\xi\left(0,A_{+},A_{-}\right),
	\end{align}
	where $A_{\pm}=\lim_{x\to\pm\infty}A\left(x\right)$ are the (self-adjoint) limit endpoints in a certain sense, and $\xi\left(\cdot,A_{+},A_{-}\right)$ is the associated spectral shift function, a concept from scattering theory (cf. \cite{Yaf}) which we will not delve into here, however it should be noted, that if the spectral shift function of the endpoints and the spectral flow of a (regular enough) family exist, then they coincide (cf. also \cite{BolGesGroSchSim}). However Pushnitski's results go beyond a generalization of the index problem. In fact, a principal trace formula (which in general carries more information on $D$ than the Fredholm index), already present in Callias paper \cite{Cal}, was shown in this more general setup:
	\begin{align}\label{introeq5}
		\tr_{L^2\left(\IR,H\right)}\left(\left(DD^{\ast}-z\right)^{-1}-\left(D^{\ast}D-z\right)^{-1}\right)=\frac{1}{2z}\tr_{H}\left(A_{+}\left(A_{+}^2-z\right)^{-\frac{1}{2}}-A_{-}\left(A_{-}^2-z\right)^{-\frac{1}{2}}\right),
	\end{align}
	for $z\in\IC\backslash\left[0,\infty\right)$. In the following years, and also quite recently, more results emerged similar to Pushnitski's trace based approach to the index problem, for example condition (\ref{introeq4}) was replaced with
	\begin{align}\label{introeq6}
		\int_{\IR}\left\|B'\left(x\right)\left(A_{-}^{2}+1\right)^{-\frac{1}{2}}\right\|_{S^1\left(H\right)}<\infty,
	\end{align}
	by Gesztesy et al. in \cite{GLMST}, while (\ref{introeq5}) still holds. Further generalizations along these lines have been discussed in \cite{Car}, and \cite{CarGesLevSuk}. We remark that the principal trace formula in these publications (\ref{introeq5}) corresponds to the case $d=1$ in the principal trace formula (\ref{introeq1}) presented in this paper, if one rewrites the right hand side of (\ref{introeq5}) with an appropriate integral (the endpoints $A_{\pm}$ can be seen as the endpoints of the function $A(t)=A_{-}+t(A_{+}-A_{-})$, an application of the fundamental theorem of calculus then yields the integrated form), and turned via the inverse Laplace transform into an exponentiated version instead of the resolvent version (\ref{introeq5}):
    \begin{align}
        \tr_{L^2\left(\IR,H\right)}\left(e^{-tD^{\ast}D}-e^{-tDD^{\ast}}\right)=\sqrt{\frac{t}{\pi}}\int_{0}^{1}\tr_{H}\left(e^{-t A(s)^{2}}\left(A_{+}-A_{-}\right)\right)\Id s.
    \end{align}
	
	Another important result, extending the classical "Index = Spectral Flow" Theorem, was provided by Callias in the paper \cite{Cal}. While the operator family $A$ acts there in a finite dimensional space, the index formula therein holds for $\di$ being the classical Dirac operator in $\IR^{d}$ for $d$ odd, which is also our setup here. The index theorem (\cite{Cal}[Theorem 2]) then reads as follows: Let $k=\frac{d-1}{2}$. With the minimal rank choice $r=2^{k}$ for the Clifford matrices, let $A$ be a smooth family of $m\times m$-matrices with $\left|A\left(x\right)\right|\geq C>0$ for all $\left|x\right|\geq c$ for some constants $c$, $C$. Assume that $A$ is asymptotically homogeneous of order $0$ as $x\to\infty$ (i.e. $A\left(x\right)$ is a $0$-order classic symbol in $x$). Then $D$ is Fredholm, and, with $U\left(x\right):=A\left(x\right)\left|A\left(x\right)\right|^{-1}$,
	\begin{align}\label{introeq7}
		\ind D=\frac{1}{2 k!}\left(\frac{\Ii}{8\pi}\right)^{k}\lim_{R\to\infty}\int_{S_{R}}\tr_{\IC^{m}}U\left(\Id U\right)^{\wedge\left(d-1\right)},
	\end{align}
	here, $S_{R}$ denotes the $d-1$-sphere of radius $R$ around $0$, and $\left(\Id U\right)^{\wedge\left(d-1\right)}$ the $d-1$-fold exterior product of (matrix-)one-forms $\Id U$. One should note that there are some problems with the trace-class membership of certain operators in \cite{Cal} in the original proof\footnote{This may be due to the fact that shortly after Callias published his results a different proof was presented by Bott and Seeley in \cite{BotSee}, which Callias already mentions in the introduction of his paper.}, however the ideas have been made rigorous (and generalized) by Gesztesy and Waurick in \cite{GesWau}.	
	
	Although beyond the scope of this paper, one should also mention that there exist a multitude of generalizations of Callias' index theorem to non-compact Riemannian manifolds (cf.\cite{Ang1},\cite{BruMos},\cite{Ang2},\cite{Rad}, \cite{Bun},\cite{Kuc},\cite{CarNis},\cite{Kot},\cite{BraShi},\cite{Shi},\cite{Dun1},\cite{Cec},\cite{SchSto},\cite{Dun2}), however still in the context of Fredholmness and not providing a general principal trace formula. The usual assumption therein is that the operator valued potential $A$ is invertible in some sense outside a compact region of the manifold, to ensure Fredholmness of $D$. We emphasize that we do not assume such a condition in this paper.
	
	With this perspective, the principal trace formula (\ref{introeq1}) stands in analogy to (\ref{introeq7}), similar in spirit to the  generalizations developed by Pushniski in the one dimensional case for the "Index=Spectral Flow"-Theorem, especially if one no longer requires $D$ to be Fredholm.
	
	\section{Trace-class properties of the resolvents of $D^{\ast}D$ and $DD^{\ast}$}
	
	In this section, we present our conditions on the family $A$ and show domain, and trace-class properties. We start however by introducing further notation.
	
	Denote by $S^{p}\left(H\right)$ the $p$-th Schatten-von Neumann class of $H$, and with $B\left(H\right)$ the linear bounded operators of $H$. Denote $\langle T\rangle_{z}:=\left(T^{\ast}T+z\right)^{\frac{1}{2}}$, $z\in\IC\backslash\left(-\infty,0\right]$, and $\langle T\rangle:=\langle T\rangle_{1}$, for an operator $T$ in a Hilbert space. Denote for a measurable function $g:\IR\to\IR$ by $g\left(X\right)$ the multiplication operator in $L^2\left(\IR^{d},\IC^{r}\otimes H\right)$, given by $\left(g\left(X\right)f\right)\left(x\right):=g\left(x\right)f\left(x\right)$, $x\in\IR^{d}$. For $z\in\rho\left(T\right)$, denote $R_{z}\left(T\right):=\left(T+z\right)^{-1}$.
	
	The following simple relation between trace and Clifford matrices is essential.
	 
	\begin{Lemma}\label{cliflem}
		\begin{enumerate}
			\item For $n\in\IN$ and $\alpha\in\left\{1,\ldots,d\right\}^{n}$,
			\begin{align}
				\tr_{\IC^{r}}\left(\prod_{j=1}^{n}c^{\alpha_{j}}\right)=0,
			\end{align}
			if ($n$, $d$ are odd, and $n<d$) or ($n$ is odd, and $d$ is even).
			\item For $d$ odd, and $\alpha\in\left\{1,\ldots,d\right\}^{d}$,
			\begin{align}\label{cliflemeq2}
				\tr_{\IC^{r}}\left(\prod_{j=1}^{d}c^{\alpha_{j}}\right)=\kappa_{c}\epsilon_{\alpha},
			\end{align}
			where $\epsilon_{\alpha}=\epsilon_{\alpha_{1}\ldots\alpha_{d}}$ is the Levi-Civita symbol, and
			\begin{align}
				\kappa_{c}:=\tr_{\IC^{r}}\left(\prod_{j=1}^{d}c^{j}\right).
			\end{align}
			If $r=2^{\frac{d-1}{2}}$, i.e. the minimal choice of $r$ to satisfy the anti commutator relation (\ref{clifdefeq1}), then
			\begin{align}\label{cliflemeq3}
				\kappa_{c}=\left(2\Ii\right)^{\frac{d-1}{2}}\left(-\Ii\right)^{d}.
			\end{align}
		\end{enumerate}
	\end{Lemma}
	
	\begin{proof}
		\begin{enumerate}
			\item Due to the anti commutator relation (\ref{clifdefeq1}), it is sufficient to consider only those $\alpha$ for which all $\alpha_{j}$, $j\in\left\{1,\ldots,n\right\}$, are pairwise different. Since $n$ is odd, we have $n<d$ in any case. Thus there exists $m\in\left\{1,\ldots,d\right\}$ such that $m\neq\alpha_{j}$, $j\in\left\{1,\ldots,n\right\}$. Then
			\begin{align}
				\tr_{\IC^{r}}\left(\prod_{j=1}^{n}c^{\alpha_{j}}\right)&=-\tr_{\IC^{r}}\left(\left(c^{m}\right)^{2}\prod_{j=1}^{n}c^{\alpha_{j}}\right)\nonumber\\
				&=-\tr_{\IC^{r}}\left(c^{m}\prod_{j=1}^{n}c^{\alpha_{j}}c^{m}\right)=\left(-1\right)^{n+1}\tr_{\IC^{r}}\left(\left(c^{m}\right)^{2}\prod_{j=1}^{n}c^{\alpha_{j}}\right)\nonumber\\
				&=-\tr_{\IC^{r}}\left(\prod_{j=1}^{n}c^{\alpha_{j}}\right).
			\end{align}
			\item Statement (\ref{cliflemeq2}) follows immediately from the anti commutator relation (\ref{clifdefeq1}), the definition of $\epsilon_{\alpha}$, and the first part of this proof. The case of $r=2^{\frac{d-1}{2}}$ follows by induction in $d$ and the representation of Clifford matrices of odd dimension $d$ by those of dimension $d-2$. The sign convention corresponds to the choice $c^{j}=-\Ii\sigma^{j}$, where $\sigma^{j}$ are the classical Pauli matrices.
		\end{enumerate}
	\end{proof}

\begin{Definition}
	For $B$ a function on $\IR^{d}$ with values in symmetric operators in $H$ with common domain $\dom\ A_{0}$, let $D=D_{B}:=\Ii\di+B+M_{0}$, $\dom\ D_{B}:=\dom\ D_{0}$, and $H_{B}:=\Delta+\left(B+M_{0}\right)^{2}$, $\dom\ H_{B}:=\dom\ H_{0}$.
\end{Definition}

\begin{Definition}
	For $B$ a operator function on $\IR^{d}$ with domains $\dom\ B\left(x\right)=\dom\ A_{0}$, we write $B\in C^{k}_{b}\left(\IR^{d},A_{0},m\right)$, for $k\in\IN_{0}$, $m\geq 0$, if $\langle A_{0}\rangle^{\beta}B\left(x\right)\langle A_{0}\rangle^{-1-\beta}$, $x\in\IR^{d}$, is a family of closable operators with closures in $B\left(H\right)$ for any $\beta\in\left[-m,m\right]$, and $x\mapsto\overline{\langle A_{0}\rangle^{\beta}B\left(x\right)\langle A_{0}\rangle^{-1-\beta}}\in C^{k}_{b}\left(\IR^{d},B\left(H\right)\right)$, where $C^{k}_{b}$ is the the space of $k$-times continuously differentiable bounded functions in the strong operator topology with bounded derivatives, where boundedness is with respect to the $B\left(H\right)$ operator norm. As usual, $C^{k}_{b}\left(\IR^{d}\right)$ just denotes the scalar/matrix valued-$C^{k}_{b}$-functions. All involved spaces come equipped with the obvious norm given by summing the appropriate $L^{\infty}$-norms of the derivatives up to order $k$.
\end{Definition}

\begin{Definition}
	We introduce the following semi-norms for operators $B\in C^{1}_{b}\left(\IR^{d},A_{0},1\right)$, with $z\in\IC\backslash\left(-\infty,0\right]$,
	\begin{align}
		\rho_{z}\left(B\right):=&\left\|B\langle M_{0}\rangle^{-1}_{z}\right\|_{L^{\infty}\left(\IR^{d},B\left(H\right)\right)}+\left\|\langle M_{0}\rangle B\langle M_{0}\rangle^{-1}\langle M_{0}\rangle^{-1}_{z}\right\|_{L^{\infty}\left(\IR^{d},B\left(H\right)\right)}\nonumber\\
		&+\left\|\nabla B\langle M_{0}\rangle^{-1}\langle M_{0}\rangle^{-1}_{z}\right\|_{L^{\infty}\left(\IR^{d},B\left(H\right)^{d}\right)}.
	\end{align}
	For $\beta\in\IR$ denote for $B\in C^{1}_{b}\left(\IR^{d},A_{0},\left|\beta\right|+1\right)$,
	\begin{align}
		\rho_{z}^{\beta}\left(B\right):=\rho_{z}\left(\langle M_{0}\rangle^{\beta}B\langle M_{0}\rangle^{-\beta}\right).
	\end{align}
\end{Definition}

\begin{Remark}\label{norminterpolrem}
	Complex interpolation implies certain operator inequalities (cf \cite{GLST}). In particular, by combining \cite{GLST}[Theorem 2.8] with Jensen's inequality, we have for $\beta\in\left[-2N,2N\right]$,
	\begin{align}
		\rho^{\beta}_{z}\left(B\right)\leq\rho^{-2N}_{z}\left(B\right)+\rho^{2N}_{z}\left(B\right).
	\end{align}
\end{Remark}

\begin{Proposition}\label{hdomainprop}
	Let a family $B\in C^{1}_{b}\left(\IR^{d},A_{0},1\right)$ of symmetric operators satisfy
	\begin{align}\label{hdomainpropeq1}
		\lim_{\mathrm{dist}\left(z,\left(-\infty,0\right]\right)\to\infty}\rho_{z}\left(B\right)=0.
	\end{align}
	Then $D_{B}$, is closed in $L^2\left(\IR^{d},\IC^{r}\otimes H\right)$ with $D_{B}^{\ast}=-\Ii\di+B+M_{0}$, $\dom\ D_{B}^{\ast}=\dom\ D_{0}$. Moreover the operators $D_{B}D_{B}^{\ast}$, $D_{B}^{\ast}D_{B}$, and $H_{B}$ are non-negative, and self-adjoint with domain $\dom\ H_{0}$.
\end{Proposition}

\begin{proof}
	The proof of the statement about $D_{B}$ and $D^{\ast}_{B}$ is similar to the proof of (\cite{GLMST}[Lemma 4.4]), although it is assumed that $d=1$, the argument works with the necessary changes for all dimensions $d\in\IN$. Let us provide at least an outline. The authors show that $B$ is relatively bounded by $D_{0}$, and $B^{\ast}$ by $D_{0}^{\ast}$, with bound $0$, provided that
	\begin{align}
		\left\|B\langle M_{0}\rangle_{z}^{-1}\right\|_{L^{\infty}\left(\IR^{d}, B\left(H\right)\right)}\xrightarrow{\IR\ni z\to\infty}0.
	\end{align}
	This condition is however satisfied by (\ref{hdomainpropeq1}). The statement then follows from the Kato-Rellich Theorems (\cite{Kato}[Theorem 4.1.1, Theorem 5.4.3]). The same idea works for $D^{\ast}_{B}D_{B}$ and $D_{B}D^{\ast}_{B}$, which are self-adjoint by definition. Let us consider $D_{B}D^{\ast}_{B}$, the proof for $D^{\ast}_{B}D_{B}$ is analogous. As differential operators we have $D_{B}D^{\ast}_{B}=\Delta+\left(B+M_{0}\right)^{2}+\Ii\left(\di B\right)$, which implies
	\begin{align}
		D_{B}D^{\ast}_{B}-H_{0}=B^{2}+M_{0}B+BM_{0}+\Ii\left(\di B\right),
	\end{align}
	on $\dom\ H_{0}$. We estimate for $z\geq 1$,
	\begin{align}
		&\left\|\left(B^{2}+M_{0}B+BM_{0}+\Ii\di^{E}B\right)R_{z^{2}}\left(H_{0}\right)\right\|_{B\left(L^2\left(\IR^{d},\IC^{r}\otimes H\right)\right)}\nonumber\\
		\leq&\left(\left\|B\langle M_{0}\rangle^{-1}\right\|_{L^{\infty}\left(\IR^{d}, B\left(H\right)\right)}+1\right)\left\|\langle M_{0}\rangle B\langle M_{0}\rangle^{-1}\langle M_{0}\rangle_{z}^{-1}\right\|_{L^{\infty}\left(\IR^{d},B\left(H\right)\right)}\nonumber\\
		&+\left\|B\langle M_{0}\rangle_{z}^{-1}\right\|_{L^{\infty}\left(\IR^{d}, B\left(H\right)\right)}+\left\|\nabla B\langle M_{0}\rangle^{-1}\langle M_{0}\rangle_{z}^{-1}\right\|_{L^{\infty}\left(\IR^{d}, B\left(H\right)^{d}\right)}\xrightarrow[(\ref{hdomainpropeq1})]{z\to\infty}0.
	\end{align}
	Hence $\left(B+M_{0}\right)^{2}+\Ii\left(\di B\right)$ is relatively bounded by $H_{0}$ with bound $0$. Thus $D_{B}D^{\ast}_{B}$ is self-adjoint on $\dom\ H_{0}$, and essentially self-adjoint on any core of $H_{0}$. Since $C_{c}^{\infty}\left(\IR^{d}\right)\otimes\IC^{r}\otimes\mathcal{D}$ is a common core of both $D_{B}D^{\ast}_{B}$, and $H_{0}$, we have $\dom\left(D_{B}D^{\ast}_{B}\right)=\dom\ H_{0}$.
	
	For $H_{B}$ we have a similar result, but we have to add that $H_{B}\geq 0$ on $\dom\ H_{0}$ follows by a direct calculation.
\end{proof}

\begin{Remark}\label{domarem}
	Under the prerequisites of Proposition \ref{hdomainprop}, it also follows that $A\left(x\right)=A_{0}+B\left(x\right)$ is self-adjoint on $\dom\ A_{0}$ for a.e. $x\in\IR^{d}$. Indeed, since $\langle M_{0}\rangle_{z}\langle D_{0}\rangle_{z}^{-1}$ is bounded in norm by $1$, uniformly for $z>0$, it follows from $\left\|T\right\|_{B\left(L^2\left(\IR^{d},H\right)\right)}=\left\|T\right\|_{L^{\infty}\left(\IR^{d},B\left(H\right)\right)}$ for measurable multiplication operators $T$, that $B\left(x\right)$ is relatively bounded by $A_{0}$ with bound $0$ for a.e. $x\in\IR^{d}$. Since $B\left(x\right)$ is symmetric, the statement follows from the Kato-Rellich Theorem.
\end{Remark}

Next, we recall some well-known facts about Schatten-von Neumann operators.

\begin{Lemma}\label{simoninterpollem}
	Let $2\leq p<\infty$, $x\mapsto T\left(x\right)\in L^{p}\left(\IR^{d},S^{p}\left(\IC^{r}\otimes H\right)\right)$, and $g\in L^{p}\left(\IR^{d}\right)$, then $Tg\left(\left|\di\right|\right)\in S^p\left(L^2\left(\IR^{d},\IC^{r}\otimes H\right)\right)$ with
	\begin{align}\label{simoninterpollemeq1}
		\left\|Tg\left(\left|\di\right|\right)\right\|_{S^p\left(L^2\left(\IR^{d},\IC^{r}\otimes H\right)\right)}\leq\left(2\pi\right)^{-\frac{d}{p}}\left\|T\right\|_{L^{p}\left(\IR^{d},S^{p}\left(\IC^{r}\otimes H\right)\right)}\left\|g\right\|_{L^{p}\left(\IR^{d}\right)}.
	\end{align}
\end{Lemma}

\begin{proof}
	The proof of this statement is analogous to the proof of Theorem 4.1 in \cite{Simon}. There, $H=\IC$, which one has to change here. We give the essential outline. For $p=2$ the statement follows from the fact that there is a $L^2\left(\IR^{d}\times\IR^{d},S^2\left(\IC^{r}\otimes H\right)\right)$-valued integral kernel $k$ of $Tg\left(\left|\di\right|\right)$, given by
	\begin{align}
		k\left(x,y\right)=\left(2\pi\right)^{-\frac{d}{2}}T\left(x\right)\check{g}\left(x-y\right),\ x,y\in\IR^{d},
	\end{align}
	where $\check{g}\left(u\right):=\left(2\pi\right)^{-\frac{d}{2}}\int_{\IR^{d}}e^{-\Ii u\cdot\xi}g\left(\xi\right)\otimes\one_{\IC^{r}}\Id\xi$. The operator $Tg\left(\left|\di\right|\right)$ is thus Hilbert Schmidt with equality in (\ref{simoninterpollemeq1}). On the other hand, $x\mapsto T\left(x\right)\in L^{\infty}\left(\IR^{d},B\left(\IC^{r}\otimes H\right)\right)$ implies that $T\in B\left(L^2\left(\IR^{d},\IC^{r}\otimes H\right)\right)$, and Borel functional calculus implies $g\left(\left|\di\right|\right)\in B\left(L^2\left(\IR^{d},\IC^{r}\otimes H\right)\right)$. We have
	\begin{align}
		\left\|Tg\left(\left|\di\right|\right)\right\|_{B\left(L^2\left(\IR^{d},\IC^{r}\otimes H\right)\right)}\leq\left\|T\right\|_{L^{\infty}\left(\IR^{d},B\left(\IC^{r}\otimes H\right)\right)}\left\|g\right\|_{L^{\infty}\left(\IR^{d}\right)}.
	\end{align}
	If $T$, and $g$ are compactly supported the case $p\geq 2$ with (\ref{simoninterpollemeq1}) follows from complex interpolation. For general $T$, and $g$ let $\left(\phi_{n}\right)_{n\in\IN}$ be a sequence of compactly supported continuous functions, which converge pointwise to $1$ on $\IR$. Then the self-adjoint operators $\phi_{n}\otimes\one_{\IC^{r}\otimes H}$, and $\phi_{n}\left(\left|\di\right|\right)$ converge strongly to $\one_{L^2\left(E\otimes H\right)}$. With $T_{n}\left(x\right):=\phi_{n}\left(x\right)T\left(x\right)$ and $g_{n}:=\phi_{n}g$, we obtain $T_{n}g_{n}\left(\left|\di\right|\right)=\phi_{n}\otimes\one_{\IC^{r}\otimes H}Tg\left(\left|\di\right|\right)\phi_{n}\left(\left|\di\right|\right)$. By Lemma \ref{traceconvlem} and since $\lim_{n\to\infty}T_{n}=T$ in $L^{p}\left(\IR^{d},S^{p}\left(\IC^{r}\otimes H\right)\right)$, and $\lim_{n\to\infty}g_{n}=g$ in $L^{p}\left(\IR^{d}\right)$, the general case follows, by taking the limits past the norms in (\ref{simoninterpollemeq1}).
\end{proof}

\begin{Corollary}\label{simoninterpolcor}
	Let $x\mapsto T\left(x\right)\in L^{p}\left(\IR^{d},S^{d}\left(\IC^{r}\otimes H\right)\right)$, such that $T\left(x\right)$ is self-adjoint for a.e. $x\in\IR^{d}$. Then for $t>\frac{d}{p}$ we have a constant $C_{d,p}<\infty$, such that,
	\begin{align}
		\langle\di\rangle^{s-t}T\langle\di\rangle^{-s}\in& S^{p}\left(L^2\left(\IR^{d},\IC^{r}\otimes H\right)\right),\nonumber\\
		\left\|\langle\di\rangle^{s-t}T\langle\di\rangle^{-s}\right\|_{S^{p}\left(L^2\left(\IR^{d},\IC^{r}\otimes H\right)\right)}\leq& C_{d,p}\left\|T\right\|_{L^{p}\left(\IR^{d},S^{p}\left(\IC^{r}\otimes H\right)\right)},\ s\in\left[0,t\right].
	\end{align}
\end{Corollary}

\begin{proof}
	First consider $s=t$. We note that $\xi\mapsto \langle\xi\rangle^{-t}\in L^{p}\left(\IR^{d}\right)$, so Lemma \ref{simoninterpollem} applies. The case $s=0$ follows by taking adjoints. The general case follows by complex interpolation (cf. \cite{GLST}).
\end{proof}

\begin{Definition}
	We introduce a second set of semi-norms for $B\in C^{1}_{b}\left(\IR^{d},A_{0},m\right)$, $m\geq 0$, and the following data: $\alpha\geq 0$, $\beta\in\IR$, with $\left|\beta-1\right|\leq m$, $\left|\beta-\alpha\right|\leq m$, $p\geq 2$, $s\geq 0$, and a unit vector field $\nu\in\Gamma\left(T\IR^{d}\right)$,
	\begin{align}
		\tau^{\alpha,\beta,\nu,p,s}\left(B\right):=\left\|x\mapsto\langle x\rangle^{s}\langle A_{0}\rangle^{-\alpha+\beta}\left(\nu B\right)\left(x\right)\langle A_{0}\rangle^{-\beta}\right\|_{L^{\infty}\left(\IR^{d},S^{p}\left(H\right)\right)},
	\end{align}
	and denote $\tau^{\alpha,\beta,\nabla,p,s}:=\sum_{i=1}^{d}\tau^{\alpha,\beta,\p_{i},p,s}$.
\end{Definition}

We introduce the set of conditions on the family operator family $A=M_{0}+B$.

\begin{Hypothesis}\label{hyp1}
	Let $\alpha\geq 1$, and $N\geq\lfloor\frac{\alpha-1}{2}\left(d+1\right)\rfloor+1$. Let $B$ be a family of symmetric operators in $H$, with domains $\dom\ A_{0}$. Assume that $B\in C^{1}_{b}\left(\IR^{d},A_{0},2N+1\right)$. Assume that
	\begin{enumerate}
		\item \begin{align}\label{hypdomeq}
			\lim_{\mathrm{dist}\left(z,\left(-\infty,0\right]\right)\to\infty}\rho_{z}^{-2N}\left(B\right)+\rho_{z}^{2N}\left(B\right)=0,
		\end{align}
		and that for all $\beta\in\left[-2N+\alpha,2N+1\right]$,
		\begin{align}\label{hyptraceeq}
			\tau^{\alpha,\beta,\nabla,d,1}\left(B\right)&<\infty,\nonumber\\
			\exists\epsilon>0:\ \tau^{\alpha,\beta,\p_{R},d,1+\epsilon}\left(B\right)&<\infty,
		\end{align}
		where $\p_{R}:=\sum_{i=1}^{d}\frac{x^{i}}{\left|x\right|}\p_{i}$ denotes the radial vector field.
		\item Let $S_{1}\left(0\right)$ denote the $\left(d-1\right)$-dimensional unit sphere. For $\phi\in\dom\ A_{0}$, a.e. $y\in S_{1}\left(0\right)$, a.e. $x\in\IR^{d}$, and $\gamma\in\IN_{0}^{d}$ with $\left|\gamma\right|\leq 1$, assume,
		\begin{align}\label{radlimeq}
			\lim_{R\to\infty}R^{\left|\gamma\right|}\left\|\left(\p^{\gamma}B\left(Ry+x\right)-\p^{\gamma}B\left(Ry\right)\right)\phi\right\|_{H}&=0. 
		\end{align}
	\end{enumerate}
\end{Hypothesis}

\begin{Remark}\label{a0rem}
	By a Kato-Rellich argument Hypothesis \ref{hyp1} implies that for a.e. $x_{0}\in\IR^{d}$, $A\left(x_{0}\right)=A_{0}+B\left(x_{0}\right)$ can take the role of the model operator $A_{0}$.
\end{Remark}

The second set of conditions concerned with the radial convergence at infinity of $A$ and $\Id A$ will also only play a role in determining the trace. Note that the prerequisites of Proposition \ref{hdomainprop} and Remark \ref{domarem} are satisfied if $B$ satisfies Hypothesis \ref{hyp1}.

In the following we will prepare for the proof of Proposition \ref{decomposelem}. In this proof we will factorize the partial trace of difference of the resolvent powers of $DD^{\ast}$ and $D^{\ast}D$ into appropriate bounded and Schatten-von Neumann class operators. These factors will fit the templates of operators we will discuss below.

\begin{Definition}
    Let $B$ satisfy Hypothesis \ref{hyp1}. For $z\in\IC$ let $w=w\left(z\right):=\mathrm{dist}\left(z,\left(-\infty,0\right]\right)$. For $C\geq 0$ let $\mathcal{D}_{C}\left(B\right):=\left\{z\in\IC:w\left(z\right)\geq c\right\}$ where $c\geq 1$ is the minimal choice, such that
	\begin{align}
		Cw^{-\frac{1}{2}}+\left(\rho^{-2N}_{1}\left(B\right)+\rho^{2N}_{1}\left(B\right)+1\right)\left(\rho^{-2N}_{z}\left(B\right)+\rho^{2N}_{z}\left(B\right)\right)\leq\frac{1}{2},\ z\in\mathcal{D}_{C}\left(B\right).
    \end{align}
\end{Definition}

\begin{Definition}
    For $s\in\IR$ denote by $Sym_{s}$ those smooth functions $f:\IR^{d}\rightarrow\IC$, such that for all $\gamma\in\IN_{0}^{d}$,
    \begin{align}
       \langle X\rangle^{\left|\gamma\right|-s}\p^{\gamma}f\in L^{\infty}\left(\IR^{d}\right).
    \end{align}
\end{Definition}

\begin{Remark}\label{symbolrem}
    Straightforward calculation verifies $C^{\infty}_{c}\left(\IR^{d}\right)\subset Sym_{-\infty}:=\bigcap_{j=0}^{\infty}Sym_{-j}$, and $\langle X\rangle^{u}\in Sym_{u}$ for $u\in\IR$. For $f\in Sym_{r}$, $g\in Sym_{s}$ we also have $fg\in Sym_{r+s}$, which follows from the Leibniz rule. Moreover for $f\in Sym_{0}$ the Leibniz rule, taking adjoints and complex interpolation implies that $\langle\di\rangle^{s}f\langle\di\rangle^{-s}$ have bounded extensions in $L^2\left(\IR^{d},\IC^{r}\otimes H\right)$ for all $s\in\IR$.
\end{Remark}

We discuss our first type of bounded operators involved in the factorization of Proposition \ref{decomposelem}.

\begin{Lemma}\label{boundedlem}
	There exists a universal constant $C$ such that if $B$ satisfies Hypothesis \ref{hyp1}, and we let $t\in\left[0,2\right]$, $s\in\left[-1,1\right]$, $k\in\IN$, $k\leq N$, and $r\in\left[-2N,2N\right]$ with $r-2\left(k-1\right)\in\left[-2N,2N\right]$, then the operators
	\begin{align}
		R:=\langle M_{0}\rangle^{r}\langle X\rangle^{-s}R_{z}\left(T\right)^{k}\langle X\rangle^{s}\langle M_{0}\rangle^{2k-r-t}\langle\di\rangle^{t},\ T\in\left\{H_{B},DD^{\ast},D^{\ast}D\right\},
	\end{align}
	possess extensions in $B\left(L^2\left(\IR^{d},\IC^{r}\otimes H\right)\right)$ with operator norm bounded by $2^{k}$, for $z\in\mathcal{D}_{C}\left(B\right)$.
\end{Lemma}

\begin{proof}
	We write $R=\left(\prod_{i=1}^{k-1}R_{i}\right)S$, with
	\begin{align}
		S=\langle M_{0}\rangle^{-2\left(k-1\right)+r}\langle X\rangle^{-s}R_{z}\left(T\right)\langle X\rangle^{s}\langle M_{0}\rangle^{2k-r-t}\langle\di\rangle^{t}\nonumber\\
		R_{i}=\langle M_{0}\rangle^{-2\left(i-1\right)+r}\langle X\rangle^{-s}R_{z}\left(T\right)\langle X\rangle^{s}\langle M_{0}\rangle^{2i-r},\ i\in\left\{1,\ldots,k-1\right\}.
	\end{align}
    We develop $R_{z}\left(T\right)$ into a Neumann series around $H_{0}$,
    \begin{align}\label{boundedlemeq1}
        R_{i}&=\sum_{l=0}^{\infty}\left(-1\right)^{l}\left(R_{z}\left(H_{0}\right)C_{i}\right)^{l}\langle M_{0}\rangle^{2}R_{z}\left(H_{0}\right),\nonumber\\
        S&=\sum_{l=0}^{\infty}\left(-1\right)^{l}\left(R_{z}\left(H_{0}\right)C_{k}\right)^{l}\langle M_{0}\rangle^{2-t}\langle\di\rangle^{t}R_{z}\left(H_{0}\right),
    \end{align}
    where
    \begin{align}
           C_{i}:=\langle M_{0}\rangle^{-2\left(i-1\right)+r}\left(B^{2}+BM_{0}+M_{0}B+\epsilon\Ii\di^{E}B+\langle X\rangle^{-s}\left[\Delta,\langle X\rangle^{s}\right]\right)\langle M_{0}\rangle^{2\left(i-1\right)-r},
    \end{align}
    where $\epsilon=-1$ for $T=D^{\ast}D$, $\epsilon=0$ for $T=H_{B}$, and $\epsilon=1$ for $T=DD^{\ast}$. Since $\langle X\rangle^{s}\in Sym_{s}$ it follows that $\langle X\rangle^{-s}\left[\Delta,\langle X\rangle^{s}\right]$ is a first order differential operator with coefficients in $Sym_{-1}$. Thus there exists a universal constant $C$ such that for $u\in\left[-1,1\right]$, $z\in\IC\backslash\left(-\infty,0\right]$ with $w:=\mathrm{dist}\left(z,\left(-\infty,0\right]\right)\geq 1$,
    \begin{align}
        \left\|\overline{R_{z}\left(H_{0}\right)\langle X\rangle^{-s}\left[\Delta,\langle X\rangle^{s}\right]}\right\|_{B\left(L^2\left(\IR^{d},\IC^{r}\otimes H\right)\right)}\leq Cw^{-\frac{1}{2}}.
    \end{align}
    Furthermore we have by Hypothesis \ref{hyp1} the norm estimate
    \begin{align}
        &\left\|\overline{R_{z}\left(H_{0}\right)C_{i}}\right\|_{B\left(L^2\left(\IR^{d},\IC^{r}\otimes H\right)\right)}\leq Cw^{-\frac{1}{2}}+\left\|\langle M_{0}\rangle \langle M_{0}\rangle_{z}R_{z}\left(H_{0}\right)\right\|_{B\left(L^2\left(\IR^{d},\IC^{r}\otimes H\right)\right)}\nonumber\\
        &\cdot\left\|\overline{\langle M_{0}\rangle_{z}^{-1}\langle M_{0}\rangle^{-2\left(i-1\right)+r-1}\left(B^2+BM_{0}+M_{0}B+\epsilon\Ii\left(\di B\right)\right)\langle M_{0}\rangle^{2\left(i-1\right)-r}}\right\|_{B\left(L^2\left(\IR^{d},\IC^{r}\otimes H\right)\right)}\nonumber\\
        \leq& Cw^{-\frac{1}{2}}+\left(\rho_{1}^{2\left(i-1\right)-r}\left(B\right)+1\right)\rho_{z}^{2\left(i-1\right)-r}\left(B\right)\leq\frac{1}{2},\ z\in\mathcal{D}_{C}\left(B\right),
    \end{align}
    where the last inequality follows by Remark \ref{norminterpolrem} and the definition of the domain $\mathcal{D}_{C}\left(B\right)$. This implies that the Neumann series of $R_{i}$ and $S$ converge in operator norm, each of which can be bounded in operator norm by $2$ for $z\in\mathcal{D}_{C}\left(B\right)$. We conclude that $\left\|\overline{R}\right\|_{B\left(L^2\left(\IR^{d},\IC^{r}\otimes H\right)\right)}\leq 2^{k}$ for $z\in\mathcal{D}_{C}\left(B\right)$.
\end{proof}

A similar result holds under weaker conditions, which we will need later in the proof of the approximation statement.

\begin{Lemma}\label{lightboundedlem}
    Let $B:\IR^{d}\rightarrow B\left(H\right)$ be a point-wise self-adjoint $C^{1}$-function with respect to the strong operator topology such that $\left\|B\right\|_{C^{1}_{b}\left(\IR^{d},B\left(H\right)\right)}=\left\|B\right\|_{L^{\infty}\left(\IR^{d},B\left(H\right)\right)}+\left\|\nabla B\right\|_{L^{\infty}\left(\IR^{d},B\left(H\right)^{d}\right)}<\infty$. Assume also that $B\in C^{1}_{b}\left(\IR^{d},A_{0},1\right)$ with $\lim_{w\to\infty}\rho_{z}\left(B\right)=0$ where $w:=\operatorname{dist}\left(z,\left(-\infty,0\right]\right)$. Let $t\in\left[0,2\right]$, $s\in\left[-1,1\right]$, and $k\in\IN$. Then the operators
    \begin{align}
        R_{k,s,t}:=\langle X\rangle ^{-s}R_{z}\left(H_{B}\right)^{k}\langle X\rangle^{s}\langle\di\rangle^{t}
    \end{align}
    admit extensions in $B\left(L^2\left(\IR^{d},\IC^{r}\otimes H\right)\right)$ with norm bounded by $C^k\left(\left\|B\right\|_{C^{1}_{b}\left(\IR^{d},B\left(H\right)\right)}+1\right)^{2k}$ for $w\geq C\left(\left\|B\right\|_{C^{1}_{b}\left(\IR^{d},B\left(H\right)\right)}+1\right)^{4}$, where $C$ is some universal constant.\\    
    For $T_{B}\in\left\{H_{B},D_{B}^{\ast}D_{B},D_{B}D_{B}^{\ast}\right\}$, and $w\geq C\left(\left\|B\right\|_{C^{1}_{b}\left(\IR^{d},B\left(H\right)\right)}+1\right)^{4}$, the operators
    \begin{align}
        \left(H_{0}+1\right)^{\frac{1}{2}}\langle X\rangle ^{-s}R_{z}\left(T_B\right)\langle X\rangle ^{s}\left(H_{0}+1\right)^{\frac{1}{2}},\ s\in\left[-1,1\right],
    \end{align}
    admit extensions in $B\left(L^2\left(\IR^{d},\IC^{r}\otimes H\right)\right)$ with norm bounded by $C^k\left(\left\|B\right\|_{C^{1}_{b}\left(\IR^{d},B\left(H\right)\right)}+1\right)^{2k}$ for $w\geq C\left(\left\|B\right\|_{C^{1}_{b}\left(\IR^{d},B\left(H\right)\right)}+1\right)^{4}$.
\end{Lemma}

\begin{proof}
    We note that by Proposition \ref{hdomainprop} the operators $T_{B}\geq 0$ are self-adjoint, so that the involved resolvents are well-defined. We show first that
    \begin{align}
        \left(H_{0}+1\right)^{\frac{1}{2}}\langle X\rangle^{-s}R_{z}\left(T_{B}\right)\langle X\rangle^{s}\left(H_{0}+1\right)^{\frac{1}{2}}
    \end{align}
    admits a bounded extension. We develop $\langle X\rangle^{-s}R_{z}\left(T_{B}\right)\langle X\rangle^{s}$ into a Neumann series around $H_{0}$
    \begin{align}\label{lightboundedlemeq1}
        &\overline{\left(H_{0}+1\right)^{\frac{1}{2}}\langle X\rangle^{-s}R_{z}\left(T_{B}\right)\langle X\rangle^{s}\left(H_{0}+1\right)^{\frac{1}{2}}}\nonumber\\
        =&\left(H_{0}+1\right)^{\frac{1}{2}}R_{z}\left(H_{0}\right)^{\frac{1}{2}}\sum_{l=0}^{\infty}\left(-1\right)^{l}\left(R_{z}\left(H_{0}\right)^{\frac{1}{2}}T_{\epsilon}R_{z}\left(H_{0}\right)^{\frac{1}{2}}\right)^{l}\left(H_{0}+1\right)^{\frac{1}{2}}R_{z}\left(H_{0}\right)^{\frac{1}{2}},
    \end{align}
    where
    \begin{align}
        T_{\epsilon}:=B^{2}+BM_{0}+M_{0}B+f_{s}\di+g_{s}+\Ii\epsilon\left(\di B\right),\ f_{s}:=2\langle X\rangle^{-s}\di\langle X\rangle^{s},\ g_{s}:=\langle X\rangle^{-s}\Delta\langle X\rangle^{s},
    \end{align}
        where $\epsilon=-1$ for $T_{B}=D_{B}^{\ast}D_{B}$, $\epsilon=0$ for $T_{B}=H_{B}$, and $\epsilon=1$ for $T_{B}=D_{B}D_{B}^{\ast}$. Since $\left(H_{0}+1\right)^{\frac{1}{2}}R_{z}\left(H_{0}\right)^{\frac{1}{2}}$ is bounded in operator norm by $1$ for $w\geq 1$, it remains to investigate the inner factor of the left-hand side of (\ref{lightboundedlemeq1}), and thus we consider
    \begin{align}
        &R_{z}\left(H_{0}\right)^{\frac{1}{2}}T_{\epsilon}R_{z}\left(H_{0}\right)^{\frac{1}{2}}=R_{z}\left(H_{0}\right)^{\frac{1}{2}}B^{2}R_{z}\left(H_{0}\right)^{\frac{1}{2}}+R_{z}\left(H_{0}\right)^{\frac{1}{2}}BM_{0}R_{z}\left(H_{0}\right)^{\frac{1}{2}}\nonumber\\
        &+M_{0}R_{z}\left(H_{0}\right)^{\frac{1}{2}}BR_{z}\left(H_{0}\right)^{\frac{1}{2}}+R_{z}\left(H_{0}\right)^{\frac{1}{2}}f_{s}\di R_{z}\left(H_{0}\right)^{\frac{1}{2}}+R_{z}\left(H_{0}\right)^{\frac{1}{2}}g_{s} R_{z}\left(H_{0}\right)^{\frac{1}{2}}\nonumber\\
        &+\Ii\epsilon R_{z}\left(H_{0}\right)^{\frac{1}{2}}\left(\di B\right)R_{z}\left(H_{0}\right)^{\frac{1}{2}}.
    \end{align}
    Since $M_{0}R_{z}\left(H_{0}\right)^{\frac{1}{2}}$ and $\di R_{z}\left(H_{0}\right)^{\frac{1}{2}}$ are uniformly bounded in operator norm by $1$ for $w\geq 1$, $B$, $B^2$, $\left(\di B\right)$ are bounded in norm operator norm by $\left(\left\|B\right\|_{C^{1}_{b}\left(\IR^{d},B\left(H\right)\right)}+1\right)^{2}$, and the functions $f_{s}\in Sym_{-1}$, $g_{s}\in Sym_{-2}$ give rise to uniformly bounded operators in operator norm for $s\in\left[-1,1\right]$, the operator $R_{z}\left(H_{0}\right)^{\frac{1}{2}}TR_{z}\left(H_{0}\right)^{\frac{1}{2}}$ is uniformly bounded in operator norm by $c\left(\left\|B\right\|_{C^{1}_{b}\left(\IR^{d},B\left(H\right)\right)}+1\right)^{2}w^{-\frac{1}{2}}$ for some universal constant $c$ and $w\geq 1$.
    We assume for the rest of the proof that $w\geq c^{2}\left(\left\|B\right\|_{C^{1}_{b}\left(\IR^{d},B\left(H\right)\right)}+1\right)^{4}$. Then the Neumann series (\ref{lightboundedlemeq1}) converges in operator norm to a bounded operator with norm bound $2$. Since $\left(\Ii+\di\right)\left(H_{0}+1\right)^{-\frac{1}{2}}$ is bounded in operator norm by $1$, we also have that
    \begin{align}\label{lightboundedlemeq2}
        \overline{\left(\Ii+\di\right)\langle X\rangle^{-s}R_{z}\left(H_{B}\right)\langle X\rangle^{s}\left(\Ii+\di\right)}
    \end{align}
    is a bounded operator with norm less than $2$. In (\ref{lightboundedlemeq2}) we commute $\Ii+\di$ and obtain
    \begin{align}\label{lightboundedlemeq4}
        &\overline{\left(\Ii+\di\right)\langle X\rangle^{-s}R_{z}\left(H_{B}\right)\langle X\rangle^{s}\left(\Ii+\di\right)}\nonumber\\
        =&\overline{\langle X\rangle^{-s}R_{z}\left(H_{B}\right)\langle X\rangle^{s}\left(\Ii+\di\right)^{2}}+\overline{\left[\di,\langle X\rangle^{-s}R_{z}\left(H_{B}\right)\langle X\rangle^{s}\right]\left(\Ii+\di\right)}
    \end{align}
    By the resolvent identity we rewrite the commutator and obtain
    \begin{align}\label{lightboundedlemeq3}
        &\overline{\left[\di,\langle X\rangle^{-s}R_{z}\left(H_{B}\right)\langle X\rangle^{s}\right]\left(\Ii+\di\right)}=-\overline{\langle X\rangle^{-s}R_{z}\left(H_{B}\right)\langle X\rangle^{s}\left[\di,T\right]\langle X\rangle^{-s}R_{z}\left(H_{B}\right)\langle X\rangle^{s}\left(\Ii+\di\right)}\nonumber\\
        =&-\overline{\langle X\rangle^{-s}R_{z}\left(H_{B}\right)\langle X\rangle^{s}\left(H_{0}+1\right)^{\frac{1}{2}}}\nonumber\\ &\overline{\left(H_{0}+1\right)^{-\frac{1}{2}}\left(\left(\di B^{2}\right)+M_{0}\left(\di B\right)+\left(\di B\right)M_{0}+\left(\di f_{s}\right)\di+\left(\di g_{s}\right)\right)\left(H_{0}+1\right)^{-\frac{1}{2}}}\nonumber\\
        &\overline{\left(H_{0}+1\right)^{\frac{1}{2}}\langle X\rangle^{-s}R_{z}\left(H_{B}\right)\langle X\rangle^{s}\left(\Ii+\di\right)}.
    \end{align}
    The first and third factor of (\ref{lightboundedlemeq3}) are both bounded in operator norm by $2$ according to the previous considerations, the second factor is bounded since $M_{0}R_{z}\left(H_{0}\right)^{\frac{1}{2}}$ and $\di R_{z}\left(H_{0}\right)^{\frac{1}{2}}$ are uniformly bounded in operator norm by $1$ for $w\geq 1$, $\left(\di B\right)$, and $\left(\di B^{2}\right)$ are bounded by $\left(\left\|B\right\|_{C^{1}_{b}\left(\IR^{d},B\left(H\right)\right)}+1\right)^{2}$, and the functions $\left(\di f_{s}\right)\in Sym_{-2}$, $\left(\di g_{s}\right)\in Sym_{-3}$ give rise to uniformly bounded operators in operator norm for $s\in\left[-1,1\right]$. In total (\ref{lightboundedlemeq3}) implies that $\overline{\left[\di,\langle X\rangle^{-s}R_{z}\left(H_{B}\right)\langle X\rangle^{s}\right]\left(\Ii+\di\right)}$ is uniformly bounded in operator norm by $C'\left(\left\|B\right\|_{C^{1}_{b}\left(\IR^{d},B\left(H\right)\right)}+1\right)^{2}$ for some universal constant $C'$. By (\ref{lightboundedlemeq4}) it follows that
    \begin{align}
        \overline{\langle X\rangle^{-s}R_{z}\left(H_{B}\right)\langle X\rangle^{s}\left(\Ii+\di\right)^{2}},
    \end{align}
    and thus $\overline{R_{1,s,t}}$ is uniformly bounded by $C''\left(\left\|B\right\|_{C^{1}_{b}\left(\IR^{d},B\left(H\right)\right)}+1\right)^{2}$ for some universal constant $C''$, and all $t\in\left[0,2\right]$. Finally, we use that $R_{k,s,t}=R_{1,s,0}^{k-1}R_{1,s,t}$ to conclude the statement.
\end{proof}

We will also need the following convergence statement, which is a fundamental property of Schatten-von Neumann operators.

\begin{Lemma}\label{traceconvlem}\cite{GLMST}[Lemma 3.4]
	Let $R_{n},S_{n},T_{n}$, $n\in\IN$ be bounded operators in a separable Hilbert space $X$. For $1\leq p<\infty$, let $S_{n}\in S^{p}\left(X\right)$, such that $\lim_{n\to\infty}S_{n}=S$ in $S^{p}\left(X\right)$-norm. Furthermore assume that $\lim_{n\to\infty}R_{n}=R$, and $\lim_{n\to\infty}T_{n}^{\ast}=T^{\ast}$ strongly. Then
	\begin{align}
		\lim_{n\to\infty}R_{n}S_{n}T_{n}=RST,
	\end{align}
	in $S^{p}\left(X\right)$-norm.
\end{Lemma}

Inspired by this Lemma and for our convenience, we introduce the following notation.

\begin{Definition}
	If $T$ is a bounded operator, and $\left(T_{n}\right)_{n\in\IN}$ is a sequence of bounded operators in a Hilbert space $X$, we say $T_{n}$ converges in SOTA of $X$ to $T$, if $T_{n}$, and $T_{n}^{\ast}$ converge in the strong operator topology to $T$, respectively to $T^{\ast}$ as $n\to\infty$.
\end{Definition}

\begin{Lemma}\label{sotlem}
    There exists a universal constant $C$ such that the following statement holds true. Let $B$ satisfy Hypothesis \ref{hyp1} and let $\left(B_{n}\right)_{n\in\IN}$ be a sequence of operators, each of which satisfy Hypothesis \ref{hyp1}, such that $\left(B_{n}\right)_{n\in\IN}$ converges in the following sense: For $\beta\in\left[-2N+1,2N-1\right]$, and a.e. $x\in\IR^{d}$, we have strong operator convergence in $H$ of
	\begin{align}\label{sotlemeq1}
		\langle A_{0}\rangle^{\beta}\left(B_{n}\left(x\right)+A_{0}\right)^{2}\langle A_{0}\rangle^{-2-\beta}\xrightarrow{n\to\infty}\langle A_{0}\rangle^{\beta}\left(B\left(x\right)+A_{0}\right)^{2}\langle A_{0}\rangle^{-2-\beta},
	\end{align}
	and with $w:=\mathrm{dist}\left(z,\left(-\infty,0\right]\right)$, assume
	\begin{align}\label{sotlemeq2}
		\sup_{n\in\IN}\left(\rho_{z}^{-2N}\left(B_{n}\right)+\rho_{z}^{2N}\left(B_{n}\right)\right)\xrightarrow{w\to\infty}0,
	\end{align}
	then we have SOTA convergence in $L^2\left(\IR^{d},\IC^{r}\otimes H\right)$ of 
	\begin{align}
		\overline{\langle M_{0}\rangle^{r}\langle X\rangle^{-s}R_{z}\left(H_{B_{n}}\right)^{k}\langle X\rangle^{s}\langle M_{0}\rangle^{2k-r-t}\langle\di\rangle^{t}}\xrightarrow{n\to\infty}\overline{\langle M_{0}\rangle^{r}\langle X\rangle^{-s}R_{z}\left(H_{B}\right)^{k}\langle X\rangle^{s}\langle M_{0}\rangle^{2k-r-t}\langle\di\rangle^{t}},
	\end{align}
for $t\in\left[0,2\right]$, $u\in\left[-1,1\right]$, $k\in\IN$, $k\leq N$, and $r\in\left[-2N,2N\right]$ with $r-2\left(k-1\right)\in\left[-2N,2N\right]$.\\
If additionally for a.e. $x\in\IR^{d}$ we have strong operator convergence in $H$ of
	\begin{align}
		\langle A_{0}\rangle^{\beta}\p_{i}B_{n}\left(x\right)\langle A_{0}\rangle^{-2-\beta}\xrightarrow{n\to\infty}\langle A_{0}\rangle^{\beta}\p_{i}B\left(x\right)\langle A_{0}\rangle^{-2-\beta},\ i\in\left\{1,\ldots,d\right\},
	\end{align}
	then we also have SOTA convergence in $L^2\left(\IR^{d},\IC^{r}\otimes H\right)$ of 
	\begin{align}
		&\overline{\langle M_{0}\rangle^{r}\langle X\rangle^{-s}R_{z}\left(D_{B_{n}}D_{B_{n}}^{\ast}\right)^{k}\langle X\rangle^{s}\langle M_{0}\rangle^{2k-r-t}\langle\di\rangle^{t}}\nonumber\\
        \xrightarrow{n\to\infty}&\overline{\langle M_{0}\rangle^{r}\langle X\rangle^{-s}R_{z}\left(D_{B}D_{B}^{\ast}\right)^{k}\langle X\rangle^{s}\langle M_{0}\rangle^{2k-r-t}\langle\di\rangle^{t}},
	\end{align}
	and
	\begin{align}
		&\overline{\langle M_{0}\rangle^{r}\langle X\rangle^{-s}R_{z}\left(D_{B_{n}}^{\ast}D_{B_{n}}\right)^{k}\langle X\rangle^{s}\langle M_{0}\rangle^{2k-r-t}\langle\di\rangle^{t}}\nonumber\\
        \xrightarrow{n\to\infty}&\overline{\langle M_{0}\rangle^{r}\langle X\rangle^{-s}R_{z}\left(D_{B}^{\ast}D_{B}\right)^{k}\langle X\rangle^{s}\langle M_{0}\rangle^{2k-r-t}\langle\di\rangle^{t}},
	\end{align}
for $t\in\left[0,2\right]$, $u\in\left[-1,1\right]$, $k\in\IN$, $k\leq N$, and $r\in\left[-2N,2N\right]$ with $r-2\left(k-1\right)\in\left[-2N,2N\right]$.
\end{Lemma}

\begin{proof}
	We only show the statement involving $H_{B}$. The proof for $D_{B}D_{B}^{\ast}$ and $D_{B}^{\ast}D_{B}$ is analogous. Consider the factorization presented in the proof of Lemma \ref{boundedlem}, now with $B_{n}$ replacing $B$ everywhere. We note that both Neumann series in (\ref{boundedlemeq1}), converge for $z\in\mathcal{D}_{C}\left(B_{n}\right)$. Due to condition (\ref{sotlemeq2}) there exists $w_{0}\geq 1$, such that $\left\{z\in\IC:w\left(z\right)\geq w_{0}\right\}\subseteq\mathcal{D}_{C}\left(B\right)\cap\bigcap_{n\in\IN}\mathcal{D}_{C}\left(B_{n}\right)$. Thus, analogous to the proof of Lemma \ref{boundedlem} the operators
	\begin{align}
		S_{n}\left(t\right):=&\langle M_{0}\rangle^{-2\left(k-1\right)+r}\langle X\rangle^{-s}R_{z}\left(H_{B_{n}}\right)\langle X\rangle^{s}\langle M_{0}\rangle^{2k-r-t}\langle\di\rangle^{t},\nonumber\\
        S\left(t\right):=&\langle M_{0}\rangle^{-2\left(k-1\right)+r}\langle X\rangle^{-s}R_{z}\left(H_{B}\right)\langle X\rangle^{s}\langle M_{0}\rangle^{2k-r-t}\langle\di\rangle^{t},
	\end{align}
	have bounded extensions for $w\left(z\right)\geq w_{0}$, where $w_{0}$ is independent of $n$, and their operator norms are uniformly bounded by $2$. By the resolvent identity we find
	\begin{align}\label{sotlemeq3}
		\overline{S_{n}\left(t\right)}-\overline{S\left(t\right)}=\overline{S_{n}\left(0\right)}\overline{\left(\langle M_{0}\rangle^{-2k+r}\left(\left(B+M_{0}\right)^{2}-\left(B_{n}+M_{0}\right)^{2}\right)\langle M_{0}\rangle^{2\left(k-1\right)-r}\right)}\overline{S\left(t\right)}.
	\end{align}
	Since $\overline{S_{n}\left(0\right)}$ is uniformly bounded by $2$, it suffices to verify the strong operator convergence in $L^2\left(\IR^{d},\IC^{r}\otimes H\right)$ of the inner factor of (\ref{sotlemeq3}), to conclude strong operator convergence of $\overline{S_{n}\left(t\right)}$ to $\overline{S\left(t\right)}$. Indeed, condition (\ref{sotlemeq1}), the uniform bound given by condition (\ref{sotlemeq2}), and the dominated convergence theorem ensure the strong convergence of the inner factor. For the remaining factors $R_{i,n}:=\langle M_{0}\rangle^{-2\left(i-1\right)+r}\langle X\rangle^{-s}R_{z}\left(H_{B_{n}}\right)\langle X\rangle^{s}\langle M_{0}\rangle^{2i-r}$, we obtain analogous results. For the strong convergence of the adjoints the claimed statement follows similarly
    if one replaces (\ref{sotlemeq3}) with
    \begin{align}\label{sotlemeq4}
		S_{n}\left(t\right)^{\ast}-S\left(t\right)^{\ast}=S_{n}\left(t\right)^{\ast}\overline{\left(\langle M_{0}\rangle^{2\left(k-1\right)-r}\left(\left(B+M_{0}\right)^{2}-\left(B_{n}+M_{0}\right)^{2}\right)\langle M_{0}\rangle^{-2k+r}\right)}S\left(0\right)^{\ast}.
	\end{align}
    using that the operator functions $B_{n}$ and $B$ are point-wise symmetric.
\end{proof}

We also provide here a variant under weaker conditions, which we will later also need.

\begin{Lemma}\label{lightsotlem}
    Let $B:\IR^{d}\rightarrow B\left(H\right)$, $B_{n}:\IR^{d}\rightarrow B\left(H\right)$ be point-wise self-adjoint $C^{1}$-functions with respect to the strong operator topology for $n\in\IN$, and assume that $\left\|B\right\|_{C^{1}_{b}\left(\IR^{d},B\left(H\right)\right)}$, and $\left\|B_{n}\right\|_{C^{1}_{b}\left(\IR^{d},B\left(H\right)\right)}$ are uniformly bounded. Assume also that $B, B_{n}\in C^{1}_{b}\left(\IR^{d},A_{0},1\right)$, and that for all $n\in\IN$,
    \begin{align}
        \lim_{w\to\infty}\rho_{z}\left(B\right)=\lim_{w\to\infty}\rho_{z}\left(B_{n}\right)=0,
    \end{align}
    where $w:=\operatorname{dist}\left(z,\left(-\infty,0\right]\right)$. Finally assume that $B_{n}\left(x\right)\xrightarrow{n\to\infty}B\left(x\right)$, and $\nabla B_{n}\left(x\right)\xrightarrow{n\to\infty}\nabla B\left(x\right)$ for a.e. $x\in\IR^{d}$ in the strong operator topology in $H$. Then there exist constants $c$ and $C$ such that for $w:=\operatorname{dist}\left(z,\left(-\infty,0\right]\right)\geq c$ we have convergence in strong operator topology of $L^2\left(\IR^{d},\IC^{r}\otimes H\right)$ of
    \begin{align}\label{lightsotlemeq0}
        \overline{\langle X\rangle^{-s}R_{z}\left(H_{B_{n}}\right)^{k}\langle X\rangle^{s}\langle\di\rangle^{t}}\xrightarrow{n\to\infty}\overline{\langle X\rangle^{-s}R_{z}\left(H_{B}\right)^{k}\langle X\rangle^{s}\langle\di\rangle^{t}}
    \end{align}
    for $t\in\left[0,2\right]$, $s\in\left[-1,1\right]$, and $k\in\IN$, and the operators are uniformly bounded in operator norm by $C$. In case $t=0$, (\ref{lightsotlemeq0}) converges in SOTA of $L^2\left(\IR^{d},\IC^{r}\otimes H\right)$. Moreover for $w\geq 1$ we have SOTA convergence in $L^2\left(\IR^{d},\IC^{r}\otimes H\right)$ of
    \begin{align}\label{lightsotlemeq2}
        R_{z}\left(D_{B_{n}}^{\ast}D_{B_{n}}\right)^{k}\xrightarrow{n\to\infty}R_{z}\left(D_{B}^{\ast}D_{B}\right)^{k},\ R_{z}\left(D_{B_{n}}D_{B_{n}}^{\ast}\right)^{k}\xrightarrow{n\to\infty}R_{z}\left(D_{B}D_{B}^{\ast}\right)^{k}.
    \end{align}
\end{Lemma}

\begin{proof}
    By the uniform boundedness principle and Lemma \ref{lightboundedlem} there exist constants $c,C$ such that for $w\geq c$ we have the uniform bound
    \begin{align}
    \left\|\overline{\langle X\rangle^{-s}R_{z}\left(H_{B_{n}}\right)^{k}\langle X\rangle^{s}\langle\di\rangle^{t}}\right\|_{B\left(L^2\left(\IR^{d},\IC^{r}\otimes H\right)\right)},\ \left\|\overline{\langle X\rangle^{-s}R_{z}\left(H_{B}\right)^{k}\langle X\rangle^{s}\langle\di\rangle^{t}}\right\|_{B\left(L^2\left(\IR^{d},\IC^{r}\otimes H\right)\right)}\leq C.
    \end{align}
    It suffices to restrict to the case $k=1$ by factorizing the operators. Due to the uniform norm bounds and the density of $C_{c}^{\infty}\left(\IR^{d}\right)$ in $L^2\left(\IR^{d}\right)$ we may also assume $t=0$. Then with
    \begin{align}
        S_{n}&:=\overline{\langle X\rangle^{-s}R_{z}\left(H_{B_{n}}\right)\langle X\rangle^{s}},\ S:=\overline{\langle X\rangle^{-s}R_{z}\left(H_{B}\right)\langle X\rangle^{s}},
    \end{align}
    the resolvent identity implies
    \begin{align}
        S_{n}-S=S_{n}\left(\left(B+M_{0}\right)^{2}-\left(B_{n}+M_{0}\right)^{2}\right)S
    \end{align}
    Since the operators $\overline{S_{n}M_{0}}$, $M_{0}S$ are uniformly bounded by Lemma \ref{lightboundedlem}, it follows that $S_{n}$ converges in the strong operator topology to $S$.
    
    For the convergence of the adjoints of (\ref{lightsotlemeq0}) in case $t=0$, we similarly may restrict to the case $k=1$, and by the resolvent identity we have
    \begin{align}\label{lightsotlemeq1}
        S_{n}^{\ast}-S^{\ast}=S_{n}^{\ast}\left(\left(B+M_{0}\right)^{2}-\left(B_{n}+M_{0}\right)^{2}\right)S^{\ast}
    \end{align}
    Since $\overline{S_{n}^{\ast}M_{0}}$, $M_{0}S^{\ast}$ are uniformly bounded by Lemma \ref{lightboundedlem}, (\ref{lightsotlemeq1}) also converges in strong operator topology.

    It remains to investigate the SOTA convergence of (\ref{lightsotlemeq2}). It suffices to show convergence in strong operator convergence for $k=1$ since we may factorize and taking adjoints only changes $z$ to $\overline{z}$. The statement then follows as before by the resolvent identity and the uniform norm bounds provided by Lemma \ref{lightboundedlem}.
\end{proof}

We now present results on the type of operators which we will encounter as Schatten- von Neumann class factors in the proof of Proposition \ref{decomposelem}. 

\begin{Lemma}\label{schattlem}
	Let $B$ satisfy Hypothesis \ref{hyp1}, let $\phi\in C^{\infty}_{c}\left(\IR^{d}\right)$, $0\leq\phi\leq 1$, $\supp\phi\subseteq \overline{B_{1}\left(0\right)}$ and $\phi\equiv 1$ on $\overline{B_{\frac{1}{2}}\left(0\right)}$, and let $\beta\in\IR$ satisfy $\left|\beta-\alpha\right|\leq 2N$, and $\left|\beta-1\right|\leq 2N$. Then there exist $\epsilon>0$ such that for all $\delta>0$,
	\begin{align}\label{schattlemeq1}
		\langle\di\rangle^{-1-\delta}\langle M_{0}\rangle^{\beta-\alpha}\langle X\rangle^{\epsilon}\left(\phi\Ii\left(\d B\right)+\left(1-\phi\right)\Ii c_{R}\left(\p_{R}B\right)\right)\langle M_{0}\rangle^{-\beta}
	\end{align}
	possesses a $S^{d}\left(L^2\left(\IR^{d},\IC^{r}\otimes H\right)\right)$-extension. On the other hand for all $\epsilon,\delta>0$,
	\begin{align}\label{schattlemeq2}
		\langle\di\rangle^{-1-\delta}\langle M_{0}\rangle^{\beta-\alpha}\langle X\rangle^{-\epsilon}\left(\phi\Ii\left(\di B\right)+\left(1-\phi\right)\Ii c_{R}\left(\p_{R} B\right)\right)\langle M_{0}\rangle^{-\beta},
	\end{align}
	and
	\begin{align}\label{schattlemeq3}
		\langle\di\rangle^{-1-\delta}\langle M_{0}\rangle^{\beta-\alpha}\langle X\rangle^{-\epsilon}\left(1-\phi\right)\Ii \left(\left(\di -c_{R}\p_{R} \right)B\right)\langle M_{0}\rangle^{-\beta}
	\end{align}
	possess $S^{d}\left(L^2\left(\IR^{d},\IC^{r}\otimes H\right)\right)$-extensions. Finally for all $\delta>0$,
	\begin{align}\label{schattlemeq4}
		\langle\di\rangle^{-1}\langle M_{0}\rangle^{\beta-\alpha}\Ii\left(\di B\right)\langle M_{0}\rangle^{-\beta},
	\end{align}
	possesses a $S^{d+\delta}\left(L^2\left(\IR^{d},\IC^{r}\otimes H\right)\right)$-extension.
\end{Lemma}

\begin{proof}
	We first consider (\ref{schattlemeq1}). We note that due to the compact support of $\phi$, for $\epsilon>0$ there is a constant $C$, such that
	\begin{align}
		&\left\|\langle M_{0}\rangle^{-\alpha+\beta}\langle X\rangle^{\epsilon}\left(\phi\Ii\left(\di B\right)+\left(1-\phi\right)\Ii c_{R}\left(\p_{R}B\right)\right)\langle M_{0}\rangle^{-\beta}\right\|_{L^{d}\left(S^{d}\left(H\right)\right)}\nonumber\\
		\leq&C\left(\tau^{\alpha,\beta,\nabla,d,1}\left(B\right)+\tau^{\alpha,\beta,\p_{R},d,1+2\epsilon}\left(B\right)\right),
	\end{align}
	which is finite for $\epsilon>0$ small enough, due to Hypothesis \ref{hyp1}. For (\ref{schattlemeq2}) and (\ref{schattlemeq3}) we similarly have for all $\epsilon>0$ a constant $C$, such that
	\begin{align}
		&\left\|\langle M_{0}\rangle^{\beta-\alpha}\langle X\rangle^{-\epsilon}\left(\phi\Ii\left(\di B\right)+\left(1-\phi\right)\Ii c_{R}\left(\p_{R} B\right)\right)\langle M_{0}\rangle^{-\beta}\right\|_{L^{d}\left(S^{d}\left(H\right)\right)}\nonumber\\
		&+\left\|\langle M_{0}\rangle^{\beta-\alpha}\langle X\rangle^{-\epsilon}\left(1-\phi\right)\Ii \left(\left(\di -c_{R}\p_{R} \right)B\right)\langle M_{0}\rangle^{-\beta}\right\|_{L^{d}\left(S^{d}\left(H\right)\right)}\leq C\tau^{\alpha,\beta,\nabla,d,1}\left(B\right),
	\end{align}
	which is also finite by Hypothesis \ref{hyp1}. For (\ref{schattlemeq4}) we have for all $\delta>0$, a constant $C$, such that
	\begin{align}
		\left\|\langle M_{0}\rangle^{\beta-\alpha}\Ii\left(\di B\right)\langle M_{0}\rangle^{-\beta}\right\|_{L^{d+\delta}\left(S^{d+\delta}\left(H\right)\right)}\leq C\tau^{\alpha,\beta,\nabla,d,1}\left(B\right),
	\end{align}
	which is finite by Hypothesis \ref{hyp1}. The claims now follow from Corollary \ref{simoninterpolcor}, and $\langle X\rangle^{-1-\delta}\in L^{d}\left(\IR^{d}\right)$ respectively $\langle X\rangle^{-1}\in L^{d+\delta}\left(\IR^{d}\right)$, for all $\delta>0$.
\end{proof}

In a similar manner one shows the weaker result under weaker conditions on the operator family $B$ below.

\begin{Lemma}\label{lightschattlem}
    Let $B:\IR^{d}\rightarrow B\left(H\right)$ be a point-wise self-adjoint $C^{1}$-function with respect to the strong operator topology, and assume that there exists $\eta>0$, such that
    \begin{align}
        \tau^{0,0,\nabla,d,1}\left(B\right)+\tau^{0,0,\p_{R},d,1+\eta}\left(B\right)<\infty.
    \end{align}
    Let $\phi\in C^{\infty}_{c}\left(\IR^{d}\right)$, $0\leq\phi\leq 1$, $\supp\phi\subseteq \overline{B_{1}\left(0\right)}$, and $\phi\equiv 1$ on $\overline{B_{\frac{1}{2}}\left(0\right)}$. Then there exist $\epsilon>0$ such that for all $\delta>0$,
	\begin{align}\label{lightschattlemeq1}
		\langle\di\rangle^{-1-\delta}\langle X\rangle^{\epsilon}\left(\phi\Ii\left(\di B\right)+\left(1-\phi\right)\Ii c_{R}\left(\p_{R} B\right)\right)
	\end{align}
	possesses a $S^{d}\left(L^2\left(\IR^{d},\IC^{r}\otimes H\right)\right)$-extension. On the other hand for all $\epsilon,\delta>0$,
	\begin{align}\label{lightschattlemeq2}
		\langle\di\rangle^{-1-\delta}\langle X\rangle^{-\epsilon}\left(\phi\Ii\left(\di B\right)+\left(1-\phi\right)\Ii c_{R}\left(\p_{R} B\right)\right),
	\end{align}
	and
	\begin{align}\label{lightschattlemeq3}
		\langle\di\rangle^{-1-\delta}\langle X\rangle^{-\epsilon}\left(1-\phi\right)\Ii\left(\left(\di -c_{R}\p_{R}\right)B\right)
	\end{align}
	possess $S^{d}\left(L^2\left(\IR^{d},\IC^{r}\otimes H\right)\right)$-extensions. Finally for all $\delta>0$,
	\begin{align}\label{lightschattlemeq4}
		\langle\di\rangle^{-1}\Ii\left(\di B\right),
	\end{align}
	possesses a $S^{d+\delta}\left(L^2\left(\IR^{d},\IC^{r}\otimes H\right)\right)$-extension.
\end{Lemma}

In the next Lemma we prepare the upcoming factorizations of Proposition \ref{decomposelem}. We will utilize the following notation.

\begin{Definition}
    Let $l\in\IN$. For $Y$ a Banach space, define the shuffle product $\shuffle$ between $l+1$-tuples of operators $T=\left(T_{0},\ldots,T_{l}\right)$ in $Y$, and $l$-tuples of operators $S=\left(S_{1},\ldots,S_{l}\right)$ in $Y$ by
	\begin{align}
		T\shuffle S:=T_{0}S_{1}T_{1}\ldots S_{l}T_{l}.
	\end{align}
    For $T$ a operator in $Y$, $S=\left(S_{1},\ldots,S_{l}\right)$ a $l$-tuple of operators in $Y$, and $k\in\left\{1,\ldots,l+1\right\}$, denote
    \begin{align}
        \iota_{T}^{k}S:=\left(S_{1},\ldots,S_{k-1},T,S_{k},\ldots,S_{l}\right),\ k\leq l,\nonumber\\
        \iota_{T}^{l+1}S:=\left(S_{1},\ldots,S_{l},T\right).
    \end{align}
    Denote for a sequence of operators $\left(S_{k}\right)_{k\in\IN_{0}}$, and $\eta\in\IN_{0}^{l}$ by $S^{\eta}$ the $l$-tuple of operators
    \begin{align}
        S^{\eta}:=\left(S_{\eta_{1}},\ldots,S_{\eta_{1}}\right).
    \end{align}
    Finally denote $\one:=\left(1,\ldots,1\right)\in\IN^{l}$.
\end{Definition}

\begin{Lemma}\label{neulem1}
	Assume Hypothesis \ref{hyp1}. Let $z\in\IC\backslash\left(-\infty,0\right]$. Then,
	\begin{align}
		&\tr_{\IC^{r}}\left(R_{z}\left(D_{B}^{\ast}D_{B}\right)^{N}-R_{z}\left(D_{B}D_{B}^{\ast}\right)^{N}\right)=\tr_{\IC^{r}}\sum_{\stackrel{\gamma\in\IN^{d+1}}{\left|\gamma\right|=N+d}}\binom{N-1}{\gamma-\one}^{-1}R_{z}\left(H_{B}\right)^{\gamma}\shuffle\left(\Ii\di B\right)^{\otimes d}\nonumber\\
		+&\tr_{\IC^{r}}\sum_{\stackrel{\gamma\in\IN^{d+3}}{\left|\gamma\right|=N+d+2}}\binom{N-1}{\gamma-\one}^{-1}\left(
\left(R_{z}\left(H_{B}\right)^{\otimes\left(d+2\right)},R_{z}\left(D_{B}^{\ast}D_{B}\right)\right)^{\gamma}\shuffle\left(\Ii\di B\right)^{\otimes\left(d+2\right)}\right.\nonumber\\
&\left.+\left(R_{z}\left(H_{B}\right)^{\otimes\left(d+2\right)},R_{z}\left(D_{B}D_{B}^{\ast}\right)\right)^{\gamma}\shuffle\left(\Ii\di B\right)^{\otimes\left(d+2\right)}\right).
	\end{align}
\end{Lemma}

\begin{proof}
	By Hypothesis \ref{hyp1}, and Proposition \ref{hdomainprop}, we have for $w=\mathrm{dist}\left(z,\left(-\infty,0\right]\right)$ large enough,
	\begin{align}
		\left\|\Ii\left(\di B\right)R_{z}\left(H_{B}\right)\right\|_{B\left(L^2\left(\IR^{d},\IC^{r}\otimes H\right)\right)}\leq w^{-\frac{1}{2}}\rho^{1}_{z}\left(B\right)<1.
	\end{align}
	We may develop the Neumann series of $R_{z}\left(D_{B}D_{B}^{\ast}\right)$ and $R_{z}\left(D^{\ast}_{B}D_{B}\right)$, and obtain
	\begin{align}
		&R_{z}\left(D^{\ast}_{B}D_{B}\right)-R_{z}\left(D_{B}D_{B}^{\ast}\right)=2R_{z}\left(H_{B}\right)\sum_{k=0}^{\infty}\left(\Ii\left(\di B\right)R_{z}\left(H_{B}\right)\right)^{2k+1},\nonumber\\
		&R_{z}\left(D_{B}^{\ast}D_{B}\right)+R_{z}\left(D_{B}D_{B}^{\ast}\right)=2R_{z}\left(H_{B}\right)\sum_{k=0}^{\infty}\left(\Ii\left(\di B\right)R_{z}\left(H_{B}\right)\right)^{2k}.
	\end{align}
	So by Lemma \ref{cliflem} we have for $w$ large enough,
	\begin{align}\label{neulem1eq1}
		\tr_{\IC^{r}}\left(R_{z}\left(D^{\ast}_{B}D_{B}\right)-R_{z}\left(D_{B}D_{B}^{\ast}\right)\right)&=2\tr_{\IC^{r}}R_{z}\left(H_{B}\right)\left(\Ii\left(\di B\right)R_{z}\left(H_{B}\right)\right)^{d}\sum_{k=0}^{\infty}\left(\Ii\left(\di B\right)R_{z}\left(H_{B}\right)\right)^{2k}\nonumber\\
		&=\tr_{\IC^{r}}\left(R_{z}\left(H_{B}\right)\Ii\left(\di B\right)\right)^{d}\left(R_{z}\left(D^{\ast}_{B}D_{B}\right)+R_{z}\left(D_{B}D_{B}^{\ast}\right)\right).
	\end{align}
	both sides are holomorphic on $\IC\backslash\left[0,\infty\right)$, which implies that the above identity extends to this domain. Differentiating (\ref{neulem1eq1}) $\left(N-1\right)$-times in $z$, yields the claim.
\end{proof}

We now arrive at the principal decomposition result. However we will first introduce further abbreviations for the involved operators.

\begin{Definition}\label{factordef}
    Let $k,l,m\in\IN_{0}$, $\beta\in\IR^{l+1}$, $\gamma\in\IN^{l+1}$, $\epsilon\in\IR^{l}$, $\zeta\in\left\{0,1,2\right\}^{l}$, $\alpha\geq 0$, and $\delta,\xi>0$. For a operator family $B$ satisfying Hypothesis \ref{hyp1} define the operators
    \begin{align}
        T\left(B\right)^{0}:=&\phi\Ii\left(\di B\right)+\left(1-\phi\right)\Ii c_{R}\left(\p_{R}B\right),\\
        T\left(B\right)^{1}:=&\left(1-\phi\right)\Ii\left(\left(\di-c_{R}\p_{R}\right)B\right),\\
        T\left(B\right)^{2}:=&\Ii\left(\di B\right)=T\left(B\right)^{0}+T\left(B\right)^{1},\\
        L:=&\left[\Delta,\left(1-\psi\right)c_{R}\right].
    \end{align}
    Now define the operators for $m\in\left\{1,\ldots,\gamma_{k}\right\}$,
    \begin{align}
        u_{z}^{a,b,c,d,k}\left(B\right):=&\langle M_{0}\rangle^{a}\langle X\rangle^{b}R_{z}\left(H_{B}\right)^{k}\langle X\rangle^{-b}\langle M_{0}\rangle^{c}\langle\di\rangle^{d},\ a,b,c,d\in\IR,\ k\in\IN,\nonumber\\
        \left(U_{z}\left(B\right)^{\alpha,\beta,\gamma,\delta,\epsilon}\right)_{i}:=&u_{z}^{\beta_{i},\epsilon_{i},2\gamma_{i}-1-\delta-\beta_{i},1+\delta,\gamma_{i}}\left(B\right),\ 0\leq i\leq l-1,\nonumber\\
        \left(U_{z}\left(B\right)^{\alpha,\beta,\gamma,\delta,\epsilon}\right)_{l}:=&u_{z}^{\beta_{l},0,0,0,\gamma_{l}}\left(B\right),\nonumber\\
        \left(V\left(B\right)^{\alpha,\beta,\delta,\epsilon,\zeta}\right)_{i}:=&\langle\di\rangle^{-1-\delta}\langle M_{0}\rangle^{\beta_{i}-\alpha}\langle X\rangle^{\epsilon_{i}}T\left(B\right)^{\zeta_{i}}\langle M_{0}\rangle^{-\beta_{i}},\ i\in\left\{1,\ldots,l\right\},\nonumber\\
        \left(W_{z}\left(B\right)^{\alpha,\beta,\gamma,\delta,\xi,k,m}\right)_{i}:=&u_{z}^{\beta_{i},-i\xi,2\gamma_{i}-1-\delta-\beta_{i},1+\delta,\gamma_{i}}\left(B\right),\ 0\leq i<k\leq d,\nonumber\\
        \left(W_{z}\left(B\right)^{\alpha,\beta,\gamma,\delta,\xi,k,m}\right)_{i}:=&u_{z}^{\beta_{i},-\left(d-i\right)\xi,2\gamma_{i}-1-\delta-\beta_{i},1+\delta,\gamma_{i}}\left(B\right),\ k<i<d,\nonumber\\
        \left(W_{z}\left(B\right)^{\alpha,\beta,\gamma,\delta,\xi,k,m}\right)_{d}:=&u_{z}^{\beta_{d},0,0,0,\gamma_{d}}\left(B\right),\ k<d,\nonumber\\
        \left(W_{z}\left(B\right)^{\alpha,\beta,\gamma,\delta,\xi,k,m}\right)_{k}:=&u_{z}^{\beta_{k},k\xi,2m-1-\beta_{k},1,m}\left(B\right)\left(\langle\di\rangle^{-1}\langle X\rangle^{k\xi}L\langle X\rangle^{\left(d-k\right)\xi}\right)\nonumber\\
        &u_{z}^{\beta_{k}+1-2m,-\left(d-k\right)\xi,2\gamma_{k}-1-\beta_{k}-\delta,1+\delta,\gamma_{k}-m+1}\left(B\right),\ 0\leq k\leq d-1,\nonumber\\
        \left(W_{z}\left(B\right)^{\alpha,\gamma,\delta,\xi,d,m}\right)_{d}:=&u_{z}^{\beta_{d},d\xi,2m-1-\beta_{d},1,m}\left(B\right)\left(\langle\di\rangle^{-1}\langle X\rangle^{d\xi}L\right)u_{z}^{\beta_{d}+1-2m,0,0,0,\gamma_{d}-m+1}\left(B\right).
    \end{align}
\end{Definition}

The next Proposition is the central result of this chapter. It is similar to the central result obtained in \cite{F1}, where the statement below was shown without the presence of the partial trace $\tr_{\IC^{r}}$, but with much stronger assumptions on the family $A$, which are not sensible for the calculation of the principal trace. In fact, if we would take the intersection of Hypothesis \ref{hyp1} and the assumptions made in \cite{F1}, we would always end up in the trivial case
\begin{align}
	\tr_{L^{2}\left(\IR^{d},H\right)}\tr_{\IC^{r}}\left(e^{-tD^{\ast}D}-e^{-tDD^{\ast}}\right)=0.
\end{align}

\begin{Proposition}\label{decomposelem}
    There exists a universal constant $C$ such that the following holds true. Assume Hypothesis \ref{hyp1}. Let $\psi,\phi\in C^{\infty}_{c}\left(\IR^{d}\right)$, with $0\leq\psi,\phi\leq 1$, $\supp\psi\subseteq\overline{B_{\frac{1}{2}}\left(0\right)}$, and $\psi\equiv 1$ on $\overline{B_{\frac{1}{4}}\left(0\right)}$, and $\supp\phi\subseteq\overline{B_{1}\left(0\right)}$, and $\phi\equiv 1$ on $\overline{B_{\frac{1}{2}}\left(0\right)}$.
    Then there exist, independent of the choices of $B$, $\psi$, $\phi$, and $z\in\mathcal{D}_{C}\left(B\right)$, multi-indices $\beta\in\left[\alpha-2N,\alpha+2N\right]\cap\left[1-2N,-1+2N\right]$, $\epsilon^{j},\eta^{j}\in\left[-1,1\right]^{d}$ for $j\in\left\{1,\ldots,5\right\}$, constants $\delta,\xi\in\left[0,1\right]$, and $C_{\gamma}^{j},C_{\gamma,k}<\infty$ for $\gamma\in\IN^{d+1}$, $k\in\left\{0,\ldots,d\right\}$, $j\in\left\{1,2,3\right\}$, such that
    \begin{align}\label{decomposelemmaineq}
        &\tr_{\IC^{r}}\left(R_{z}\left(D_{B}^{\ast}D_{B}\right)^{N}-R_{z}\left(D_{B}D_{B}^{\ast}\right)^{N}\right)\nonumber\\
        =&\tr_{\IC^{r}}\left[\sum_{\stackrel{\gamma\in\IN^{d+1},}{\left|\gamma\right|=N+d}}C_{\gamma}^{1}\sum_{\stackrel{\zeta\in\left\{0,1\right\}^{d},}{\left|\zeta\right|\leq d-1}}U_{z}^{\alpha,\beta,\gamma,\delta,\epsilon^{1}}\left(B\right)\shuffle V^{\alpha,\beta,\delta,\eta^{1},\zeta}\left(B\right)\right]\nonumber\\
        &+\tr_{\IC^{r}}\left[\sum_{\stackrel{\gamma\in\IN^{d+1},}{\left|\gamma\right|=N+d}}C_{\gamma}^{2}\left(\left(1-\phi\right)U_{z}\left(B\right)^{\alpha,\beta,\gamma,\delta,\epsilon^{2}}\shuffle V\left(B\right)^{\alpha,\beta,\delta,\eta^{2},\one}\phi\right.\right.\nonumber\\
        &\left.\left.+\phi U_{z}\left(B\right)^{\alpha,\beta,\gamma,\delta,\epsilon^{3}}\shuffle V\left(B\right)^{\alpha,\beta,\delta,\eta^{3},\one}\right)\right]\nonumber\\
        &+\tr_{\IC^{r}}\left[c_{R}\left(1-\phi\right)\sum_{\stackrel{\gamma\in\IN^{d+1},}{\left|\gamma\right|=N+d}}\sum_{k=0}^{d}C_{\gamma,k}\sum_{m=1}^{\gamma_{k}}W_{z}\left(B\right)^{\alpha,\beta,\gamma,\delta,\xi,k,m}\shuffle V\left(B\right)^{\alpha,\beta,\delta,\eta^{4},\one}\right]\nonumber\\
        &+\tr_{\IC^{r}}\left[\sum_{\stackrel{\gamma=\left(\widehat{\gamma},\gamma_{d+2}\right)\in\IN^{d+2}\times\IN}{\left|\gamma\right|=N+d+2}}C_{\gamma}^{3}\left(U_{z}^{\alpha,\beta,\widehat{\gamma},0,0}\left(B\right)\shuffle V^{\alpha,\beta,0,0,2\cdot\one}\right)\right.\nonumber\\
        &\left.\cdot\Ii\left(\di B\right)\left(R_{z}\left(D_{B}^{\ast}D_{B}\right)^{\gamma_{d+2}}+R_{z}\left(D_{B}D_{B}^{\ast}\right)^{\gamma_{d+2}}\right)\right],\ z\in\mathcal{D}_{C}\left(B\right),
    \end{align}
    where each summand in $\left[\cdot\right]$-brackets is a trace-class operator in $L^2\left(\IR^{d},\IC^{r}\otimes H\right)$ with trace norm bound given by a $z$-independent polynomial in
    \begin{align}\label{decomposelemeq0}
        \tau^{\alpha,\beta,\nabla,d,1}\left(B\right)+\tau^{\alpha,\beta,\p_{R},d,1+\chi}\left(B\right),
    \end{align}
    for $\beta$ running over $z$-independent finite subsets of $\left[\alpha-2N,\alpha+2N\right]\cap\left[1-2N,-1+2N\right]$, and some $z$-independent $\chi>0$ small enough, such that (\ref{decomposelemeq0}) is finite.
\end{Proposition}

\begin{proof}
    We continue with the formula from Lemma \ref{neulem1} and further decompose the involved terms. First, we start with the sum indexed over $\gamma\in\IN^{d+1}$, $\left|\gamma\right|=N+d$, whose summands, up to factors only dependent on $\gamma$ and $N$, are
    \begin{align}\label{decomposelemeq1}
        R_{z}\left(H_{B}\right)^{\gamma}\shuffle\left(\Ii\di B\right)^{\otimes d}.
    \end{align}
    If we expand $\Ii\left(\di B\right)=T\left(B\right)^{2}=T\left(B\right)^{0}+T\left(B\right)^{1}$ in (\ref{decomposelemeq1}), we obtain $2^{d}-1$ summands, where at least one $T\left(B\right)^{0}$ appears, and one summand which only involves $T\left(B\right)^{1}$. The summands which do involve $T\left(B\right)^{0}$ can be decomposed as
    \begin{align}
        U_{z}^{\alpha,\beta,\gamma,\delta,\epsilon^{1}}\left(B\right)\shuffle V^{\alpha,\beta,\delta,\epsilon^{1},\zeta}\left(B\right),
    \end{align}
        where $\zeta\in\left\{0,1\right\}^{d}$, $\left|\zeta\right|\leq d-1$, and $\beta$, $\delta$, $\epsilon^{1}$, and $\eta^{1}$ are chosen as follows. Let $j\in\left\{0,\ldots,d\right\}$ such that $\zeta_{j}=0$. For $\delta>0$ set $\beta_{0}:=0$, $\beta_{i+1}:=\beta_{i}+\alpha+\delta+1-2\gamma_{i}$. For $\xi>0$ set $\epsilon_{i}^{1}:=i\xi$ for $0\leq i\leq j-1$, and $\epsilon_{i}^{1}:=-\left(d-i\right)\xi$ for $j\leq i\leq d-1$. Set $\eta_{i}^{1}:=-\xi$ for $i\neq j$, and $\eta_{j}^{1}:=d\xi$. Chose $\xi>0$ small enough such that $\overline{\left(V^{\alpha,\beta,\delta,\eta^{1},\zeta}\left(B\right)\right)_{i}}\in S^{d}\left(L^2\left(\IR^{d},\IC^{r}\otimes H\right)\right)$ for all $i$, which exists according to Lemma \ref{schattlem}. Lemma \ref{boundedlem} on the other hand ensures that $\overline{\left(U_{z}^{\alpha,\beta,\gamma,\delta,\epsilon^{1}}\right)_{i}}\in B\left(L^2\left(\IR^{d},\IC^{r}\otimes H\right)\right)$, for $0\leq i\leq d-1$, with operator norms bounded by $4^{\gamma_{i}}$ for $w$ satisfying (\ref{decomposelemeq0}). Since $\left(\alpha+1\right)d<2\left(N+d\right)$ by Hypothesis \ref{hyp1}, we may choose $\delta>0$ small enough such that $\beta_{d}<2\gamma_{d}$, which then also implies that $\overline{\left(U_{z}^{\alpha,\beta,\gamma,\delta,\epsilon^{1}}\right)_{d}}\in B\left(L^2\left(\IR^{d},\IC^{r}\otimes H\right)\right)$ by Lemma \ref{boundedlem}.

        Next, we decompose
        \begin{align}T:=R_{z}\left(H_{B}\right)^{\gamma}\shuffle\left(T\left(B\right)^{1}\right)^{\otimes d},
        \end{align}
        via $T=\left(1-\phi\right)T\phi+\phi T+\left(1-\phi\right)T\left(1-\phi\right)$. Those summands which do involve a factor $\phi$ can be written as
        \begin{align}
            \left(1-\phi\right)\left[U_{z}\left(B\right)^{\alpha,\beta,\gamma,\delta,\epsilon^{2}}\shuffle V\left(B\right)^{\alpha,\beta,\delta,\eta^{2},\one}\right]\phi,
        \end{align}
        and
        \begin{align}
            \phi\left[U_{z}\left(B\right)^{\alpha,\beta,\gamma,\delta,\epsilon^{3}}\shuffle V\left(B\right)^{\alpha,\beta,\delta,\eta^{3},\one}\right],
        \end{align}
        where we set $\epsilon_{i}^{2}:=-i\xi$, and $\epsilon_{d}^{2}:=d\xi$, $\epsilon_{i}^{3}:=-\left(d-i\right)\xi$, and $\epsilon_{0}^{3}:=d\xi$, and $\eta_{i}^{2}:=\eta_{i}^{3}:=-\xi$. If we choose $\delta,\xi>0$ sufficiently small, the same reason as above imply that we have $\overline{\left(U_{z}^{\alpha,\beta,\gamma,\delta,\epsilon^{2}}\right)_{i}}\in B\left(L^2\left(\IR^{d},\IC^{r}\otimes H\right)\right)$, for $0\leq i\leq d-1$, $\overline{\left(U_{z}^{\alpha,\beta,\gamma,\delta,\epsilon^{2}}\right)_{d}\phi}\in B\left(L^2\left(\IR^{d},\IC^{r}\otimes H\right)\right)$, $\overline{\phi\left(U_{z}^{\alpha,\beta,\gamma,\delta,\epsilon^{3}}\right)_{0}}\in B\left(L^2\left(\IR^{d},\IC^{r}\otimes H\right)\right)$, and $\overline{\left(U_{z}^{\alpha,\beta,\gamma,\delta,\epsilon^{3}}\right)_{i}}\in B\left(L^2\left(\IR^{d},\IC^{r}\otimes H\right)\right)$ for $1\leq i\leq d$, where the operator norm is bounded by $4^{\gamma_{i}}$. We also have $\overline{\left(V^{\alpha,\beta,\delta,\eta^{2},\one}\right)_{i}}$, $\overline{\left(V^{\alpha,\beta,\delta,\eta^{3},\one}\right)_{i}}\in S^{d}\left(L^2\left(\IR^{d},\IC^{r}\otimes H\right)\right)$.

        For the remaining summand $\left(1-\phi\right)T\left(1-\phi\right)$ requires more work. First note that
	\begin{align}
		\left[H_{B},\left(1-\psi\right)c_{R}\right]=\left[\Delta,\left(1-\psi\right)c_{R}\right]=L,
	\end{align}
	and thus
	\begin{align}
		\left[\left(1-\psi\right)c_{R},R_{z}\left(H_{B}\right)^{n}\right]=\sum_{m=1}^{n}R_{z}\left(H_{B}\right)^{m}LR_{z}\left(H_{B}\right)^{n-m+1}.
	\end{align}
	On the other hand $\left(1-\psi\right)c_{R}$ anti-commutes with $T\left(B\right)^{1}$, since $c_{R}$ anti-commutes with $\di-c_{R}\p_{R}$, because $\p_{R}$ is a component in a orthonormal system given by coordinates.
	Consequently, since there is a odd number of factors $T\left(B\right)^{1}$, we obtain, with $\left\{\cdot,\cdot\right\}$ denoting the anti-commutator,
	\begin{align}
		\left\{\left(1-\psi\right)c_{R},T\right\}=\sum_{k=0}^{d}\left(-1\right)^{k}\sum_{m=1}^{\gamma_{k}}R_{z}\left(H_{B}\right)^{\left(\gamma_{0},\ldots,\gamma_{k-1},m,\gamma_{k}-m+1,\gamma_{k+1},\ldots,\gamma_{d}\right)}\shuffle\left(\iota_{L}^{k+1}\left(T\left(B\right)^{1}\right)^{\otimes d}\right).
	\end{align}
	Applying $\tr_{\IC^{r}}$, and bringing $c_{R}$ around, we thus obtain
	\begin{align}
		&\tr_{\IC^{r}}\left(1-\phi\right)R_{z}\left(H_{B}\right)^{\gamma}\shuffle\left(T\left(B\right)^{1}\right)^{\otimes d}\left(1-\phi\right)\nonumber\\
  =&-\frac{1-\phi}{2}\tr_{\IC^{r}}c_{R}\sum_{k=0}^{d}\left(-1\right)^{k}\sum_{m=1}^{\gamma_{k}}R_{z}\left(H_{B}\right)^{\left(\gamma_{0},\ldots,\gamma_{k-1},m,\gamma_{k}-m+1,\gamma_{k+1},\ldots,\gamma_{d}\right)}\shuffle\left(\iota_{L}^{k+1}\left(T\left(B\right)^{1}\right)^{\otimes d}\right)\nonumber\\
  =&\tr_{\IC^{r}}c_{R}\left(1-\phi\right)\sum_{k=0}^{d}c_{\gamma,k}\sum_{m=1}^{\gamma_{k}}W_{z}\left(B\right)^{\alpha,\beta,\gamma,\delta,\xi,k,m}\shuffle V\left(B\right)^{\alpha,\beta,\delta,\eta_{4},\one},
	\end{align}
    where we set $\eta^{4}_{i}:=-\xi$. As before $\overline{\left(V^{\alpha,\beta,\delta,\eta^{4},\one}\right)_{i}}\in S^{d}\left(L^2\left(\IR^{d},\IC^{r}\otimes H\right)\right)$, and $\overline{\left(W_{z}^{\alpha,\beta,\gamma,\delta,\xi,k,m}\right)_{i}}\in B\left(L^2\left(\IR^{d},\IC^{r}\otimes H\right)\right)$ for $i\neq k$, with operator norm bounded by $4^{\gamma_{i}}$. For $i=k$ the operator $\left(W_{z}^{\bullet}\left(B\right)\right)_{k}$ contains two factors of the type $u_{z}^{\bullet}\left(B\right)$, which both have bounded extensions with operator norm bounded by $4^{m}$, and $4^{\gamma_{k}-m}$ respectively, by Lemma \ref{boundedlem}. The remaining factor of $\left(W_{z}^{\bullet}\left(B\right)\right)_{k}$ is
    \begin{align}\label{decomposelemeq2}
        \langle\di\rangle^{-1}\langle X\rangle^{k\xi}\left[\Delta,\left(1-\psi\right)c_{R}\right]\langle X\rangle^{\left(d-k\right)\xi}.
    \end{align}
    Since $\left(1-\psi\right)c_{R}\in Sym_{0}$, it follows that $\langle X\rangle^{k\xi}\left[\Delta,\left(1-\psi\right)c_{R}\right]\langle X\rangle^{\left(d-k\right)\xi}$ is a first order operator with coefficients in $Sym_{-1+d\xi}\subseteq Sym_{0}$ for $\xi>0$ small enough. In particular (\ref{decomposelemeq2}) has a bounded operator extension with operator norm $C$. Therefore $\overline{\left(W_{z}^{\alpha,\beta,\gamma,\delta,\xi,k,m}\right)_{k}}\in B\left(L^2\left(\IR^{d},\IC^{r}\otimes H\right)\right)$ with operator norm bounded by $4^{\gamma_{k}}C$.

    Finally, in the formula of Lemma \ref{neulem1} we consider the remaining sum indexed by $\gamma\in\IN^{d+3}$, $\left|\gamma\right|=N+d+2$. The summands are up to a constant of the form
    \begin{align}
    \left(R_{z}\left(H_{B}\right)^{\otimes\left(d+2\right)},R_{z}\left(T\right)\right)^{\gamma}\shuffle\left(\Ii\di B\right)^{\otimes\left(d+2\right)},\ T\in\left\{DD^{\ast},D^{\ast}D\right\},
    \end{align}
    which we factorize into
    \begin{align}
        \left(U_{z}^{\alpha,\beta,\widehat{\gamma},0,0}\left(B\right)\shuffle V^{\alpha,\beta,0,0,2\cdot\one}\right)\Ii\left(\di B\right)\langle M_{0}\rangle^{-1}\langle M_{0}\rangle R_{z}\left(T\right)^{\gamma_{d+2}},
    \end{align}
    where $\widehat{\gamma}:=\left(\gamma_{0},\ldots,\gamma_{d+1}\right)$, $\beta_{0}^{2}:=0$, and $\beta_{i+1}^{2}:=\beta_{i}^{2}+\alpha+1-\gamma_{i}$. As noted before we have
    $\overline{\left(U_{z}^{\alpha,\beta,\widehat{\gamma},0,0}\left(B\right)\right)_{i}}\in B\left(L^2\left(\IR^{d},\IC^{r}\otimes H\right)\right)$ with operator norm bounded by $4^{\gamma_{i}}$ by Lemma \ref{boundedlem}. Also $\Ii\di^{E}B\langle M_{0}\rangle^{-1}\in B\left(L^2\left(\IR^{d},\IC^{r}\otimes H\right)\right)$ by Hypothesis \ref{hyp1}, and the fact that for measurable $S$, $\left\|S\right\|_{B\left(L^2\left(\IR^{d},\IC^{r}\otimes H\right)\right)}=\left\|S\right\|_{L^{\infty}\left(\IR^{d},B\left(\IC^{r}\otimes H\right)\right)}$. Furthermore also $\langle M_{0}\rangle R_{z}\left(T\right)^{\gamma_{d+2}}$ is bounded with operator norm bounded by $4^{\gamma_{d+2}}$ by Lemma \ref{boundedlem}. On the other hand Lemma \ref{schattlem} implies that $\overline{V^{\alpha,\beta,0,0,2\cdot\one}_{i}}\in S^{d+1}\left(L^2\left(\IR^{d},\IC^{r}\otimes H\right)\right)$.

    We note that by the proof of Lemma \ref{schattlem} there exist $z$-independent finite sets
    \begin{align}
        I\subset\left[\alpha-2N,2N+\alpha\right]\cap\left[1-2N,-1+2N\right],
    \end{align}
    and $z$-independent constants $C$, and $\chi>0$ small enough, such that the $S^{d}$- respectively $S^{d+1}$-norms of the various operators $\overline{V^{\bullet}\left(B\right)}$ can be bounded uniformly by
    \begin{align}
        C\sum_{\beta\in I,\chi\in J}\left(\tau^{\alpha,\beta,\nabla,d,1}\left(B\right)+\tau^{\alpha,\beta,\p_{R},d,1+\chi}\left(B\right)\right)<\infty.
    \end{align}
\end{proof}

To close this chapter, we show that the partial trace of the semi-group difference of $DD^{\ast}$ and $D^{\ast}D$ is trace-class. 

\begin{Proposition}\label{exptraceprop}
    Let $B$ satisfy Hypothesis \ref{hyp1}, and let $t>0$. Then $\tr_{\IC^{r}}\left(e^{-tD^{\ast}_{B}D_{B}}-e^{-tD_{B}D^{\ast}_{B}}\right)\in S^1\left(L^2\left(\IR^{d},H\right)\right)$.
\end{Proposition}

\begin{proof}
	Let $C>0$ be large enough, and let $\rho\in\left(0,\frac{\pi}{2}\right)$ be small enough. Let $\Gamma:=\Gamma_{-}\circ\Gamma_{+}^{-1}$, where $\Gamma_{\pm}\left(s\right):=-C\pm e^{\Ii\rho}s$, $s\in\left[0,\infty\right)$, and $\circ$, ${}^{-1}$ denote the composition of paths, respectively the reversal of a path. Note that $\Gamma$ encloses the non-negative real axis once. Then for $\lambda\geq 0$, we have by Cauchy's integral formula
	\begin{align}
		e^{-t\lambda}=\p_{\lambda}^{N-1}\left(\left(-t\right)^{1-N}e^{-t\lambda}\right)=\frac{\left(N-1\right)!\left(-t\right)^{1-N}}{2\pi\Ii}\int_{\Gamma}\frac{e^{-tz}}{\left(z-\lambda\right)^{N}}\Id z.
	\end{align}
	Via holomorphic functional calculus we thus have
	\begin{align}
		e^{-tD^{\ast}_{B}D_{B}}-e^{-tD_{B}D_{B}^{\ast}}=\frac{\left(N-1\right)!\left(-t\right)^{1-N}}{2\pi\Ii}\int_{\Gamma}e^{-tz}\left(R_{-z}\left(D_{B}D^{\ast}_{B}\right)^{N}-R_{-z}\left(D^{\ast}_{B}D_{B}\right)^{N}\right)\Id z,
	\end{align}
	where the integral exists as a Bochner integral in $B\left(L^2\left(\IC^{r}\otimes H\right)\right)$ in the norm topology. Therefore we may interchange $\tr_{\IC^{r}}$ and integration to obtain
	\begin{align}\label{exptracepropeq1}
		&\tr_{\IC^{r}}\left(e^{-tD^{\ast}_{B}D_{B}}-e^{-tD_{B}D^{\ast}_{B}}\right)\nonumber\\
		=&\frac{\left(N-1\right)!\left(-t\right)^{1-N}}{2\pi\Ii}\int_{\Gamma}e^{-tz}\tr_{\IC^{r}}\left(R_{-z}\left(D_{B}D^{\ast}_{B}\right)^{N}-R_{-z}\left(D^{\ast}_{B}D_{B}\right)^{N}\right)\Id z.
	\end{align}
	We claim that the right hand side of (\ref{exptracepropeq1}) converges even as a Bochner integral in $S^1\left(L^2\left(\IR^{d},H\right)\right)$. Indeed, by Proposition \ref{decomposelem}, we have some constant $C'$, such that we may estimate the integrand by
	\begin{align}\label{exptracepropeq2}
		\left\|e^{-tz}\tr_{\IC^{r}}\left(R_{-z}\left(DD^{\ast}\right)^{N}-R_{-z}\left(D^{\ast}D\right)^{N}\right)\right\|_{S^1\left(L^2\left(\IR^{d},H\right)\right)}\leq C'e^{-t\left(-C+\cos\left(\rho\right)\left|s\right|\right)},\  z=\Gamma\left(s\right),
	\end{align}
	which is integrable over $s\in\IR$.
\end{proof}

\section{Approximation of the operator function $B$}

The goal of this chapter is to construct a approximating sequence $B_{\bullet}$ of the operator function $B$, which enables us to calculate the principal trace in the following chapters. We remark that a similar but simpler approximation argument has been used as a standard technique for example in \cite{Push} and \cite{GLMST} with the same goal. In the multidimensional setup in this paper we unfortunately need a more elaborate approximation procedure caused by the requirement that $\nabla B_{\bullet}$ should be sufficiently smooth and trace-class valued, a fact which comes for free in the one dimensional case.
Before we proceed with an approximation statement, we present the following convolution result which allows us smoothing of the operator family without destroying its decay properties.

\begin{Definition}\label{mollifierdef}
    Let $O(d)$ be the orthogonal group in dimension $d$ with (bi-invariant) Haar measure $m$, and let $\left(g_{n}\right)_{n\in\IN}\subset C^{\infty}\left(O(d)\right)$ be a delta sequence with supports decreasing to $\left\{\one\right\}$ of $O(d)$ with respect to the convolution given by the Haar measure\footnote{Since $O(d)$ is a smooth Lie group the usual construction of approximate units in $L^1(G)$ of locally compact groups (cf. \cite{Dixmier}[Chapter 13.2]) can be refined to be smooth.}. For $\psi^{d}\in C_{c}^{\infty}\left(\IR^{d}\right)$ non-negative with $\int_{\IR}\psi\left(x\right)\Id x=1$ let $\psi_{n}^{d}\left(x\right):=n^{d}\psi^{d}\left(nx\right)$ be the associated delta sequence of $\IR^{d}$. Let $\phi\in C^{\infty}_{c}\left(\IR^{d}\right)$ with $\phi\equiv 1$ near $0$. For $\epsilon>0$ let
    \begin{align}
        \IF_{\epsilon}&:=\left\{f\in C^{1}\left(\IR^{d},\IC^{k\times k}\right)\left|\left\|f\right\|_{\IF_{\epsilon}}<\infty\right.\right\},\nonumber\\
        \left\|f\right\|_{\IF_{\epsilon}}&:=\left\|f\right\|_{L^{\infty}}+\left\|\langle X\rangle\nabla f\right\|_{L^{\infty}}+\left\|\langle X\rangle^{1+\epsilon}\p_{R}f\right\|_{L^{\infty}}.
    \end{align}
    For $f\in\IF_{\epsilon}$ define
    \begin{align}\label{radialconvolutionlemeq0}
\delta_{n}\left(f\right)\left(x\right):=&\phi\left(x\right)\int_{\IR^{d}}\psi_{n}^{d}\left(x-y\right)f\left(y\right)\Id y\nonumber\\
        &+\left(1-\phi\left(x\right)\right)\int_{0}^{\infty}\int_{O(d)}\psi^{1}_{n}\left(\log r\right)g_{n}\left(Y\right)f\left(r^{-1}Y^{\ast}x\right)\Id m\left(Y\right)\frac{\Id r}{r}.
        \end{align}
\end{Definition}

\begin{Lemma}\label{radialconvolutionlem}
        For $\epsilon>0$ the operator family $\delta_{n}$ is uniformly bounded in $\IF_{\epsilon}$, and $\delta_{n}\left(f\right)$ respectively $\nabla \delta_{n}\left(f\right)$ converge point-wise to $f$ respectively $\nabla f$. For $f\in\mathcal{F}_{\epsilon}$ the function $\delta_{n}\left(f\right)$ is smooth with bounded derivatives for all $n\in\IN$.
\end{Lemma}

\begin{proof}
    Let $f_{n}:=\delta_{n}\left(f\right)$. For $x_{0}\in\IR^{d}\backslash\left\{0\right\}$ extend $e_{1}:=\frac{x_{0}}{\left|x_{0}\right|}$ to a orthonormal basis $(e_{i})_{i=1}^{d}$ of $\IR^{d}$. For a neighbourhood $U$ of $x_{0}$ small enough, the vectors $\left(\frac{x}{\left|x\right|},e_{2},\ldots,e_{d}\right)$, $x\in U$, are still a basis of $\IR^{d}$. Denote by $O_{U}\left(x\right):=\operatorname{GS}\left(\frac{x}{\left|x\right|},e_{2},\ldots,e_{d}\right)\in O\left(d\right)$ the Gram-Schmidt procedure applied to this basis. The map $O_{U}$ is smooth on $U$ and thus for $x\in U$, using the symmetry of convolution,
    \begin{align}\label{radialconvolutionlemeq1}
        &\int_{0}^{\infty}\int_{O_{d}}\psi^{1}_{n}\left(\log r\right)g_{n}\left(Y\right)f\left(r^{-1} Y^{\ast}x\right)\Id m\left(Y\right)\frac{\Id r}{r}\nonumber\\
        =&\int_{0}^{\infty}\int_{O(d)}\psi^{1}_{n}\left(\log\frac{\left|x\right|}{r}\right)g_{n}\left(O_{U}\left(x\right)Y^{\ast}\right)f\left(rY_{1}\right)\Id m\left(Y\right)\frac{\Id r}{r},
    \end{align}
    where $Y_{1}$ denotes the first column of $Y$. Thus (\ref{radialconvolutionlemeq1}) is smooth in $x$ on $U$. Since $1-\phi$ cuts away $0$, we conclude that $f_{n}$ is smooth, because all other components of (\ref{radialconvolutionlemeq0}) are clearly smooth. Since $\psi_{n}^{d}$ and $\phi$ are compactly supported, $1-\phi$ is supported away from $0$ and $\psi_{n}^{1}$ is again compactly supported, one easily concludes by standard estimates on the integrals, that all derivatives of $f_{n}$ are bounded. It is also apparent that $f_{n}$, and $\nabla f_{n}$ converge point-wise to $f$ respectively $\nabla f$, since for each point, due to the compact support of the Dirac sequences, the integrals in (\ref{radialconvolutionlemeq0}) can be changed to take place over compact sets, on which $f$ and $\nabla f$ are uniformly continuous. It remains to check that there is a constant $C$ such that $\left\|f_{n}\right\|_{\IF_{\epsilon}}\leq C$. Clearly $\left|f_{n}\left(x\right)\right|\leq\left\|f\right\|_{L^{\infty}\left(\IR^{d}\right)}$ for all $x\in\IR^{d}$. Since $\phi$ is compactly supported it is also clear that there is a constant $C$ only dependent on $\phi$ and $s\geq 0$ such that
    \begin{align}
        \left\|\langle X\rangle^{s}\nabla\left(\phi\cdot\left(\psi^{d}_{n}\ast_{\IR^{d}}f\right)\right)\right\|_{L^{\infty}\left(\IR^{d}\right)}\leq C\left(\left\|f\right\|_{L^{\infty}\left(\IR^{d}\right)}+\left\|\nabla f\right\|_{L^{\infty}\left(\IR^{d}\right)}\right).
    \end{align}
    For the second summand of (\ref{radialconvolutionlemeq0}) note that if a derivative is applied to $1-\phi$ we obtain a compactly supported function, hence a similar arbitrary decay result as for the first summand shown above. Differentiation of $f\left(r^{-1}Y^{\ast}x\right)$ yields
    \begin{align}
        \p_{x^{i}}f\left(r^{-1}Y^{\ast}x\right)=r^{-1}\langle\left(\nabla f\right)\left(Y^{\ast}\frac{x}{r}\right),Y_{i}\rangle,\ \p_{R}f\left(r^{-1}Y^{\ast}x\right)=r^{-1}\left(\p_{R}f\right)\left(r^{-1}Y^{\ast}x\right),
    \end{align}
    which implies for $f\in\IF_{\epsilon}$, and $x\in\supp\left(1-\phi\right)$,
    \begin{align}
        \left|\langle x\rangle\p_{x^{i}}\int_{0}^{\infty}\int_{O(d)}\psi_{n}^{1}\left(\log r\right)g_{n}\left(Y\right)f\left(r^{-1}Y^{\ast}x\right)\Id m\left(Y\right)\frac{\Id r}{r}\right|\leq C\left\|f\right\|_{\IF_{\epsilon}},
    \end{align}
    where $C$ only depends on $\phi$. Similarly there are constants $C', C''$ independent of $f$ and $n$, such that for $f\in\IF_{\epsilon}$, and $x\in\supp\left(1-\phi\right)$,
    \begin{align}
        &\left|\langle x\rangle^{1+\epsilon}\p_{R}\int_{0}^{\infty}\int_{O(d)}\psi_{n}^{1}\left(\log r\right)g_{n}\left(Y\right)f\left(r^{-1}Y^{\ast}x\right)\Id m\left(Y\right)\frac{\Id r}{r}\right|\leq C'\left\|f\right\|_{\IF_{\epsilon}}\int_{0}^{\infty}\psi_{n}^{1}\left(\log r\right)r^{\epsilon}\frac{\Id r}{r}\nonumber\\
        =&C'\left\|f\right\|_{\IF_{\epsilon}}\int_{\IR}\psi^{1}\left(s\right)e^{n^{-1}\epsilon s}\Id s\leq C''\left\|f\right\|_{\IF_{\epsilon}}.
    \end{align}
    We conclude that $\left\|f_{n}\right\|_{\IF_{\epsilon}}$ is uniformly bounded.
\end{proof}

We now arrive at the central result of this chapter. The operator family $B$ is approximated by a sequence of finite rank valued operator families with improved regularity. In spirit, this statement is close to the approximation argument used for example in \cite{Push} and \cite{GLMST} in the one-dimensional case.

\begin{Propandef}\label{approxpropandef}
	Denote $P_{n}:=\one_{\left[-n,n\right]}\left(A_{0}\right)$, $n\in\IN$. Let $\left(\phi_{n}\right)_{n\in\IN}\subset\mathcal{D}$ be an orthonormal basis of $H$, and let $K_{n}$ be the finite rank projection onto $\mathrm{span}\left\{\left(\phi_{j}\right)_{j=1}^{n}\right\}$ in $H$ for $n\in\IN$. Identify the range of $K_{n}$ with $\IC^{n}$ and let $\delta_{l}$ be the mollification operator defined in Definition \ref{mollifierdef} with a fixed choice of cut-off functions and delta sequences. 
    
    Define for $B$ satisfying Hypothesis \ref{hyp1},
	\begin{align}
		B_{n}\left(x\right):=P_{n}B\left(x\right)P_{n},\ B_{n,m}\left(x\right):=K_{m}B_{n}\left(x\right)K_{m},\ B_{n,m,l}:=\delta_{l}\left(B_{n,m}\right),\ n,m,l\in\IN.
	\end{align}
	Then there are constants $C$, $C_{n}$, and $C_{n,m}$, where $C$ is universal, $C_{n}$ is independent of $m$ and $C_{n,m}$ is independent of $l$, such that
	\begin{align}\label{approxpropandefeq2}
        \lim_{n\to\infty}\tr_{\IC^{r}}\left(R_{z}\left(D^{\ast}_{B_{n}}D_{B_{n}}\right)^{N}-R_{z}\left(D_{B_{n}}D_{B_{n}}^{\ast}\right)^{N}\right)=\tr_{\IC^{r}}\left(R_{z}\left(D^{\ast}_{B}D_{B}\right)^{N}-R_{z}\left(D_{B}D_{B}^{\ast}\right)^{N}\right),
    \end{align}
       for $z\in\mathcal{D}_{C}\left(B\right)$,
       \begin{align}
        &\lim_{m\to\infty}\tr_{\IC^{r}}\left(R_{z}\left(D^{\ast}_{B_{n,m}}D_{B_{n,m}}\right)^{N}-R_{z}\left(D_{B_{n,m}}D_{B_{n,m}}^{\ast}\right)^{N}\right)\nonumber\\
        &=\tr_{\IC^{r}}\left(R_{z}\left(D^{\ast}_{B_{n}}D_{B_{n}}\right)^{N}-R_{z}\left(D_{B_{n}}D_{B_{n}}^{\ast}\right)^{N}\right),
        \end{align}
        for $\operatorname{dist}\left(z,\left(-\infty,0\right]\right)\geq C_{n}$, and
        \begin{align}
        &\lim_{l\to\infty}\tr_{\IC^{r}}\left(R_{z}\left(D^{\ast}_{B_{n,m,l}}D_{B_{n,m,l}}\right)^{N}-R_{z}\left(D_{B_{n,m,l}}D_{B_{n,m,l}}^{\ast}\right)^{N}\right)\nonumber\\
        &=\tr_{\IC^{r}}\left(R_{z}\left(D^{\ast}_{B_{n,m}}D_{B_{n,m}}\right)^{N}-R_{z}\left(D_{B_{n,m}}D_{B_{n,m}}^{\ast}\right)^{N}\right),
        \end{align}
        for $\operatorname{dist}\left(z,\left(-\infty,0\right]\right)\geq C_{n,m}$, are limits in $S^1\left(L^2\left(\IR^{d},\IC^{r}\otimes H\right)\right)$.
\end{Propandef}

\begin{proof}
	We start with
	\begin{align}\label{approxpropandefeq3}
		\lim_{n\to\infty}\tr_{\IC^{r}}\left(R_{z}\left(D^{\ast}_{B_{n}}D_{B_{n}}\right)^{N}-R_{z}\left(D_{B_{n}}D_{B_{n}}^{\ast}\right)^{N}\right)=\tr_{\IC^{r}}\left(R_{z}\left(D^{\ast}_{B}D_{B}\right)^{N}-R_{z}\left(D_{B}D_{B}^{\ast}\right)^{N}\right).
	\end{align}
	We expand both sides of (\ref{approxpropandefeq3}), according to the decomposition described in Proposition \ref{decomposelem}, into a finite sum of trace-class operators. Each summand consists (up to a constant scalar factor and possibly a constant bounded operator $\overline{\langle\di\rangle^{-1}\langle X\rangle^{k\xi}L\langle X\rangle^{(d-k)\xi}}$) of a product of operators with factors $u_{z}^{\bullet}\left(B_{n}\right)$, $V^{\bullet}\left(B_{n}\right)$, $\Ii\left(\di B_{n}\right)\langle M_{0}\rangle^{-1}$, and $\langle M_{0}\rangle \left(R_{z}\left(D_{B_{n}}^{\ast}D_{B_{n}}\right)^{k}+R_{z}\left(D_{B_{n}}D_{B_{n}}^{\ast}\right)^{k}\right)$ with upper indices independent of $n$ and $B$ (suppressed here by $\bullet$), respectively the same factors with $B_{n}$ replaced by $B$. We first note that, since $P_{n}$ commutes with $A_{0}$ and have norm $1$,
	\begin{align}
		\sup_{n\in\IN}\left(\rho_{z}^{-2N}\left(B_{n}\right)+\rho^{2N}_{z}\left(B_{n}\right)\right)\leq\rho_{z}^{-2N}\left(B\right)+\rho^{2N}_{z}\left(B\right)\xrightarrow{\mathrm{dist}\left(z,\left(-\infty,0\right]\right)\to\infty}0,
	\end{align}
	and because $P_{n}$ converges strongly to $\one_{H}$ for $n\to\infty$, we also have for any $\beta\in\left[-2N,2N\right]$, and a.e. $x\in\IR^{d}$, strong operator convergences in $H$ of
	\begin{align}
		\langle A_{0}\rangle^{\beta}\left(B_{n}\left(x\right)+A_{0}\right)^{2}\langle A_{0}\rangle^{-2-\beta}\xrightarrow{n\to\infty}\langle A_{0}\rangle^{\beta}\left(B\left(x\right)+A_{0}\right)^{2}\langle A_{0}\rangle^{-2-\beta},
	\end{align}
	and of
	\begin{align}
		\langle A_{0}\rangle^{\beta}\p_{i} B_{n}\left(x\right)\langle A_{0}\rangle^{-2-\beta}\xrightarrow{n\to\infty}\langle A_{0}\rangle^{\beta}\p_{i} B\left(x\right)\langle A_{0}\rangle^{-2-\beta},\ i\in\left\{1,\ldots,d\right\}.
	\end{align}
	The prerequisites of Lemma \ref{sotlem} are satisfied, such that we have SOTA convergences of of all factors $\overline{u_{z}^{\bullet}\left(B_{n}\right)}$ to $\overline{u_{z}^{\bullet}\left(B\right)}$, and of 
    \begin{align}
        \langle M_{0}\rangle \left(R_{z}\left(D_{B_{n}}^{\ast}D_{B_{n}}\right)^{k}+R_{z}\left(D_{B_{n}}D_{B_{n}}^{\ast}\right)^{k}\right)\xrightarrow{n\to\infty}\langle M_{0}\rangle \left(R_{z}\left(D_{B}^{\ast}D_{B}\right)^{k}+R_{z}\left(D_{B}D_{B}^{\ast}\right)^{k}\right).
    \end{align}
    Clearly the same argument also implies SOTA convergence of $\Ii\left(\di B_{n}\right)\langle M_{0}\rangle^{-1}$ to $\Ii\left(\di B\right)\langle M_{0}\rangle^{-1}$. It remains to discuss the $S^{d}\left(L^2\left(\IR^{d},\IC^{r}\otimes H\right)\right)$- respectively $S^{d+1}\left(L^2\left(\IR^{d},\IC^{r}\otimes H\right)\right)$-valued operators $\overline{V^{\bullet}\left(B_{n}\right)}$. Since $P_{n}$ commutes with $A_{0}$ and is constant on $\IR^{d}$ we have $V^{\bullet}\left(B_{n}\right)=\widetilde{P}_{n}V^{\bullet}\left(B\right)\widetilde{P}_{n}$ with $\widetilde{P}_{n}:=\one_{L^2\left(\IR^{d},\IC^{r}\right)}\otimes P_{n}$ which converges in SOTA to $\one_{L^2\left(\IR^{d},\IC^{r}\otimes H\right)}$. Lemma \ref{traceconvlem} thus implies that $\overline{V^{\bullet}\left(B_{n}\right)}$ converges in $S^{d}\left(L^2\left(\IR^{d},\IC^{r}\otimes H\right)\right)$ respectively $S^{d+1}\left(L^2\left(\IR^{d},\IC^{r}\otimes H\right)\right)$ to $\overline{V^{\bullet}\left(B\right)}$. We conclude that (\ref{approxpropandefeq3}) converges in trace-class by Lemma \ref{traceconvlem} for $z\in\mathcal{D}_{C}\left(B\right)$ for some universal constant $C$.
	
	For the second step we fix $n\in\IN$, and consider the statement
	\begin{align}\label{approxpropandefeq5}
		&\lim_{m\to\infty}\tr_{\IC^{r}}\left(R_{z}\left(D^{\ast}_{B_{n,m}}D_{B_{n,m}}\right)^{N}-R_{z}\left(D_{B_{n,m}}D_{B_{n,m}}^{\ast}\right)^{N}\right)\nonumber\\
		=&\tr_{\IC^{r}}\left(R_{z}\left(D^{\ast}_{B_{n}}D_{B_{n}}\right)^{N}-R_{z}\left(D_{B_{n}}D_{B_{n}}^{\ast}\right)^{N}\right).
	\end{align}
	Note that we may not proceed analogously to the first step, since $K_{m}$, in contrast to $P_{n}$, does not necessarily commute with $A_{0}$. However, since $\langle M_{0}\rangle^{s}P_{n}$ is a bounded operator for all $s\in\IR$, $n\in\IN$, we note for $B_{n}$ the stronger condition below holds.
	\begin{align}\label{approxpropandefeq6}
		\tau^{0,0,\nabla,d,1}\left(B_{n}\right)&<\infty,\nonumber\\
		\exists\epsilon>0:\ \tau^{0,0,\p_{R},d,1+\epsilon}\left(B_{n}\right)&<\infty.
	\end{align}
	This allows us to simplify the decomposition in Proposition \ref{decomposelem} to set $\alpha=0$ and $\beta=0$ in all appearing indices. In particular no factor of the decomposition contains a function of $M_{0}$. We suppress the new choice of indices by denoting them with $\widetilde{\bullet}$. The same decomposition can be deployed for $B_{n,m}$ instead of $B_{n}$. Moreover there exists $\eta>0$ uniformly in $m$, and a $m$ independent constant $C$ such that
    \begin{align}
        \tau^{0,0,\nabla,d,1}\left(B_{n,m}\right)+\tau^{0,0,\p_{R},d,1+\eta}\left(B_{n,m}\right)\leq C<\infty.
    \end{align}
    Since $K_{m}$ is constant on $\IR^{d}$, we have $V^{\widetilde{\bullet}}\left(B_{n,m}\right)=\widetilde{K_{m}}V^{\widetilde{\bullet}}\left(B_{n}\right)\widetilde{K_{m}}$, where the self-adjoint $\widetilde{K_{m}}:=\one_{L^2\left(\IR^{d},\IC^{r}\right)}\otimes K_{m}$ converge in strong operator topology to $\one_{L^2\left(\IR^{d},\IC^{r}\otimes H\right)}$. Thus Lemma \ref{traceconvlem} and Lemma \ref{lightschattlem} imply that $V^{\widetilde{\bullet}}\left(B_{n,m}\right)$ converges in $d$th respectively $(d+1)$th Schatten-von Neumann norm to $V^{\widetilde{\bullet}}\left(B_{n}\right)$ as $m\to\infty$.
    Since $\left\|B_{n,m}\right\|_{C^{1}_{b}\left(\IR^{d},B\left(H\right)\right)}$, $\left\|B_{n}\right\|_{C^{1}_{b}\left(\IR^{d},B\left(H\right)\right)}$ are uniformly bounded in $m$, Lemma \ref{lightsotlem} implies that all factors $\overline{u_{z}^{\widetilde{\bullet}}\left(B_{n,m}\right)}$ converge in strong operator topology to $\overline{u_{z}^{\widetilde{\bullet}}\left(B_{n}\right)}$, and that the last factor of type $\overline{u_{z}^{\widetilde{\bullet}}\left(B_{n,m}\right)}$ in each product of the decomposition converges even in SOTA (because they do not involve $\di$). Similarly also $R_{z}\left(D_{B_{n,m}}^{\ast}D_{B_{n,m}}\right)^{k}+R_{z}\left(D_{B_{n,m}}D_{B_{n,m}}^{\ast}\right)^{k}$ converges in SOTA to $R_{z}\left(D_{B_{n}}^{\ast}D_{B_{n}}\right)^{k}+R_{z}\left(D_{B_{n}}D_{B_{n}}^{\ast}\right)^{k}$. The domain of convergence for the parameter $z$ is uniformly given by $w=\operatorname{dist}\left(z,\left(-\infty,0\right]\right)\geq C$, for some constant $C$ independent of $m$, according to Lemma \ref{lightboundedlem}, since $\left\|B_{n,m}\right\|_{C^{1}_{b}\left(\IR^{d},B\left(H\right)\right)}$ is uniformly bounded in $m$. Clearly also $\Ii\left(\di B_{n,m}\right)$ converges in SOTA to $\Ii\di^{E}B_{n}$. We conclude that (\ref{approxpropandefeq5}) converges in trace class for $w\geq C$ for some $m$ independent constant $C$ by Lemma \ref{traceconvlem}.

    For the third step fix both $n,m\in\IN$, and consider the statement
	\begin{align}\label{approxpropandefeq7}
		&\lim_{l\to\infty}\tr_{\IC^{r}}\left(R_{z}\left(D^{\ast}_{B_{n,m,l}}D_{B_{n,m,l}}\right)^{N}-R_{z}\left(D_{B_{n,m,l}}D_{B_{n,m,l}}^{\ast}\right)^{N}\right)\nonumber\\
		=&\tr_{\IC^{r}}\left(R_{z}\left(D^{\ast}_{B_{n,m}}D_{B_{n,m}}\right)^{N}-R_{z}\left(D_{B_{n,m}}D_{B_{n,m}}^{\ast}\right)^{N}\right).
	\end{align}
    First note that $B_{n,m}\left(x\right)=K_{m}B_{n,m}\left(x\right)K_{m}$, and thus also $B_{n,m,l}\left(x\right)=K_{m}B_{n,m,l}\left(x\right)K_{m}$, and $\nabla B_{n,m,l}\left(x\right)=K_{m}\nabla B_{n,m,l}\left(x\right)K_{m}$. In particular, if we identify $\rg\left(K_{m}\right)\simeq\IC^{m}$, there is some $l$ independent $\eta>0$ such that
    \begin{align}
        \tau^{0,0,\nabla,d,1}\left(B_{n,m,l}\right)+\tau^{0,0,\p_{R},d,1+\eta}\left(B_{n,m,l}\right)\leq\left\|B_{n,m,l}\right\|_{\mathcal{F}_{\epsilon}}\leq C<\infty.
    \end{align}
    Additionally $B_{n,m,l}$, and $\nabla B_{n,m,l}$ converge point-wise to $B_{n,m}$, and $\nabla B_{n,m}$ respectively by Lemma \ref{radialconvolutionlem}. We use the same decomposition with $\alpha=0$ and $\beta=0$ from the second step, then Lemma \ref{lightsotlem} implies that all factors $\overline{u_{z}^{\widetilde{\bullet}}\left(B_{n,m,l}\right)}$ converge in strong operator topology to $\overline{u_{z}^{\widetilde{\bullet}}\left(B_{n,m}\right)}$, and that the last factor of type $\overline{u_{z}^{\widetilde{\bullet}}\left(B_{n,m,l}\right)}$ in each product of the decomposition converges even in SOTA. Similarly also $R_{z}\left(D_{B_{n,m,l}}^{\ast}D_{B_{n,m,l}}\right)^{k}+R_{z}\left(D_{B_{n,m,l}}D_{B_{n,m,l}}^{\ast}\right)^{k}$ converges in SOTA to $R_{z}\left(D_{B_{n,m}}^{\ast}D_{B_{n,m}}\right)^{k}+R_{z}\left(D_{B_{n,m}}D_{B_{n,m}}^{\ast}\right)^{k}$. Clearly also $\Ii\left(\di B_{n,m,l}\right)$ converges in SOTA to $\Ii\left(\di B_{n,m}\right)$. The domain of convergence for the parameter $z$ is uniformly given by $\operatorname{dist}\left(z,\left(-\infty,0\right]\right)\geq C$, for some constant $C$ independent of $l$, according to Lemma \ref{lightboundedlem}, since $\left\|B_{n,m,l}\right\|_{C^{1}_{b}\left(\IR^{d},B\left(H\right)\right)}\leq\left\|B_{n,m,l}\right\|_{\mathcal{F}_{\epsilon}}$ is uniformly bounded in $l$ by Lemma \ref{radialconvolutionlem}.
    Because $K_{m}$ is finite rank, and $\nabla B_{n,m,l}$, respectively $\p_{R} B_{n,m,l}$ can be bounded in $S^{d}$-norm by multiples of $\langle X\rangle^{-1}$ respectively of $\langle X\rangle^{-1-\epsilon}$ uniformly in $l$, the dominated convergence theorem and Corollary \ref{simoninterpolcor} imply that the factors $V^{\widetilde{\bullet}}\left(B_{n,m,l}\right)$ converge in $S^{d}\left(L^2\left(\IR^{d},\IC^{r}\otimes H\right)\right)$ respectively $S^{d+1}\left(L^2\left(\IR^{d},\IC^{r}\otimes H\right)\right)$. Together we conclude that (\ref{approxpropandefeq7}) converges in trace class for $w\geq C$ for some $l$ independent constant $C$ by Lemma \ref{traceconvlem}.
\end{proof}

A simple consequence of holomorphic functional calculus and the possibility of choosing uniform integration paths, allows us to conclude the approximation of the principal trace in its exponentiated form.

\begin{Corollary}\label{expapproxcor}
	Let $B$ satisfy Hypothesis \ref{hyp1}, and let $t>0$. Then the iterated limits
	\begin{align}
		\lim_{n\to\infty}\lim_{m\to\infty}\lim_{l\to\infty}\tr_{\IC^{r}}\left(e^{-tD_{B_{n,m,l}}^{\ast}D_{B_{n,m,l}}}-e^{-tD_{B_{n,m,l}}D^{\ast}_{B_{n,m,l}}}\right)=\tr_{\IC^{r}}\left(e^{-tD_{B}^{\ast}D_{B}}-e^{-tD_{B}D^{\ast}_{B}}\right),
	\end{align}
	are limits in $S^1\left(L^2\left(\IR^{d},H\right)\right)$.
\end{Corollary}

\begin{proof}
	According to Proposition and Definition \ref{approxpropandef}, the trace norm bound from Proposition \ref{decomposelem} holds uniformly in $n\in\IN$ for $B_{n}$ instead of $B$, on the domain $\operatorname{dist}\left(z,\left(-\infty,0\right]\right)\geq C$ for $C$ independent of $n$. So we may replace $D_{B}$, and $D_{B}^{\ast}$ wit $D_{B_{n}}$ and $D^{\ast}_{B_{n}}$ in equation (\ref{exptracepropeq1}) with the same integrable dominant (\ref{exptracepropeq2}), and the dominated convergence theorem allows us to pass the limit in trace norm past the integral in (\ref{exptracepropeq1}). In a similar fashion we argue for the two iterated limits in $m$ respectively $l$, where we might have to modify the integration curve dependent on $n$ for the limit in $m$ (respectively dependent on $n,m$ for the limit in $l$) in the proof of Proposition \ref{exptraceprop} accordingly. This is possible due to statement in Proposition and Definition \ref{approxpropandef} pertaining to the domains of convergence for the parameter $z$.
\end{proof}

We present some properties of the approximating sequence $B_{n,m,l}$ in the upcoming Remark, which closes this chapter on the approximation scheme.

\begin{Definition}
	For $l,s\in\IN$, $p\in\left[1,\infty\right]$, and a Banach space $X$, denote by $C^{l,p}\left(\IR^{s},X\right)$ the space of $l$-times continuously differentiable functions with derivatives in $L^{p}$ over $\IR^{d}$ with values in $X$. For $l,m,n,s\in\IN_{0}$ with $m\geq n$, denote the Banach space,
	\begin{align}
		C^{l}_{m,n}\left(\IR^{s}\right):=C^{l,\infty}\left(\IR^{s},B\left(\IC^{r}\right)\otimes\langle A_{0}\rangle^{-m}B\left(H\right)\langle A_{0}\rangle^{n}\right),
	\end{align}
	with the norm
	\begin{align}
		\left\|f\right\|_{l,m,n}:=\sum_{\alpha\in\IN^{s}_{0},\left|\alpha\right|\leq l}\left\|\p^{\alpha}f\right\|_{L^{\infty}\left(\IR^{d},B\left(\IC^{r}\right)\otimes\langle A_{0}\rangle^{-m}B\left(H\right)\langle A_{0}\rangle^{n}\right)},\ f\in C^{l}_{m,n}\left(\IR^{s}\right).
	\end{align}
	Let $C^{l,S^1}_{m,n}\left(\IR^{s}\right)$ denote the Banach space $C^{l,p}\left(\IR^{s},B\left(\IC^{r}\right)\otimes\langle A_{0}\rangle^{-m}S^1\left(H\right)\langle A_{0}\rangle^{n}\right)$ with the norm
	\begin{align}
		\left\|f\right\|_{l,m,n}^{S^1}:=\sum_{\alpha\in\IN^{s}_{0},\left|\alpha\right|\leq l}\left\|\p^{\alpha}f\right\|_{L^{\infty}\left(\IR^{d},B\left(\IC^{r}\right)\otimes\langle A_{0}\rangle^{-m}S^{1}\left(H\right)\langle A_{0}\rangle^{n}\right)},\ f\in C^{l,S^{1}}_{m,n}\left(\IR^{s}\right).
	\end{align}
\end{Definition}

\begin{Remark}\label{limitrem}
	Assume that $B$ satisfies Hypothesis \ref{hyp1}.
	\begin{enumerate}
		\item By (\ref{radlimeq}) for $\phi\in\dom\ A_{0}$, a.e. $y\in S_{1}\left(0\right)$, a.e. $x\in\IR^{d}$, and $\gamma\in\IN_{0}^{d}$ with $\left|\gamma\right|\leq 1$, we have
	\begin{align}
		\lim_{R\to\infty}R^{\left|\gamma\right|}\left\|\left(\p^{\gamma}B\left(Ry+x\right)-\p^{\gamma}B\left(Ry\right)\right)\phi\right\|_{H}&=0. 
	\end{align}
	Then for a.e. $y\in S_{1}\left(0\right)$, a.e. $x\in\IR^{d}$, and $\gamma\in\IN_{0}^{d}$ with $\left|\gamma\right|\leq 1$, we have
	\begin{align}
		\lim_{R\to\infty}R^{\left|\gamma\right|}\left\|\p^{\gamma}B_{n,m,l}\left(Ry+x\right)-\p^{\gamma}B_{n,m,l}\left(Ry\right)\right\|_{\langle A_{0}\rangle^{-i}S^{1}\left(H\right)\langle A_{0}\rangle^{j}}&=0.
	\end{align}
	for all $n,m,l\in\IN$, $i,j\in\IN_{0}$.
	\item For all $n,m,l\in\IN$, $i,j\in\IN_{0}$, we have
	\begin{align}
		\left(B_{n,m,l}+M_{0}\right)^{2}-M_{0}^{2}, \ \Ii\left(\di B_{n,m,l}\right)\in C^{\infty,S^1}_{i,j}\left(\IR^{d}\right).
	\end{align}
	\end{enumerate}
The projection $P_{n}$ commutes with $A_{0}$ and $K_{m}$ is finite-rank, which maps into any $\dom\ A_{0}^{i}$, $i\in\IN_{0}$, thus the first statement holds for $B_{n,m}$ instead of $B_{n,m,l}$. But then $B_{n,m,l}=\delta_{l}\left(B_{n,m}\right)$, which together with the dominated convergence theorem (amenable since $\left\|B_{n,m}\right\|_{\mathcal{F}_{\epsilon}}<\infty$ for some $\epsilon>0$) yields the first statement. The second statement follows again since $P_{n}$ commutes with $A_{0}$ and $K_{m}$ is finite-rank, mapping into any $\dom\ A_{0}^{i}$, $i\in\IN_{0}$, and because the range of $\delta_{l}$ is smooth by Lemma \ref{radialconvolutionlem}.
\end{Remark}

\section{Construction of auxiliary integral kernels related to the semigroups of $D^{\ast}D$ and $DD^{\ast}$}

The overarching goal of this and the following chapters is to calculate the principal trace. Since we assume Hypothesis \ref{hyp1} for $B$, we may facilitate this task, by replacing $B$ by $B_{n,m,l}$ (and $D_{B}$, $D^{\ast}_{B}$ by $D_{B_{n,m,l}}$, $D_{B_{n,m,l}}^{\ast}$ accordingly), Corollary \ref{expapproxcor} enables us to then take the iterated limits in the end to obtain the desired result. For convenience of notation, in the bulk of this and the following chapters we drop the indices $n,m,l$, which will be reintroduced at the end.

Let us first introduce some short hand notations.
\begin{Definition}
    Denote by $\Delta_{l}\subset\left[0,1\right]^{l}$ the $l$-simplex
\begin{align}
	\Delta_{l}=\left\{w\in\left[0,1\right]^{l}\left|w_{l}\leq w_{l-1}\leq\ldots\leq w_{1}\right.\right\}.
\end{align}
We will frequently use the following parametrization of $\Delta_{l}$ by $l+1$ coordinates:
\begin{align}
	s_{0}=1-w_{1},\ s_{1}=w_{1}-w_{2},\ldots, s_{l-1}=w_{l-1}-w_{l},\ s_{l}=w_{l},
\end{align}
which satisfies $s_{j}\in\left[0,1\right]$, and $\sum_{j=0}^{l}s_{j}=1$.
Note that $\vol\left(\Delta_{l}\right)=\frac{1}{l!}$. Denote for a Hilbert space $Y$, a self-adjoint operator $T$ in $Y$, bounded below, and $s\in\left(\IR^{\geq 0}\right)^{l}$
\begin{align}
	e^{-sT}:=\left(e^{-s_{1}T},\ldots,e^{-s_{l}T}\right)\in B\left(Y\right)^{l},
\end{align}
and
\begin{align}
	T^{\otimes l}:=\left(T,\ldots, T\right)\in B\left(Y\right)^{l}.
\end{align}

Denote by $L^{0}\left(X,B\right)$ the measurable functions on a measure space $X$ with values in a Banach space $B$. For $M\in L^{0}\left(X,B\right)$ denote
\begin{align}
	M^{\otimes l}:=\left(M,\ldots, M\right)\in\left(L^{0}\left(X,B\right)\right)^{l},
\end{align}
and for $s\in\left(\IR^{\geq 0}\right)^{l}$, $\IR^{\geq{0}}\ni t\mapsto N_{t}\in L^{0}\left(X,B\right)$,
\begin{align}
	N_{s}:=\left(N_{s_{1}},\ldots,N_{s_{l}}\right)\in\left(L^{0}\left(X,B\right)\right)^{l}.
\end{align}
Define the convolution shuffle product $\shuffle_{\ast}$ between $\left(L^{0}\left(\IR^{d}\times\IR^{d},B\left(Y\right)\right)\right)^{l+1}$, and\\
$\left(L^{0}\left(\IR^{d},B\left(Y\right)\right)\right)^{l}$ with values in $L^{0}\left(\IR^{d}\times\IR^{d},B\left(Y\right)\right)$ by
\begin{align}
	&\left(q\shuffle_{\ast}M\right)\left(x,y\right):=\int_{u\in\left(\IR^{d}\right)^{l}}q_{0}\left(x,u_{1}\right)M_{1}\left(u_{1}\right)q_{1}\left(u_{1},u_{2}\right)\ldots M_{l}\left(u_{l}\right)q_{l}\left(u_{l},y\right)\Id u\in B\left(Y\right),\nonumber\\
	&q=\left(q_{0},\ldots,q_{l}\right)\in \left(L^{0}\left(\IR^{d}\times\IR^{d},B\left(Y\right)\right)\right)^{l+1},\ M=\left(M_{1},\ldots,M_{l}\right)\in\left(L^{0}\left(\IR^{d},B\left(Y\right)\right)\right)^{l},
\end{align}
whenever the integral converges absolutely in $B\left(Y\right)$-norm for a.e. $\left(x,y\right)\in\IR^{d}\times\IR^{d}$.
\end{Definition}

	We want to calculate the trace of an operator in $L^2\left(\IR^{d},H\right)$, and thus it is helpful to find a $B\left(H\right)$-valued integral kernel, and integrate its trace in $H$ over the diagonal, which is intuitive, considering the scalar case covered by the well-known theorem by Mercer. A Hilbert space version is provided in the appendix of \cite{BruSee}, albeit only for $\IR^{d}$, $d=1$. However for general $d\in\IN$ the original proof works with the necessary simple changes (which we omit here).
	
	\begin{Theorem}[\cite{BruSee}]\label{bruseethm}
		Let $Y$ be a separable Hilbert space, and $T\in S^1\left(L^2\left(\IR^{d},Y\right)\right)$. Then there exists a unique $S^{1}\left(Y\right)$-valued measurable integral kernel $t$ of $T$, such that
		\begin{align}
			\left(Tf\right)\left(x\right)=\int_{\IR^{d}}t\left(x,y\right)f\left(y\right)\Id y,\ f\in L^2\left(\IR^{d},Y\right),
		\end{align}
		and
		\begin{align}
			h\mapsto\left(x\mapsto t\left(x,x+h\right)\right),
		\end{align}
		is a bounded continuous map of $\IR^{d}$ to $L^1\left(\IR^{d},S^1\left(Y\right)\right)$. Moreover
		\begin{align}
			\int_{\IR^{d}}\left\|t\left(x,x\right)\right\|_{S^1\left(Y\right)}\Id x&\leq\left\|T\right\|_{S^1\left(L^2\left(\IR^{d},Y\right)\right)},\nonumber\\
			\int_{\IR^{d}}\tr_{Y}\left(t\left(x,x\right)\right)\Id x&=\tr_{L^2\left(\IR^{d},Y\right)}\left(T\right).
		\end{align}
	\end{Theorem}
	
	\begin{Remark}\label{bruseerem}
		We note that the variational principle and the separability of $Y$ imply that if we find a jointly continuous $B\left(Y\right)$-valued integral kernel $\tilde{t}$ of $T$, by the above stated uniqueness, we have already $t=\tilde{t}$ as elements of $C_{b}\left(\IR^{d},L^1\left(\IR^{d},S^1\left(Y\right)\right)\right)$. Therefore if we want to calculate the trace of $T$, it suffices to integrate the $Y$-trace of the diagonal of the continuous kernel of $T$, if it exists.
	\end{Remark}
	
	The next goal therefore is the construction of a continuous, $B\left(H\right)$-valued integral kernel of $\tr_{\IC^{r}}\left(e^{-tD^{\ast}D}-e^{-tDD^{\ast}}\right)$. To that end we will use a classical Volterra series argument.
	
	\begin{Lemma}\label{volterralem}
		Let $T,S$ be self-adjoint operators in a Hilbert space $Y$, and assume that $S$ is bounded with norm $c$, and $T\geq 0$. Then for $t>0$,
		\begin{align}\label{eq15}
			e^{-t\left(T+S\right)}=e^{-tT}+\sum_{l=1}^{\infty}\left(-1\right)^{l}\int_{s\in t\Delta_{l}}e^{-sT}\shuffle S^{\otimes l}\Id s,
		\end{align}
		where sum and integral converge absolutely in operator norm.
	\end{Lemma}
	
	\begin{proof}
		Duhamel's principle implies for $t>0$,
		\begin{align}\label{eq14}
			e^{-t\left(T+S\right)}=e^{-tT}-\int_{0}^{t}e^{-\left(t-s\right)T}Se^{-s\left(T+S\right)}\Id s,
		\end{align}
		where the integral converges in operator norm. The iterative implementation of formula (\ref{eq14}) then implies (\ref{eq15}). It remains to show that (\ref{eq15}) converges absolutely in operator norm. We estimate
		\begin{align}
			\int_{s\in t\Delta_{l}}\left\|e^{-sT}\shuffle S^{\otimes l}\right\|_{B\left(Y\right)}\Id s\leq\frac{\left(ct\right)^{l}}{l!},
		\end{align}
		which is absolutely summable over $l\in\IN$.
	\end{proof}
	
	\begin{Definition}
		Let $t>0$, $k\in\IN$, define
		\begin{align}
			R_{t}^{k}:=\begin{cases}
				e^{-tH_{B}},& k=0,\\
				\int_{s\in t\Delta_{k}}e^{-sH_{B}}\shuffle\left(\Ii\di^{E}B\right)^{\otimes k}\Id s, & k\geq 1.
			\end{cases}
		\end{align}
	\end{Definition}
	
	Since $\Ii\di^{E}B$ is a bounded operator, the integrals defining $R_{t}^{k}$ converge absolutely in operator norm. We furthermore note that the norms of $R_{t}^{k}$ are absolutely summable over $k\in\IN$.
	
	\begin{Lemma}\label{haserieslem}
		Let $t>0$, then
		\begin{align}
			e^{-tDD^{\ast}}=&\sum_{l=0}^{\infty}\left(-1\right)^{l}R_{t}^{l},\nonumber\\
			e^{-tD^{\ast}D}=&\sum_{l=0}^{\infty}R_{t}^{l},
		\end{align}
		where the series converges absolutely in operator norm.
	\end{Lemma}
	
	\begin{proof}
		Apply Lemma \ref{volterralem} to $DD^{\ast}=H_{B}+\Ii\left(\di B\right)$, and $D^{\ast}D=H_{B}-\Ii\left(\di B\right)$.
	\end{proof}
	
	\begin{Definition}\label{kdef}
		Let $t>0$, $l\in\IN$. We define
		\begin{align}
			K_{t}^{-,l}:=&\sum_{j=\frac{l-1}{2}}^{\infty}R^{2j+1}_{t},\text{ $l$ odd},\ K_{t}^{+,l}:=\sum_{j=\frac{l}{2}}^{\infty}R^{2j}_{t},\text{ $l$ even},\nonumber\\
			k_{t}^{-,l}:=&\sum_{j=0}^{\frac{l-3}{2}}R^{2j+1}_{t},\text{ $l$ odd},\ k_{t}^{+,l}:=\sum_{j=1}^{\frac{l-2}{2}}R^{2j}_{t},\text{ $l$ even},
		\end{align}
		where the series converge in operator norm.
	\end{Definition}
	
	Lemma \ref{haserieslem} then allows to re-express the semi-group difference in terms of commutators.
	
	\begin{Lemma}\label{auxfunccalclem}
		Let $t>0$, then
		\begin{align}\label{auxfunccalclemeq1}
			&e^{-tD^{\ast}D}-e^{-tDD^{\ast}}=\int_{0}^{t}\frac{1}{2}\left(\overline{\left[D,D^{\ast}e^{-sDD^{\ast}}\right]}-\overline{\left[D^{\ast},De^{-sD^{\ast}D}\right]}\right)\Id s\nonumber\\
			=&\int_{0}^{t}\left(\overline{\left[\Ii\di,\Ii\di\frac{1}{2}\left(e^{-sD^{\ast}D}-e^{-sDD^{\ast}}\right)\right]}-\overline{\left[A,A\frac{1}{2}\left(e^{-sD^{\ast}D}-e^{-sDD^{\ast}}\right)\right]}\right.\nonumber\\
			&\left.+\overline{\left[\Ii\di,A\frac{1}{2}\left(e^{-sD^{\ast}D}+e^{-sDD^{\ast}}\right)\right]}-\overline{\left[A,\Ii\di\frac{1}{2}\left(e^{-sD^{\ast}D}+e^{-sDD^{\ast}}\right)\right]}\right)\Id s\nonumber\\
			=&\int_{0}^{t}\left(\overline{\left[\Ii\di,\Ii\di\left(K_{s}^{-,d-2}+k_{s}^{-,d-2}\right)\right]}-\overline{\left[A,A\left(K_{s}^{-,d}+k_{s}^{-,d}\right)\right]}\right.\nonumber\\
			&\left.+\overline{\left[\Ii\di,A\left(K_{s}^{+,d-1}+k_{s}^{+,d-1}+e^{-sH_{B}}\right)\right]}-\overline{\left[A,\Ii\di\left(K_{s}^{+,d-1}+k_{s}^{+,d-1}+e^{-sH_{B}}\right)\right]}\right)\Id s,
		\end{align}
		where the integrals converge strongly on $\dom\ H_{0}$.
	\end{Lemma}
	
	\begin{proof}
		Since $De^{-sD^{\ast}D}\supseteq e^{-sDD^{\ast}}D$, and $D^{\ast}e^{-sDD^{\ast}}\supseteq e^{-sD^{\ast}D}D^{\ast}$, $s>0$, we have
		\begin{align}
			\p_{s}\left(e^{-sD^{\ast}D}-e^{-sDD^{\ast}}\right)=DD^{\ast}e^{-sDD^{\ast}}-D^{\ast}De^{-sD^{\ast}D}\supseteq\frac{1}{2}\left(\left[D,D^{\ast}e^{-sDD^{\ast}}\right]-\left[D^{\ast},De^{-sD^{\ast}D}\right]\right),
		\end{align}
		which implies the first line. The second line follows by expanding $D=\Ii\di+A$, $D^{\ast}=-\Ii\di+A$, and the third line by Lemma \ref{haserieslem}.
	\end{proof}

	We exploit the fact that $H_{B}$ commutes with the Clifford matrices $c^{j}$, and the above expansion reduces if we apply the partial trace $\tr_{\IC^{r}}$.
	
	\begin{Lemma}\label{reduxlem}
		Let $t>0$, then
		\begin{align}\label{reduxlemeq1}
			\tr_{\IC^{r}}\left(e^{-tD^{\ast}D}-e^{-tDD^{\ast}}\right)=&\int_{0}^{t}\left(\tr_{\IC^{r}}\left(\left[\Ii\di,\Ii\di K_{s}^{-,d-2}+AK_{s}^{+,d-1}\right]\right)\right.\nonumber\\
			&\left.-\left[A,\tr_{\IC^{r}}\left(AK_{s}^{-,d}+\Ii\di K_{s}^{+,d-1}\right)\right]\right)\Id s,
		\end{align}
		where the right hand side converges in operator norm of $L^2\left(\IR^{d},H\right)$.
	\end{Lemma}
	
	\begin{proof}
		We remind ourselves that $\tr_{\IC^{r}}=\sum_{i=1}^{r}\langle\cdot e^{i},e^{i}\rangle_{\IC^{r}}$, for any orthonormal basis $\left(e^{i}\right)_{i=1}^{r}$, which in particular implies that $\tr_{\IC^{r}}$ is continuous from $B\left(L^2\left(\IR^{d},\IC^{r}\otimes H\right)\right)$ to $B\left(L^2\left(\IR^{d},H\right)\right)$, both equipped with the weak operator/the strong operator/the norm topology. We consider the last line of equation (\ref{auxfunccalclemeq1}) of Lemma \ref{auxfunccalclem}, and restrict the integrand to $\dom\ H_{0}=\IC^{r}\otimes X$, where $X=\left(H^{2}\left(\IR^{d},H\right)\cap L^2\left(\IR^{d},\dom\ A_{0}^2\right)\right)$ is dense in $L^2\left(\IR^{d},H\right)$. Since the integral converges strongly on $\dom\ H_{0}$ we may interchange $\tr_{\IC^{r}}$ with integration. Also we note that all summands in the last line of (\ref{auxfunccalclemeq1}), which involve $e^{-sH_{B}}$ or the operators $k_{s}^{\pm,j}$ contain at most $d-1$ Clifford matrices $c^{i}$ as factors. We consider Definition \ref{kdef}, and otherwise the factors $e^{-rH_{B}}$, $r>0$, and $A$, which commute with all Clifford matrices: We obtain analogously to Lemma \ref{cliflem}, that these terms vanish, if $\tr_{\IC^{r}}$ is applied. So (\ref{reduxlemeq1}) holds, at least on $X$. We may estimate for $k\geq d-2\geq 1$, and $R,S\in\left\{\Ii\di,A\right\}$, using the fact that $\Ii\left(\di B\right)$ is a $C^{1}_{b}$-family with bounded derivatives of finite rank operators with range in $\dom\ A_{0}$, and the defining integral of $R_{s}^{k}$,
		\begin{align}\label{reduxlemeq2}
			\left\|\overline{STR_{s}^{k}}\right\|_{B\left(L^2\left(\IR^{d},\IC^{r}\otimes H\right)\right)},\ \left\|\overline{SR_{s}^{k}T}\right\|_{B\left(L^2\left(\IR^{d},\IC^{r}\otimes H\right)\right)}\leq C'\frac{c^{k-1}s^{k-1}}{\left(k-1\right)!},
		\end{align}
		where $c:=\left\|\Ii\left(\di B\right)\right\|_{B\left(L^2\left(\IR^{d},\IC^{r}\otimes H\right)\right)}$, and $C'$ is a suitable constant. In particular the norms in (\ref{reduxlemeq2}) are absolutely summable over $k\geq d-2$, so we have
		\begin{align}\label{reduxlemeq3}
			\left\|\overline{STK_{s}^{j}}\right\|_{B\left(L^2\left(\IR^{d},\IC^{r}\otimes H\right)\right)},\ \left\|\overline{SK_{s}^{j}T}\right\|_{B\left(L^2\left(\IR^{d},\IC^{r}\otimes H\right)\right)}\leq C'\sum_{k=j-1}^{\infty}\frac{c^{k}s^{k}}{k!},
		\end{align}
		which is $O\left(s^{j-1}\right)$ for $s$ near $0$, and in particular we see that $\overline{STK_{s}^{j}}$, and $\overline{SK_{s}^{j}T}$ are Bochner integrable in operator norm on $s\in\left(0,t\right)$ for any $t>0$. Since $e^{-tD^{\ast}D}-e^{-tDD^{\ast}}$ is also bounded, and $\tr_{\IC^{r}}$ is a continuous operator, equation (\ref{reduxlemeq1}) holds by density of $X$ in $L^2\left(\IR^{d},H\right)$, and the right hand side of (\ref{reduxlemeq1}) converges in operator norm.
	\end{proof}

    We introduce the following building blocks of the desired integral kernel.
	
	\begin{Definition}\label{kerneldef}
		Let $t>0$. Denote for $j\geq 1$,
		\begin{align}
			q_{t}\left(x\right):=&\left(4\pi t\right)^{-\frac{d}{2}}e^{-\frac{\left|x\right|^{2}}{4t}},\ x\in\IR^{d},\ Q_{t}:=\one_{\IC^{r}}\otimes e^{-tA_{0}^2}\otimes q_{t},\nonumber\\
			\theta_{t}^{j}:=&\int_{s\in t\Delta_{j}}Q_{s}\shuffle_{\ast}\left(A^2-M_{0}^2\right)^{\otimes j}\Id s,\ \theta_{t}:=Q_{t}+\sum_{j=1}^{\infty}\left(-1\right)^{j}\theta_{t}^{j},\nonumber\\
			r_{t}^{j}:=&\int_{s\in t\Delta_{j}}\theta_{s}\shuffle_{\ast}\left(\Ii\di B\right)^{\otimes j}\Id s.
		\end{align}
	\end{Definition}

    \begin{Remark}
        In the following we will use the convention, that if a statement is made for $C^{l',\left(S^1\right)}_{m',n}$ it means that it holds both for $C^{l'}_{m',n}$ respectively $C^{l',S^1}_{m',n}$ consistently in place of all $C^{l',\left(S^1\right)}_{m',n}$ appearing throughout the the statement. If no other domain is specified, it will always be $\IR^{d}$.
    \end{Remark}

    In the following part we will successively construct integral kernels together with estimates, which we build up to arrive at the desired integral kernels needed to calculate the principal trace. The arguments in each step are quite similar, the techniques involved are classical (for example cf. Chapter 2 in \cite{BerGetVer}).
	
	\begin{Lemma}\label{qproplem}
		Let $l,l',m,m',n,n'\in\IN_{0}$ with $m'\geq m\geq n\geq n'$, and $l'\geq l$. For $t>0$ define the linear operators $Q_{t}^{L/R}$, given by
		\begin{align}
			\left(Q_{t}^{L}f\right)\left(x\right):=\int_{\IR^{d}}Q_{t}\left(x,y\right)f\left(y\right)\Id y,\ \left(Q_{t}^{R}f\right)\left(x\right):=\int_{\IR^{d}}f\left(y\right)Q_{t}\left(x,y\right)\Id y,\ f\in C^{l,\left(S^{1}\right)}_{m,n}\left(\IR^{d}\right).
		\end{align}
		Then $Q_{t}^{L}$ is continuous as a map to $C^{l',\left(S^1\right)}_{m',n}$, and $Q_{t}^{R}$ is continuous as a map to $C^{l',\left(S^1\right)}_{m,n'}$, and we have a $t$-independent constants $C$, such that
		\begin{align}\label{qproplemeq1}
			\left\|Q_{t}^{L}\right\|_{B\left(C^{l,\left(S^1\right)}_{m,n},C^{l',\left(S^1\right)}_{m',n}\right)}\leq Ct^{-\frac{l'-l+m'-m}{2}},\ \left\|Q_{t}^{R}\right\|_{B\left(C^{l,\left(S^1\right)}_{m,n'},C^{l',\left(S^1\right)}_{m,n}\right)}\leq Ct^{-\frac{l'-l+n-n'}{2}}.
		\end{align}
	\end{Lemma}
	
	\begin{proof}
		Let $f\in C^{l,\left(S^1\right)}_{m,n}$, $\alpha\in\IN^{d}_{0}$ with $\left|\alpha\right|\leq l'$. For $\left|\alpha\right|\leq l$, let $\beta=0\in\IN_{0}^{d}$, otherwise let $\beta\in\IN_{0}^{d}$, such that $\alpha\geq\beta$, $\left|\beta\right|=l'-l$ and $\left|\alpha-\beta\right|=l$. Then
		\begin{align}
			\langle A_{0}\rangle^{m'}\p^{\alpha}\left(Q_{t}^{L}f\right)\left(x\right)\langle A_{0}\rangle^{-n}&=\int_{\IR^{d}}\left(\p^{\beta}q_{t}\right)\left(x-y\right)\langle A_{0}\rangle^{m'-m}e^{-tA_{0}^{2}}\langle A_{0}\rangle^{m}\left(\p^{\alpha-\beta}f\right)\left(y\right)\langle A_{0}\rangle^{-n}\Id y,\nonumber\\
			\langle A_{0}\rangle^{m}\p^{\alpha}\left(Q_{t}^{R}f\right)\left(x\right)\langle A_{0}\rangle^{-n'}&=\int_{\IR^{d}}\langle A_{0}\rangle^{m}\left(\p^{\alpha-\beta}f\right)\left(y\right)\langle A_{0}\rangle^{-n}\langle A_{0}\rangle^{n-n'}e^{-tA_{0}^{2}}\left(\p^{\beta}q_{t}\right)\left(x-y\right)\Id y.
		\end{align}
		There are constants $C_{\beta}, c_{r}$, such that $\left\|\p^{\beta}q_{t}\right\|_{L^{1}\left(\IR^{d}\right)}\leq C_{\beta}t^{-\frac{\left|\beta\right|}{2}}$, and $\left\|\langle A_{0}\rangle^{r}e^{-tA_{0}^{2}}\right\|_{B\left(H\right)}\leq c_{r}t^{-\frac{r}{2}}$, while
		\begin{align}
			\left\|\langle A_{0}\rangle^{m}\left(\p^{\alpha-\beta}f\right)\langle A_{0}\rangle^{-n}\right\|_{L^{\infty}\left(\IR^{d},B\left(H\right)\right)}&\leq\left\|f\right\|_{l,m,n},\nonumber\\
			\left\|\langle A_{0}\rangle^{m}\left(\p^{\alpha-\beta}f\right)\langle A_{0}\rangle^{-n}\right\|_{L^{\infty}\left(\IR^{d},S^1\left(H\right)\right)}&\leq\left\|f\right\|_{l,m,n}^{S^1},
		\end{align}
		and so the claims follow.
	\end{proof}
	
	\begin{Lemma}\label{thetaproplem}
		Let $n\in\IN_{0}$, then there is a constant $c$, such that for $t>0$, $j\in\IN$,
		\begin{align}\label{thetaproplemeq2}
			\int_{s\in t\Delta_{j}}\left\|\left(Q_{s}\shuffle_{\ast}\left(A^{2}-M_{0}^{2}\right)^{\otimes j}\right)\left(x,y\right)\right\|_{B\left(\IC^{r}\otimes \dom\ A_{0}^{n}\right)}\Id s\leq\frac{\left(ct\right)^{j}}{j!}q_{t}\left(x-y\right),\ x,y\in\IR^{d}.
		\end{align}
		Let $l,l',m,m',n\in\IN_{0}$ with $m'\geq m\geq n\geq n'$, and $l'\geq l$. For $t>0$ and $j\in\IN_{0}$ define the linear operators $\Theta_{t}^{j,L/R}$, given by
		\begin{align}
			\left(\Theta_{t}^{j,L}f\right)\left(x\right):=\int_{\IR^{d}}\theta_{t}^{j}\left(x,y\right)f\left(y\right)\Id y,\  \left(\Theta_{t}^{j,R}f\right)\left(x\right):=\int_{\IR^{d}}f\left(y\right)\theta_{t}^{j}\left(x,y\right)\Id y,\ f\in C^{l,\left(S^1\right)}_{m,n}\left(\IR^{d}\right).
		\end{align}
		Then $\Theta_{t}^{j,L}$ is continuous as a map to $C^{l',\left(S^1\right)}_{m',n}$, and $\Theta_{t}^{j,R}$ is continuous as a map to $C^{l',\left(S^1\right)}_{m,n'}$, and there is a constant $C$, independent of $j$ and $t$, such that
		\begin{align}
			\left\|\Theta_{t}^{j,L}\right\|_{B\left(C^{l,\left(S^1\right)}_{m,n},C^{l',\left(S^1\right)}_{m',n}\right)},\ \left\|\Theta_{t}^{j,R}\right\|_{B\left(C^{l,\left(S^1\right)}_{m,n},C^{l',\left(S^1\right)}_{m,n'}\right)}\leq\frac{C^{j}}{\prod_{i=1}^{j}\max\left(1,i-\frac{l'-l}{2}\right)}t^{j-\frac{l'-l}{2}}.
		\end{align}
	\end{Lemma}
	
	\begin{proof}
		We begin with the estimate of (\ref{thetaproplemeq2}). Let $c=\sup_{x\in\IR^{d}}\left\|\overline{A\left(x\right)^{2}-A_{0}^2}\right\|_{B\left(\IC^{r}\otimes\dom\ A_{0}^{n}\right)}$. We have for $x,y\in\IR^{d}$,
		\begin{align}
			&\int_{s\in t\Delta_{j}}\left\|\left(Q_{s}\shuffle_{\ast}\left(A^{2}-M_{0}^{2}\right)^{\otimes j}\right)\left(x,y\right)\right\|_{B\left(\IC^{r}\otimes \dom\ A_{0}^{n}\right)}\Id s\nonumber\\
            \leq&\int_{s\in t\Delta_{j}}\left(q_{s}\shuffle_{\ast}c^{\otimes j}\right)\left(x,y\right)\Id s=\frac{\left(ct\right)^{j}}{j!}q_{t}\left(x-y\right).
		\end{align}
		For the following parts of the proof, we therefore may interchange iterated integration freely. We show only the properties of $\Theta_{t}^{j,L}$, the proof for $\Theta_{t}^{j,R}$ is analogous. Let $L_{f}$ denote the operator given by $\left(L_{f}g\right)\left(x\right)=\left(fg\right)\left(x\right)$, for $x\in\IR^{d}$, $f\in C^{l'}_{m',m}\left(\IR^{d}\right)$ and $g\in C^{l',\left(S^1\right)}_{m,n}\left(\IR^{d}\right)$, we have, by the Leibniz rule and the ideal property of $S^1$, that
		\begin{align}\label{thetaproplemeq1}
			\left\|L_{f}\right\|_{B\left(C^{l',\left(S^1\right)}_{m,n},C^{l',\left(S^1\right)}_{m',n}\right)}\leq d^{l'}\left\|f\right\|_{C^{l'}_{m',m}}.
		\end{align}
		In particular there is a constant $c'$ such that
		\begin{align}
			\left\|L_{\overline{A^2-M_{0}^2}}\right\|_{B\left(C^{l',\left(S^1\right)}_{m,n},C^{l',\left(S^1\right)}_{m',n}\right)}\leq c'd^{l'}.
		\end{align}
		We proceed inductively with convention $\theta_{t}^{0}=Q_{t}$. For $j=0$, the statement is Lemma \ref{qproplem}. Let $C'$ be the constant given by (\ref{qproplemeq1}), which is independent of $j$ and $t$. For $j\geq 0$ we have
		\begin{align}
			\theta_{t}^{j+1}\left(x,y\right)=\int_{0}^{t}\int_{\IR^{d}}\theta_{s}^{j}\left(x,z\right)\left(A^{2}\left(z\right)-A_{0}^{2}\right)Q_{t-s}\left(y,z\right)\Id z\ \Id s,\ x,y\in\IR^{d}.
		\end{align}
		Thus by (\ref{thetaproplemeq1}), we have
		\begin{align}
			&\left\|\Theta_{t}^{j+1,L}\right\|_{B\left(C^{l,\left(S^1\right)}_{m,n},C^{l',\left(S^1\right)}_{m',n}\right)}\leq c'd^{l'}\int_{0}^{\frac{t}{2}}\left\|\Theta_{s}^{j,L}\right\|_{B\left(C^{l',\left(S^1\right)}_{m',n}\right)}\left\|Q_{t-s}^{L}\right\|_{B\left(C^{l,\left(S^1\right)}_{m,n},C^{l',\left(S^1\right)}_{m,n}\right)}\Id s\nonumber\\
			&+c'd^{l}\int_{\frac{t}{2}}^{t}\left\|\Theta_{s}^{j,L}\right\|_{B\left(C^{l,\left(S^1\right)}_{m',n},C^{l',\left(S^1\right)}_{m',n}\right)}\left\|Q_{t-s}^{L}\right\|_{B\left(C^{l,\left(S^1\right)}_{m,n}\right)}\Id s\nonumber\\
			&\leq c'C'd^{l'}\frac{C^{j}}{\prod_{i=1}^{j}\max\left(1,i-\frac{l'-l}{2}\right)}\left(\int_{0}^{\frac{t}{2}}s^{j}\left(t-s\right)^{-\frac{l'-l}{2}}\Id s+\int_{\frac{t}{2}}^{t}s^{j-\frac{l'-l}{2}}\Id s\right)\nonumber\\
			&\leq cC'd^{l'}t^{j+1-\frac{l'-l}{2}}\frac{C^{j}}{\prod_{i=1}^{j}\max\left(1,i-\frac{l'-l}{2}\right)}\frac{2^{2+\frac{l'-l}{2}}}{\max\left(1,j+1-\frac{l'-l}{2}\right)}\leq\frac{C^{j+1}}{\prod_{i=1}^{j+1}\max\left(1,i-\frac{l'-l}{2}\right)}t^{j+1-\frac{l'-l}{2}},
		\end{align}
		which holds if we set $C:=2^{2+\frac{l'-l}{2}}c'C'd^{l'}$, which is independent of $j$ and $t$.
	\end{proof}
	
	\begin{Corollary}\label{thetacor}
		Let $n\in\IN_{0}$, then there is a constant $c$, such that for $t>0$,
		\begin{align}\label{thetacoreq1}
			\sum_{j=0}^{\infty}\left\|\theta_{t}^{j}\left(x,y\right)\right\|_{B\left(\IC^{r}\otimes \dom\ A_{0}^{n}\right)}\leq e^{ct}q_{t}\left(x-y\right),\ x,y\in\IR^{d}.
		\end{align}
		Let $t>0$. Let $l,l',m,m',n,n'\in\IN_{0}$ with $m'\geq m\geq n\geq n'$, and $l'\geq l$. Define the linear operators $\Theta_{t}^{L/R}$, given by
		\begin{align}
			\left(\Theta_{t}^{L}f\right)\left(x\right):=\int_{\IR^{d}}\theta_{t}\left(x,y\right)f\left(y\right)\Id y,\ \left(\Theta_{t}^{R}f\right)\left(x\right):=\int_{\IR^{d}}f\left(y\right)\theta_{t}\left(x,y\right)\Id y,\ f\in C^{l,\left(S^1\right)}_{m,n}\left(\IR^{d}\right).
		\end{align}
		Then $\Theta_{t}^{L}$ is continuous as a map to $C^{l',\left(S^1\right)}_{m',n}$, and $\Theta_{t}^{R}$ is continuous as a map to $C^{l',\left(S^1\right)}_{m,n'}$, and there is a continuous function $h$, such that
		\begin{align}
			&\left\|\Theta_{t}^{L}\right\|_{B\left(C^{l,\left(S^1\right)}_{m,n},C^{l',\left(S^1\right)}_{m',n}\right)}=t^{-\frac{l'-l+m'-m}{2}}h\left(t\right),\ \left\|\Theta_{t}^{R}\right\|_{B\left(C^{l,\left(S^1\right)}_{m,n},C^{l',\left(S^1\right)}_{m,n'}\right)}=t^{-\frac{l'-l+n-n'}{2}}h\left(t\right).
		\end{align}
	\end{Corollary}
	
	\begin{proof}
		Inequality (\ref{thetacoreq1}) follows from inequality (\ref{thetaproplemeq2}), and enables us to exchange integration and summation in the following. We have for $f\in C^{l,\left(S^1\right)}_{m,n}\left(\IR^{d}\right)$, by Lemma \ref{qproplem}, and Lemma \ref{thetaproplem},
		\begin{align}
			&\left\|\Theta_{t}^{L}f\right\|_{C^{l',\left(S^1\right)}_{m',n}}=\left\|\int_{\IR^{d}}\left(\sum_{j=0}^{\infty}\left(-1\right)^{j}\theta_{t}^{j}\left(\cdot,y\right)\right)f\left(y\right)\Id y\right\|_{C^{l',\left(S^1\right)}_{m',n}}=\left\|\sum_{j=0}^{\infty}\left(-1\right)^{j}\Theta_{t}^{j,L}f\right\|_{C^{l',\left(S^1\right)}_{m',n}}\nonumber\\
			&\leq Ct^{-\frac{l'-l+m'-m}{2}}+t^{-\frac{l'-l}{2}}\left(\sum_{j=1}^{\infty}\frac{\left(Ct\right)^{j}}{\prod_{i=1}^{j}\max\left(1,i-\frac{l'-l}{2}\right)}\right)\left\|f\right\|_{C^{l,\left(S^1\right)}_{m,n}}=h\left(t\right)t^{-\frac{l'-l+m'-m}{2}}\left\|f\right\|_{C^{l,\left(S^1\right)}_{m,n}},
		\end{align}
		for some continuous function $h$. $\Theta_{t}^{R}$ can be treated analogously.
	\end{proof}
	
	\begin{Lemma}
		Let $l,m,n\in\IN_{0}$ with $m\geq n$. Then there exists a constant $C$, such that for $t>0$, $j\in\IN_{0}$,
		\begin{align}
			\left\|\theta_{t}^{j}\right\|_{C^{l}_{m,n}\left(\IR^{2d}\right)}\leq\begin{cases}
				Ct^{-\frac{l+m-n+d}{2}},& j=0,\\
				\frac{C^{j}}{\prod_{i=1}^{j}\max\left(1,i-\frac{l+d}{2}\right)}t^{j-\frac{l+d}{2}},& j\geq 1.
			\end{cases}
		\end{align}
	\end{Lemma}
	
	\begin{proof}
		We proceed inductively in $j$. For $j=0$, the statement follows from
		\begin{align}
			\left\|Q_{t}\right\|_{C^{l}_{m,n}\left(\IR^{2d}\right)}\leq\left\|\langle A_{0}\rangle^{m-n}e^{-tA_{0}^2}\right\|_{B\left(H\right)}\left\|q_{t}\right\|_{C^{l}_{b}\left(\IR^{d}\right)}\leq C't^{-\frac{l+m-n+d}{2}},
		\end{align}
		where $C'$ is some constant independent of $t$. Recall that there is a constant $c$, such that
		\begin{align}
			\left\|L_{\overline{A^2-M_{0}^2}}\right\|_{B\left(C^{l}_{n,n},C^{l}_{m,n}\right)}\leq cd^{l}.
		\end{align}
		Let $R_f$ denote the operator $\left(R_{f}g\right)\left(x\right):=g\left(x\right)f\left(x\right)$. Then similarly we have
		\begin{align}
			\left\|R_{\overline{A^2-M_{0}^2}}\right\|_{B\left(C^{l}_{n,n},C^{l}_{m,n}\right)}\leq cd^{l}.
		\end{align}
		For $j\geq 0$, we have for $x,y\in\IR^{d}$,
		\begin{align}\label{auxeq1}
			\theta_{t}^{j+1}\left(x,y\right)&=\int_{0}^{t}\int_{\IR^{d}}\theta_{s}^{j}\left(x,z\right)\left(A^{2}\left(z\right)-A_{0}^{2}\right)Q_{t-s}\left(y,z\right)\Id z\ \Id s,\nonumber\\
			&=\int_{0}^{\frac{t}{2}}\Theta_{s}^{j,L}\left(\left(A^{2}-M_{0}^{2}\right)Q_{t-s}\left(\cdot,y\right)\right)\left(x\right)\Id s+\int_{\frac{t}{2}}^{t}Q_{t-s}^{R}\left(\theta_{s}^{j}\left(x,\cdot\right)\left(A^{2}-M_{0}^{2}\right)\right)\left(y\right)\Id s
		\end{align}
		and thus, we have constant $C''$, independent of $j$ and $t$, such that
		\begin{align}\label{auxeq2}
			\left\|\theta^{j+1}_{t}\right\|_{C^{l}_{m,n}\left(\IR^{2d}\right)}\leq& cd^{l}C'C''\frac{C^{j}}{j!}\int_{0}^{\frac{t}{2}}s^{j}\left(t-s\right)^{-\frac{l+d}{2}}\Id s+cd^{l}\frac{C^{j}}{\prod_{i=1}^{j}\max\left(1,i-\frac{l+d}{2}\right)}\int_{\frac{t}{2}}^{t}s^{j-\frac{l+d}{2}}\Id s\nonumber\\
			\leq&\frac{C^{j+1}}{\prod_{i=1}^{j+1}\max\left(1,i-\frac{l+d}{2}\right)}t^{j+1-\frac{l+d}{2}},
		\end{align}
		if we choose $C$, dependent on $c, C', C'',l$, large enough.
	\end{proof}
	
	\begin{Corollary}\label{thetakernelcor}
		Let $l,m,n\in\IN_{0}$ with $m\geq n$. Then $\theta_{t}\in C^{l}_{m,n}\left(\IR^{2d}\right)$ for $t>0$, and there is a continuous function $h$, such that
		\begin{align}
			\left\|\theta_{t}\right\|_{C^{l}_{m,n}\left(\IR^{2d}\right)}=h\left(t\right)t^{-\frac{l+m-n+d}{2}},\ t>0.
		\end{align}
	\end{Corollary}
	
	\begin{proof}
		We have
		\begin{align}
			\left\|\theta_{t}\right\|_{C^{l}_{m,n}\left(\IR^{2d}\right)}\leq&\sum_{j=0}^{\infty}\left\|\theta_{t}^{j}\right\|_{C^{l}_{m,n}\left(\IR^{2d}\right)}\leq C t^{-\frac{l+m-n+d}{2}}+\left(\sum_{j=1}^{\infty}\frac{C^{j}}{\prod_{i=1}^{j}\max\left(1,i-\frac{l+d}{2}\right)}t^{j}\right)t^{-\frac{l+d}{2}}\nonumber\\
			=:&h\left(t\right)t^{-\frac{l+m-n+d}{2}}.
		\end{align}
	\end{proof}
	
	\begin{Lemma}\label{rproplem}
		Let $n\in\IN_{0}$, then there is a constant $c$, such that for $t>0$, $j\in\IN$, and $x,y\in\IR^{d}$,
		\begin{align}\label{rproplemeq1}
			\int_{s\in t\Delta_{j}}\left\|\left(\theta_{s}\shuffle_{\ast}\left(\Ii\di B\right)^{\otimes j}\right)\left(x,y\right)\right\|_{B\left(\IC^{r}\right)\otimes\langle A_{0}\rangle^{-n}S^{1}\left(H\right)\langle A_{0}\rangle^{n}}\Id s\leq\frac{\left(ct\right)^{j}}{j!}e^{ct}q_{t}\left(x-y\right).
		\end{align}
		Let $l,l',m,m',n\in\IN_{0}$ with $m'\geq m\geq n\geq n'$, and $l'\geq l$. For $t>0$ and $j\in\IN$ define the linear operator $R_{t}^{j,L}$, given by
		\begin{align}
			\left(R_{t}^{j,L}f\right)\left(x\right):=\int_{\IR^{d}}r_{t}^{j}\left(x,y\right)f\left(y\right)\Id y,\ f\in C^{l}_{m,n}\left(\IR^{d}\right).
		\end{align}
		Then $R_{t}^{j,L}$ is continuous as a map to $C^{l',S^1}_{m',n}$, and for $0<t\leq t_{0}$ there exists a constant $C$, independent of $j\in\IN$ and $t>0$, such that
		\begin{align}
			\left\|R_{t}^{j,L}\right\|_{B\left(C^{l}_{m,n},C^{l',S^1}_{m',n}\right)}\leq\frac{C^{j}}{\prod_{i=1}^{j}\max\left(1,i-\frac{l'-l}{2}\right)}t^{j-\frac{l'-l}{2}}.
		\end{align}
	\end{Lemma}
	
	\begin{proof}
		By (\ref{thetacoreq1}) we have for $n\in\IN_{0}$ a constant $c'$, such that for $t>0$,
		\begin{align}
			\left\|\theta_{t}\right\|_{B\left(\IC^{r}\otimes \dom\ A_{0}^{n}\right)}\leq e^{c't}q_{t}\left(x-y\right).
		\end{align}
		We recall that $\left(\di B\right)$ is finite rank valued and thus, with \begin{align}
			c:=\max\left(c',\sup_{x\in\IR^{d}}\left\|\Ii\left(\di B\right)\right\|_{B\left(\IC^{r}\right)\otimes\langle A_{0}\rangle^{-n} S^{1}\left(H\right)\langle A_{0}\rangle^{n}}\right),
		\end{align}
		we have for $x,y\in\IR^{d}$,
		\begin{align}
			&\int_{s\in t\Delta_{j}}\left\|\left(\theta_{s}\shuffle_{\ast}\left(\Ii\di B\right)^{\otimes j}\right)\left(x,y\right)\right\|_{B\left(\IC^{r}\right)\otimes\langle A_{0}\rangle^{-n} S^{1}\left(H\right)\langle A_{0}\rangle^{n}}\Id s\nonumber\\
            &\leq\int_{s\in t\Delta_{j}}\left(e^{cs}q_{s}\shuffle_{\ast} c^{\otimes j}\right)\left(x,y\right)\Id s=\frac{\left(ct\right)^{j}}{j!}e^{ct}q_{t}\left(x-y\right),
		\end{align}
		and thus we may exchange integration freely in the following. We note that there is a constant $c$, such that
		\begin{align}
			\left\|L_{\Ii\left(\di B\right)}\right\|_{B\left(C^{l}_{n,n},C^{l,S^1}_{m,n}\right)}+\left\|R_{\Ii\left(\di B\right)}\right\|_{B\left(C^{l}_{m,m},C^{l,S^1}_{m,n}\right)}\leq cd^{l}.
		\end{align}
		With convention $r_{t}^{0}:=\theta_{t}$, $R_{t}^{0,L}:=\Theta_{t}^{L}$, we have by Corollary \ref{thetacor}, for $0< t\leq t_{0}$, a constant $C'$, independent of $t$, such that
		\begin{align}
			\left\|\Theta_{t}^{L}\right\|_{B\left(C^{l,\left(S^1\right)}_{m,n},C^{l',\left(S^1\right)}_{m',n}\right)}\leq C't^{-\frac{l'-l+m'-m}{2}},
		\end{align}
		and thus,
		\begin{align}
			&\left\|R_{t}^{j+1,L}\right\|_{B\left(C^{l}_{m,n},C^{l',S^1}_{m',n}\right)}\leq cd^{l'}\int_{0}^{\frac{t}{2}}\left\|R_{s}^{j,L}\right\|_{B\left(C^{l',S^1}_{m',n}\right)}\left\|\Theta^{L}_{t-s}\right\|_{B\left(C^{l}_{m,n},C^{l'}_{m,n}\right)}\Id s\nonumber\\
			&+cd^{l}\int_{\frac{t}{2}}^{t}\left\|R_{s}^{j,L}\right\|_{B\left(C^{l}_{m',n},C^{l'}_{m',n}\right)}\left\|\Theta^{L}_{t-s}\right\|_{B\left(C^{l,S^1}_{m,n}\right)}\Id s\nonumber\\
			&\leq cC'd^{l'}\frac{C^{j}}{\prod_{i=1}^{j}\max\left(1,i-\frac{l-l'}{2}\right)}\left(\int_{0}^{\frac{t}{2}}s^{j}\left(t-s\right)^{-\frac{l'-l}{2}}\Id s+\int_{\frac{t}{2}}^{t}s^{j-\frac{l'-l}{2}}\Id s\right)\nonumber\\
			&\leq cC'd^{l'}t^{j+1-\frac{l'-l}{2}}\frac{C^{j}}{\prod_{i=1}^{j}\max\left(1,i-\frac{l'-l}{2}\right)}\frac{2^{2+\frac{l'-l}{2}}}{\max\left(1,j+1-\frac{l'-l}{2}\right)}\leq\frac{C^{j+1}}{\prod_{i=1}^{j+1}\max\left(1,i-\frac{l'-l}{2}\right)}t^{j+1-\frac{l'-l}{2}},
		\end{align}
		which holds if we set $C:=2^{2+\frac{l'-l}{2}}cC'd^{l'}$, which is independent of $j$ and $t$.
	\end{proof}
	
	\begin{Lemma}\label{rkernellem}
		Let $n\in\IN_{0}$, then there is a constant $c$, such that for $t>0$,
		\begin{align}
			\sum_{j=1}^{\infty}\left\|r_{t}^{j}\left(x,y\right)\right\|_{B\left(\IC^{r}\right)\otimes\langle A_{0}\rangle^{-n}S^1\left(H\right)\langle A_{0}\rangle^{n}}\leq e^{ct}q_{t}\left(x-y\right),\ x,y\in\IR^{d}.
		\end{align}
		Let $l,m,n\in\IN_{0}$ with $m\geq n$, $1\geq l$. For $0<t\leq t_{0}$ exists a constant $C$, independent of $t$ and $j\in\IN$, such that
		\begin{align}
			\left\|r_{t}^{j}\right\|_{C^{l,S^1}_{m,n}\left(\IR^{2d}\right)}\leq\frac{C^{j}}{\prod_{i=1}^{j}\max\left(1,i-\frac{l+d}{2}\right)}t^{j-\frac{l+d}{2}}.
		\end{align}
	\end{Lemma}
	
	\begin{proof}
		Let $0<t\leq t_{0}$ and $C'$ be a constant, independent of $t$, such that
		\begin{align}
			\left\|\theta_{t}\right\|_{C^{l}_{m,n}\left(\IR^{2d}\right)}\leq C't^{-\frac{l+m-n+d}{2}}.
		\end{align}
		Recall that there is a constant $c$, such that
		\begin{align}
			\left\|L_{\Ii\left(\di B\right)}\right\|_{B\left(C^{l}_{n,n},C^{l,S^1}_{m,n}\right)}+\left\|R_{\Ii\left(\di B\right)}\right\|_{B\left(C^{l}_{m,m},C^{l,S^1}_{m,n}\right)}\leq cd^{l}.
		\end{align}
		We have for $j\geq 0$ with convention $r_{t}^{0}:=\theta_{t}$, and $x,y\in\IR^{d}$,
		\begin{align}
			r_{t}^{j+1}\left(x,y\right)=\int_{0}^{\frac{t}{2}}R_{s}^{j,L}\left(\left(\di B\right)\theta_{t-s}\left(\cdot,y\right)\right)\left(x\right)\Id s+\int_{\frac{t}{2}}^{t}\Theta_{t-s}^{R}\left(r_{s}^{j}\left(x,\cdot\right)\left(\di B\right)\right)\left(y\right)\Id s,
		\end{align}
		and thus we have a constant $C''$ independent of $j$ and $t$, such that
		\begin{align}
			\left\|r^{j+1}_{t}\right\|_{C^{l,S^1}_{m,n}\left(\IR^{2d}\right)}\leq& cd^{l}C'C''\frac{C^{j}}{j!}\int_{0}^{\frac{t}{2}}s^{j}\left(t-s\right)^{-\frac{l+d}{2}}\Id s+cd^{l}C''\frac{C^{j}}{\prod_{i=1}^{j}\max\left(1,i-\frac{l+d}{2}\right)}\int_{\frac{t}{2}}^{t}s^{j-\frac{l+d}{2}}\Id s\nonumber\\
			\leq&\frac{C^{j+1}}{\prod_{i=1}^{j+1}\max\left(1,i-\frac{l+d}{2}\right)}t^{j+1-\frac{l+d}{2}},
		\end{align}
		if we choose $C$, dependent on $c, C', C'',l$, large enough.
	\end{proof}

    \section{Construction of the main integral kernels}

	With the auxiliary kernels constructed in the previous chapter, we may now proceed to construct the integral kernels of the operators appearing in Lemma \ref{reduxlem}, which involve commutators.
	
	\begin{Definition}
		Let $t>0$, $k\in\IN$. Define
		\begin{align}
			\kappa_{t}^{-,k}:=\sum_{j=\frac{k-1}{2}}^{\infty}r^{2j+1}_{t},\ k\text{ odd, }\kappa_{t}^{+,k}:=\sum_{j=\frac{k}{2}}^{\infty}r^{2j}_{t},\ k\text{ even.}
		\end{align}
		Moreover define for	$\epsilon,r>0$, with $\di_{x/y/1}$ denoting the application of $\di$ to the $x$-variable/$y$-variable/first variable, for $x,y\in\IR^{d}$,
		\begin{align}
			\rho_{t}^{\epsilon,r}\left(x,y\right):=&\one_{B_{r}\left(0\right)}\left(x\right)\one_{B_{r}\left(0\right)}\left(y\right)\nonumber\\
			&\tr_{\IC^{r}}\left(\Ii\left(\di_{x}+\di_{y}\right)\left(\left(\Ii\di q_{\epsilon}\otimes q_{\epsilon}\right)\ast\kappa_{t}^{-,d-2}+\left(q_{\epsilon}\otimes q_{\epsilon}\right)\ast\left(A\kappa_{t}^{+,d-1}\right)\right)\left(x,y\right)\right),\nonumber\\
			\sigma_{t}^{\epsilon,r}\left(x,y\right):=&-\one_{B_{r}\left(0\right)}\left(x\right)\one_{B_{r}\left(0\right)}\left(y\right)\left(\left(q_{\epsilon}\otimes q_{\epsilon}\right)\ast\left[A,\tr_{\IC^{r}}\left(A\kappa_{t}^{-,d}+\Ii\di_{1}\kappa_{t}^{+,d-1}\right)\right]\right)\left(x,y\right),\nonumber\\
			\sigma_{t}^{r}\left(x,y\right):=&-\one_{B_{r}\left(0\right)}\left(x\right)\one_{B_{r}\left(0\right)}\left(y\right)\left[A,\tr_{\IC^{r}}\left(A\kappa_{t}^{-,d}+\Ii\di_{1}\kappa_{t}^{+,d-1}\right)\right]\left(x,y\right).
		\end{align}
	\end{Definition}	

    We collect first some properties of the introduced integral kernels.
	
	\begin{Corollary}\label{kkernelcor}
		Let $l,m,n\in\IN_{0}$ with $m\geq n$. Let $t>0$, $k\in\IN$, then $\kappa^{\pm,k}_{t}\in C^{l}_{m,n}\left(\IR^{2d}\right)$, and there are continuous functions $h_{k}$, such that
		\begin{align}
			\left\|\kappa^{-,k}_{t}\right\|_{C^{l,S^1}_{m,n}\left(\IR^{2d}\right)}+\left\|\kappa^{+,k}_{t}\right\|_{C^{l,S^1}_{m,n}\left(\IR^{2d}\right)}\leq h_{k}\left(t\right)t^{k-\frac{l+d}{2}}.
		\end{align}
		Moreover, for $\epsilon,r>0$, and $\psi\in C^{\infty}_{c}\left(\IR^{d}\right)$ with $\supp\psi\in B_{r}\left(0\right)$, we have $\rho_{t}^{\epsilon,r},\sigma_{t}^{\epsilon,r},\sigma_{t}^{r}\in C_{0,0}^{0,S^1}\left(\IR^{2d}\right)$, such that for some continuous functions $h^{1}$, $h^{2}$,
		\begin{align}
			\left\|\left(\psi\otimes\psi\right)\rho_{t}^{\epsilon,r}\right\|_{C_{0,0}^{0,S^1}\left(\IR^{2d}\right)}\leq h^{1}\left(t\right)t^{-\frac{1}{2}},\ \left\|\left(\psi\otimes\psi\right)\sigma_{t}^{\epsilon,r}\right\|_{C_{0,0}^{0,S^1}\left(\IR^{2d}\right)},\left\|\left(\psi\otimes\psi\right)\sigma_{t}^{r}\right\|_{C_{0,0}^{0,S^1}\left(\IR^{2d}\right)}\leq h^{2}\left(t\right),
		\end{align}
		where the estimate on $\left(\psi\otimes\psi\right)\sigma_{t}^{\epsilon,r}$ holds uniformly in $\epsilon>0$.
	\end{Corollary}
	
	\begin{proof}
		We have for $k$ odd, and $0<t\leq t_{0}$,
		\begin{align}\label{kkernelcoreq1}
			\left\|\kappa^{-,k}_{t}\right\|_{C^{l,S^1}_{m,n}\left(\IR^{2d}\right)}\leq&\sum_{j=\frac{k-1}{2}}^{\infty}\left\|r_{t}^{2j+1}\right\|_{C^{l,S^1}_{m,n}\left(\IR^{2d}\right)}\leq\left(\sum_{j=k}^{\infty}\frac{C^{j}}{\prod_{i=1}^{j}\max\left(1,i-\frac{l+d}{2}\right)}t^{j-k}\right)t^{k-\frac{l+d}{2}}\nonumber\\
			=:&h_{k}\left(t\right)t^{k-\frac{l+d}{2}}.
		\end{align}
		For $k\geq 1$ even, and $0<t\leq t_{0}$,
		\begin{align}
			\left\|\kappa^{+,k}_{t}\right\|_{C^{l,S^1}_{m,n}\left(\IR^{2d}\right)}\leq&\sum_{j=\frac{k}{2}}^{\infty}\left\|r_{t}^{2j}\right\|_{C^{l,S^1}_{m,n}\left(\IR^{2d}\right)}\leq\left(\sum_{j=k}^{\infty}\frac{C^{j}}{\prod_{i=1}^{j}\max\left(1,i-\frac{l+d}{2}\right)}t^{j-k}\right)t^{k-\frac{l+d}{2}}\nonumber\\
			=:&h_{k}\left(t\right)t^{k-\frac{l+d}{2}}.
		\end{align}
		As a first consequence, we obtain the claimed estimate for $\left(\psi\otimes\psi\right)\sigma^{r}_{t}$, and since $\left\|q_{\epsilon}\right\|_{L^1\left(\IR^{d}\right)}=1$, the statement for $\left(\psi\otimes\psi\right)\sigma^{\epsilon,r}_{t}$ also holds, uniformly in $\epsilon$. Letting the involved derivatives fall on $q_{\epsilon}$, which is a Schwartz function, we verified the claimed estimates for $\left(\psi\otimes\psi\right)\rho^{\epsilon,r}_{t}$, and $\left(\psi\otimes\psi\right)\sigma^{\epsilon,r}_{t}$.
	\end{proof}

    We will also need a more refined limiting result in Schatten-von Neumann class, which we provide now.
	
	\begin{Lemma}\label{approxlimitlem}
		Let $X$ be a separable Hilbert space, and denote for $r,\epsilon>0$, $f\in L^2\left(\IR^{d},X\right)$, $T\in B\left(L^2\left(\IR^{d},X\right)\right)$, $x\in\IR^{d}$,
		\begin{align}
			\left(\chi_{r}f\right)\left(x\right):=&\one_{B_{r}\left(0\right)}\left(x\right)f\left(x\right),\ \left(\chi^{L}_{r}T\right)\left(f\right):=\chi_{t}Tf,\ \left(\chi^{R}_{r}T\right)\left(f\right):=T\chi_{t}f,\nonumber\\
			p_{\epsilon}f\:=&q_{\epsilon}\ast f,\ \left(p_{\epsilon}^{L}T\right)\left(f\right):=p_{\epsilon}Tf,\ \left(p_{\epsilon}^{R}T\right)\left(f\right):=Tp_{\epsilon}f.
		\end{align}
		If $T\in S^p\left(L^2\left(\IR^{d},X\right)\right)$ for $p\in\left[1,\infty\right)$, then
		\begin{align}
			\chi_{r}^{L}\chi_{r}^{R}p_{\epsilon}^{L}p_{\epsilon}^{R}T\xrightarrow{\epsilon\searrow 0}\chi_{r}^{L}\chi_{r}^{R}T,
		\end{align}
		where the convergence above is uniform in $r$ in $S^{p}\left(L^2\left(\IR^{d},X\right)\right)$-norm. In particular the double limit
		\begin{align}
			\lim_{\left(\epsilon,r\right)\to\left(0,\infty\right)}\chi_{r}^{L}\chi_{r}^{R}p_{\epsilon}^{L}p_{\epsilon}^{R}T=T
		\end{align}
		exists and equals the iterated limits.
	\end{Lemma}
	
	\begin{proof}
		Denote $Y=L^2\left(\IR^{d},X\right)$. Since $T\in S^{p}\left(Y\right)$ can be approximated in $S^{p}$ norm by finite rank operators and since $\left\|\chi_{r}\right\|_{B\left(Y\right)},\left\|p_{\epsilon}\right\|_{B\left(Y\right)}\leq 1$, hold uniformly in $r,\epsilon>0$, we may assume, without loss, that $T$ is finite rank. By linearity we may further restrict ourselves to $T=\langle\cdot,\phi\rangle\phi$, for $\phi\in Y$. Since $\left\|\langle\cdot,f\rangle g\right\|_{S^{p}\left(Y\right)}=\left\|f\right\|_{Y}\left\|g\right\|_{Y}$, $f,g\in Y$, we have
		\begin{align}
			&\left\|\chi_{r}^{L}\chi_{r}^{R}p_{\epsilon}^{L}p_{\epsilon}^{R}T-\chi_{r}^{L}\chi_{r}^{R}T\right\|_{S^{p}\left(Y\right)}=\left\|\langle\cdot,\chi_{r}\left(p_{\epsilon}-1\right)\phi\rangle\chi_{r}\phi+\langle\cdot,\chi_{r}p_{\epsilon}\phi\rangle\chi_{r}\left(p_{\epsilon}-1\right)\phi\right\|_{S^{p}\left(Y\right)}\nonumber\\
			&\leq\left\|\chi_{r}\left(p_{\epsilon}-1\right)\phi\right\|_{Y}\left\|\chi_{r}\phi\right\|_{Y}+\left\|\chi_{r}p_{\epsilon}\phi\right\|_{Y}\left\|\chi_{r}\left(p_{\epsilon}-1\right)\phi\right\|_{Y}.
		\end{align}
		Since $\left\|\chi_{r}\phi\right\|_{Y}$, $\left\|\chi_{r}p_{\epsilon}\phi\right\|_{Y}$ are uniformly bounded by $\left\|\phi\right\|_{Y}$, it remains to show that \break$\left\|\chi_{r}\left(p_{\epsilon}-1\right)\phi\right\|_{Y}\xrightarrow{\epsilon\searrow 0} 0$, uniformly in $r$. However this indeed is the case, since
		\begin{align}
			\sup_{r>0}\left\|\chi_{r}\left(p_{\epsilon}-1\right)\phi\right\|_{Y}\leq\left\|\left(p_{\epsilon}-1\right)\phi\right\|_{L^2\left(\IR^{d},X\right)}\xrightarrow{\epsilon\searrow 0}0.
		\end{align}
		The remaining statements are a direct consequence of uniform convergence and the interchangeability of limits.
	\end{proof}
	
	If we apply this approximation argument to Proposition \ref{exptraceprop} and Lemma \ref{reduxlem}, we obtain the following Corollary.
	
	\begin{Corollary}\label{traceapproxcor}
		Denote for $\epsilon,r>0$, and $X=H$, $P_{\epsilon}:=\one_{\IC^{r}}\otimes p_{\epsilon}$, then for $t>0$,
		\begin{align}\label{traceapproxcoreq1}
			\tr_{\IC^{r}}\left(e^{-tD^{\ast}D}-e^{-tDD^{\ast}}\right)=\lim_{\left(\epsilon,r\right)\to\left(0,\infty\right)}&\int_{0}^{t}\left(\chi_{r}\tr_{\IC^{r}}\left(\left[\Ii\di,\left(\Ii\di P_{\epsilon} K_{s}^{-,d-2}+P_{\epsilon}AK_{s}^{+,d-1}\right)P_{\epsilon}\right]\right)\chi_{r}\right.\nonumber\\
			&\left.-\chi_{r}p_{\epsilon}\left[A,\tr_{\IC^{r}}\left(AK_{s}^{-,d}+\Ii\di K_{s}^{+,d-1}\right)\right]p_{\epsilon}\chi_{r}\right)\Id s,
		\end{align}
		where the double limit is in $S^1\left(L^2\left(\IR^{d},H\right)\right)$.
	\end{Corollary}

	The next Lemma is a table of all relevant integral kernels, and the corresponding operators.
	
	\begin{Lemma}\label{kernelwelllem}
		Let $t>0$, $j\in\IN_{0}$, $k\in\IN$, then for $f\in L^{2}\left(\IC^{r}\otimes H\right)$, and a.e. $x\in\IR^{d}$,
		\begin{align}\label{kernelwelllemeq1}
			\int_{\IR^{d}}\theta^{j}_{t}\left(x,y\right)f\left(y\right)\Id y&=\left(\int_{s\in\Delta_{j}}e^{-sH_{0}}\shuffle\left(A^2-M_{0}^2\right)^{\otimes j}\Id s f\right)\left(x\right),\nonumber\\
			\int_{\IR^{d}}\theta_{t}\left(x,y\right)f\left(y\right)\Id y&=\left(e^{-tH_{B}}f\right)\left(x\right),\ \int_{\IR^{d}}r^{j}_{t}\left(x,y\right)f\left(y\right)\Id y=\left(R_{t}^{j} f\right)\left(x\right),\nonumber\\
			\int_{\IR^{d}}\kappa^{\pm,k}_{t}\left(x,y\right)f\left(y\right)\Id y&=\left(K^{\pm,k}_{t}f\right)\left(x\right),
		\end{align}
		and for $t,\epsilon,r>0$, $f\in L^2\left(\IR^{d},H\right)$, and a.e. $x\in\IR^{d}$,
		\begin{align}\label{kernelwelllemeq2}
			\int_{\IR^{d}}\rho^{\epsilon,r}_{t}\left(x,y\right)f\left(y\right)\Id y&=\left(\chi_{r}\overline{\tr_{\IC^{r}}\left(\left[\Ii\di,\left(\Ii\di P_{\epsilon} K_{t}^{-,d-2}+P_{\epsilon}AK_{t}^{+,d-1}\right)P_{\epsilon}\right]\right)}\chi_{r}f\right)\left(x\right),\nonumber\\
			\int_{\IR^{d}}\sigma^{\epsilon,r}_{t}\left(x,y\right)f\left(y\right)\Id y&=-\left(\chi_{r}p_{\epsilon}\left[A,\tr_{\IC^{r}}\left(AK_{s}^{-,d}+\Ii\di K_{s}^{+,d-1}\right)\right]p_{\epsilon}\chi_{r}f\right)\left(x\right),\nonumber\\
			\int_{\IR^{d}}\sigma^{r}_{t}\left(x,y\right)f\left(y\right)\Id y&=-\left(\chi_{r}\left[A,\tr_{\IC^{r}}\left(AK_{s}^{-,d}+\Ii\di K_{s}^{+,d-1}\right)\right]\chi_{r}f\right)\left(x\right).
		\end{align} 
	\end{Lemma}
	
	\begin{proof}
		The pointwise estimates of the involved kernels by multiples of $q_{t}\left(x-y\right)$ enable us to use Fubini's theorem on appearing infinite sums and integrals, and obtain successively the claimed integral kernel properties (\ref{kernelwelllemeq1}). For $\sigma_{t}^{\epsilon,r}$ and $\sigma_{t}^{r}$ in (\ref{kernelwelllemeq2}), the claim follows also by Fubini's theorem, the memberships $\kappa^{\pm,k}_{t}\in C^{1,S^1}_{m,n}\left(\IR^{2d}\right)$, and that $\chi_{r}$ is compactly supported. Furthermore, $\eta_{t}:=\left(\Ii\di q_{\epsilon}\otimes q_{\epsilon}\right)\ast\kappa_{t}^{-,d-2}+\left(q_{\epsilon}\otimes q_{\epsilon}\right)\ast\left(A\kappa_{t}^{+,d-1}\right)$ is smooth, because $q_{\epsilon}$ is a Schwartz function. We also have for $f\in C^{\infty}_{c}\left(\IR^{d},\IC^{r}\right)$, via integration by parts,
		\begin{align}
			\int_{\IR^{d}}\left(\Ii\left(\di_{1}+\di_{2}\right)\eta_{t}\right)\left(x,y\right)f\left(y\right)\Id y=&\Ii\di\left(\left(\Ii\di P_{\epsilon}K_{t}^{-,d-2}+P_{\epsilon}AK_{t}^{+,d-1}\right)P_{\epsilon}\right)f\nonumber\\
			&-\Ii\sum_{j=1}^{d}c^{j}\left(\left(\Ii\di P_{\epsilon}K_{t}^{-,d-2}+P_{\epsilon}AK_{t}^{+,d-1}\right)P_{\epsilon}\right)\p_{j}f
		\end{align}
		We apply $\tr_{\IC^{r}}$, change $c^{j}$ to the end, and have for $f\in C^{\infty}_{c}\left(\IR^{d},H\right)$,
		\begin{align}
			\int_{\IR^{d}}\tr_{\IC^{r}}\left(\left(\Ii\left(\di_{1}+\di_{2}\right)\eta_{t}\right)\left(x,y\right)\right)f\left(y\right)\Id y=\tr_{\IC^{r}}\left(\left[\Ii\di,\left(\Ii\di P_{\epsilon} K_{t}^{-,d-2}+P_{\epsilon}AK_{t}^{+,d-1}\right)P_{\epsilon}\right]f\right)\left(x\right).
		\end{align}
		Since the right hand side, due to the presence of $P_{\epsilon}$, possesses a bounded extension (which must equal its closure) in $L^2\left(\IR^{d},H\right)$, we obtain the claim for $\rho^{\epsilon,r}_{t}$ in (\ref{kernelwelllemeq2}).
	\end{proof}
	
	To improve the postulated decay $t\searrow 0$ of some of the operators, we show the following formula, similar to Duhamel's principle.
	
	\begin{Lemma}\label{dicommlem}
        Let $T$ be a operator in $H$ such that $TH_{B}$ and $H_{B}T$ are densely defined, and such that $\overline{\left[T,H_{B}\right]}$ is a bounded operator. Then for $t>0$ the commutator $\left[T,e^{-t H_{B}}\right]$ admits a bounded operator extension,
		\begin{align}
			\overline{\left[T,e^{-t H_B}\right]}=-\int_{0}^{t}e^{-\left(t-s\right)H_{B}}\overline{\left[T, H_{B}\right]}e^{-sH_{B}}\Id s,
		\end{align}
		where the integral converges in operator norm.
	\end{Lemma}
	
	\begin{proof}
		For $z\in\IC\backslash\left[0,\infty\right)$, we have
		\begin{align}
			0=\left[T,\left(z-H_{B}\right)^{-1}\left(z-H_{B}\right)\right]=\left[T,\left(z-H_{B}\right)^{-1}\right]\left(z-H_{B}\right)-\left(z-H_{B}\right)^{-1}\left[T,H_{B}\right].
		\end{align}
		Thus
		\begin{align}
			\overline{\left[T,\left(z-H_{B}\right)^{-1}\right]}=\left(z-H_{B}\right)^{-1}\overline{\left[T,H_{B}\right]}\left(z-H_{B}\right)^{-1}.
		\end{align}
		Let $C>0$, and let $\rho\in\left(0,\frac{\pi}{2}\right)$. Let $\Gamma:=\Gamma_{C,\rho}:=\Gamma_{-}\circ\Gamma_{+}^{-1}$, where $\Gamma_{\pm}\left(s\right):=-C\pm e^{\Ii\rho}s$, $s\in\left[0,\infty\right)$. Note that $\Gamma$ encloses the non-negative real axis once. Let $C'>C$ and $\rho\in\left(0,\frac{\pi}{2}\right)$, $\rho'>\rho$, and denote $\Gamma':=\Gamma_{C',\rho'}$. The curve $\Gamma'$ encloses $\Gamma$ once. By holomorphic functional calculus we have
		\begin{align}\label{dicommlemeq1}
			&\frac{1}{\left(2\pi\Ii\right)^{2}}\int_{\lambda\in\Gamma}\int_{\mu\in\Gamma'}\frac{e^{-t\mu}}{\lambda-\mu}\left(\lambda-H_{B}\right)^{-1}\overline{\left[T,H_{B}\right]}\left(\mu-H_{B}\right)^{-1}\Id\mu\ \Id\lambda\nonumber\\
			=&\frac{1}{2\pi\Ii}\int_{\Gamma}\left(\lambda-H_{B}\right)^{-1}\overline{\left[T,H_{B}\right]}\left(\lambda-H_{B}\right)^{-1}\Id\lambda\ e^{-tH_{B}}\nonumber\\
			=&\overline{\left[T,\frac{1}{2\pi\Ii}\int_{\Gamma}\left(\lambda-H_{B}\right)^{-1}\Id\lambda\right]}e^{-tH_{B}}=\overline{\left[T,1\right]}e^{-tH_{B}}=0.
		\end{align}
		Thus,
		\begin{align}
			&\overline{\left[T,e^{-tH_{B}}\right]}=\frac{1}{2\pi\Ii}\int_{\lambda\in\Gamma}e^{-t\lambda}\left(\lambda-H_{B}\right)^{-1}\overline{\left[T,H_{B}\right]}\left(\lambda-H_{B}\right)^{-1}\Id\lambda\nonumber\\
			&=\frac{1}{\left(2\pi\Ii\right)^{2}}\int_{\lambda\in\Gamma}\int_{\mu\in\Gamma_{\lambda}}\frac{e^{-t\lambda}}{\lambda-\mu}\left(\lambda-H_{B}\right)^{-1}\overline{\left[T,H_{B}\right]}\left(\mu-H_{B}\right)^{-1}\Id\mu\ \Id\lambda\nonumber\\
			&\stackrel{(\ref{dicommlemeq1})}{=}\frac{1}{\left(2\pi\Ii\right)^{2}}\int_{\lambda\in\Gamma}\int_{\mu\in\Gamma_{\lambda}}\frac{e^{-t\lambda}-e^{-t\mu}}{\lambda-\mu}\left(\lambda-H_{B}\right)^{-1}\overline{\left[T,H_{B}\right]}\left(\mu-H_{B}\right)^{-1}\Id\mu\ \Id\lambda\nonumber\\
			&=-\frac{1}{\left(2\pi\Ii\right)^{2}}\int_{\lambda\in\Gamma}\int_{\mu\in\Gamma_{\lambda}}\int_{0}^{t}e^{-\left(t-s\right)\lambda}e^{-s\mu}\Id s\ \left(\lambda-H_{B}\right)^{-1}\overline{\left[T,H_{B}\right]}\left(\mu-H_{B}\right)^{-1}\Id\mu\ \Id\lambda\nonumber\\
			&=-\int_{0}^{t}e^{-\left(t-s\right)H_{B}}\overline{\left[T,H_{B}\right]}e^{-sH_{B}}\Id s.
		\end{align}
	\end{proof}

	The following result follows directly from Corollary \ref{simoninterpolcor}.

\begin{Corollary}\label{simonshufflecor}
	For $j\geq d+1$, and $\epsilon>0$, let $F:=\left(f_{i}\right)_{i=1}^{j}\subset L^{d+\epsilon}\left(\IR^{d}\right)\cap L^{\infty}\left(\IR^{d}\right)$, with $\left\|f_{i}\right\|_{L^{d+\epsilon}}+\left\|f_{i}\right\|_{L^{\infty}}\leq c$, $i\in\left\{1,\ldots,j\right\}$. Then for $t>0$, $s\in t\Delta_{j}$,
	\begin{align}
		e^{-s\Delta_{\IR^{d}}}\shuffle F\in S^1\left(L^2\left(\IR^{d}\right)\right),
	\end{align}
	and there is a constant $C$, independent of $j$ and $s$, such that
	\begin{align}
		\left\|e^{-s\Delta_{\IR^{d}}}\shuffle F\right\|_{S^1\left(L^2\left(\IR^{d}\right)\right)}\leq C c^j\prod_{i=1}^{d+1}s_{i}^{-\frac{1}{2}}.
	\end{align}
	For $\epsilon>0$, let $F:=\left(f_{i}\right)_{i=1}^{d}\subset L^{d+\epsilon}\left(\IR^{d}\right)\cap L^{\infty}\left(\IR^{d}\right)$, and assume there exists $i\in\left\{1,\ldots,d\right\}$, such that $f_{i}\in L^{\frac{d+\epsilon}{1+\epsilon}}\left(\IR^{d}\right)$. Then for $t>0$, $s\in t\Delta_{d}$,
	\begin{align}
		e^{-s\Delta_{\IR^{d}}}\shuffle F\in S^1\left(L^2\left(\IR^{d}\right)\right),
	\end{align}
	and there is a constant $C$, independent of $j$ and $s$, such that
	\begin{align}
		\left\|e^{-s\Delta_{\IR^{d}}}\shuffle F\right\|_{S^1\left(L^2\left(\IR^{d}\right)\right)}\leq C\left(\prod_{k=1,k\neq i}^{d}s_{k}^{-\frac{1}{2}}\right)s_{i}^{-\frac{1+\epsilon}{2}}.
	\end{align}
\end{Corollary}

\begin{Lemma}\label{rl1lem}
	Let $t>0$. Then for all $m,n\in\IN_{0}$ with $n+m\leq 1$, we have
	\begin{align}
		x\mapsto \tr_{\IC^{r}}A\left(x\right)r_{t}^{d}\left(x,x\right)\in L^1\left(\IR^{d},\langle A_{0}\rangle^{-n}S^{1}\left(H\right)\langle A_{0}\rangle^{-m}\right).
	\end{align}
	Moreover for any $\delta>0$ small enough there exists a function $h$ bounded near $0$, such that for $t_{0}>0$, $t_{0}\geq t>0$,
	\begin{align}
		\left\|x\mapsto \tr_{\IC^{r}}A\left(x\right)r_{t}^{d}\left(x,x\right)\right\|_{L^1\left(\IR^{d},\langle A_{0}\rangle^{-n}S^{1}\left(H\right)\langle A_{0}\rangle^{-m}\right)}\leq h\left(t\right)t^{-\delta}.
	\end{align}
\end{Lemma}

\begin{proof}
	Let $\psi,\phi\in C^{\infty}_{c}\left(\IR^{d}\right)$, with $0\leq\psi,\phi\leq 1$, $\supp\psi\subseteq\overline{B_{\frac{1}{2}}\left(0\right)}$, and $\psi\equiv 1$ on $\overline{B_{\frac{1}{4}}\left(0\right)}$, and $\supp\phi\subseteq\overline{B_{1}\left(0\right)}$, and $\phi\equiv 1$ on $\overline{B_{\frac{1}{2}}\left(0\right)}$. Then, with $w_{t}\left(x\right):=A\left(x\right)r_{t}^{d}\left(x,x\right)$, $x\in\IR^{d}$, we have
	\begin{align}
		w_{t}=-\left(1-\phi\right)c_{R}\left(1-\psi\right)c_{R}w_{t}\left(1-\phi\right)+\left(1-\phi\right)w_{t}\phi+\phi w_{t}.
	\end{align}
	The last two summands are in $L^1\left(\IR^{d},\langle A_{0}\rangle^{-n}S^{1}\left(H\right)\langle A_{0}\rangle^{-m}\right)$, since $\phi$ has compact support, and $r_{t}^{d}$ is continuous. With the estimate from Lemma \ref{rkernellem}, we then have a constant $C$, such that
	\begin{align}
		\left\|\tr_{\IC^{r}}\left(\left(1-\phi\right)w_{t}\phi+\phi w_{t}\right)\right\|_{L^1\left(\IR^{d},S^1\left(H\right)\right)}\leq Ct^{\frac{d}{2}}.
	\end{align}
	For the first summand $-\left(1-\phi\right)c_{R}\left(1-\psi\right)c_{R}w_{t}\left(1-\phi\right)$ expand $\Ii\left(\di B\right)=T_{0}^{B}+T_{1}^{B}$ with $T_{0}^{B}=\phi\Ii\left(\di B\right)+\left(1-\phi\right)\Ii c_{R}\left(\p_{R}B\right)$, and $T_{1}^{B}=\left(1-\phi\right)\Ii\left(\left(\di-c_{R}\p_{R}\right)B\right)$, in each factor of the definition of $r^{d}_{t}$ (see Definition \ref{kerneldef}). We obtain $2^{d}-1$ summands where at least one factor $T_{0}^{B}$ appears, and one summand which only involves $T_{1}^{B}$. Let us deal first with the summands with at least one factor $T_{0}^{B}$. We note that $T_{0}^{B}\in L^{d-\epsilon}\left(\IR^{d},\langle A_{0}\rangle^{-n}S^{1}\left(H\right)\langle A_{0}\rangle^{-m}\right)$ for some $\epsilon>0$ and all $n,m\in\IN_{0}$, owed to the presence of either the cut-off function or the radial derivative. Thus, setting $T^{B}_{\alpha}:=\left(T^{B}_{\alpha_{1}},\ldots,T^{B}_{\alpha_{d}}\right)$, for $\alpha\in\left\{0,1\right\}^{d}$, where for some $i\in\left\{1,\ldots,d\right\}$ we have $\alpha_{i}=0$, we have, deploying point-wise estimates on $\theta_{s_{j}}$ and on $A\langle A_{0}\rangle^{-1}$,
	\begin{align}\label{rl1lemeq1}
		\left\|\tr_{\IC^{r}}\left(A\left(x\right)\int_{s\in t\Delta_{d}}\left(\theta_{s}\shuffle_{\ast}T_{\alpha}^{B}\right)\left(x,x\right)\Id s\right)\right\|_{\langle A_{0}\rangle^{-n}S^1\left(H\right)\langle A_{0}\rangle^{-m}}\leq ce^{ct}\int_{s\in t\Delta_{d}}\left(q_{s}\shuffle_{\ast}F\right)\left(x,x\right)\Id s,
	\end{align}
	where $F=\left(f_{i}\right)_{i=1}^{d}$ is a collection of continuous functions with $F\subset L^{d+\delta}\left(\IR^{d}\right)\cap L^{\infty}\left(\IR^{d}\right)$, and $f_{i}\in L^{d-\epsilon}\left(\IR^{d}\right)$, for some $\delta>0$ small enough. Thus Corollary \ref{simonshufflecor} applies, and the right-hand side of (\ref{rl1lemeq1}) is the diagonal entry at $x$ of the continuous integral kernel of a trace-class operator $S$ in $L^2\left(\IR^{d}\right)$, with trace-norm estimate
	\begin{align}
		\left\|S\right\|_{S^1\left(L^2\left(\IR^{d}\right)\right)}\leq Ce^{ct}\int_{s\in t\Delta_{d}}\left(\prod_{j=1,j\neq i}^{d}s_{j}^{-\frac{1}{2}}\right)s_{i}^{-\frac{1+\delta}{2}}\Id s\leq C'e^{ct}t^{-\frac{\delta}{2}}=h\left(t\right)t^{-\frac{\delta}{2}},\ t_{0}\geq t>0,
	\end{align}
	for $t_{0}>0$, some constants $C,C'$, and some continuous function $h$. Now consider the remaining case $\alpha_{j}=1$ for all $j\in\left\{1,\ldots,d\right\}$. By Remark \ref{symbolrem}, we have that $\left[H_{B},\left(1-\psi\right)c_{R}\right]=\left[\Delta,\left(1-\psi\right)c_{R}\right]$ is a first order differential operator $P$, with smooth coefficients in $Sym_{-1}$. By Lemma \ref{dicommlem}, we have for $\tau>0$,
	\begin{align}\label{rl1lemeq2}
		\overline{\left[\left(1-\psi\right)c_{R},e^{-\tau H_{B}}\right]}=\int_{0}^{\tau}e^{-\left(\tau-\sigma\right)H_{B}}Pe^{-\sigma H_{B}}\Id\sigma,
	\end{align}
	the integral convergent in $B\left(L^2\left(\IR^{d},\IC^{r}\otimes H\right)\right)$. For $t>0$, let
	\begin{align}
		V_{t}:=\int_{s\in t\Delta_{d}}e^{-s H_{B}}\shuffle\left(T_{1}^{B}\right)^{\otimes d}\Id s,
	\end{align}
	where the integral converges in operator norm. We note that $\left(1-\psi\right)c_{R}$ anti-commutes with $T_{1}^{B}$, since $1-\psi\equiv 1$ on $\supp\left(1-\phi\right)$, and $c_{R}$ anti-commutes with $\di-c_{R}\p_{R}$, because $\p_{R}$ is a component in a orthonormal system given by coordinates. Consequently, since there is a odd number of factors $T_{1}^{B}$ anti-commuting $\left(1-\psi\right)c_{R}$ with $V_{t}$ we obtain by (\ref{rl1lemeq2}),
	\begin{align}
		\left\{\left(1-\psi\right)c_{R},V_{t}\right\}=\sum_{i=0}^{d}\int_{s\in t\Delta_{d+1}}e^{-s H_{B}}\shuffle C^{i}\Id s,
	\end{align}
	where $C^{i}_{j}:= T^{B}_{1}$, $j\neq i$, $C^{i}_{i}:=P$, and $\left\{\cdot,\cdot\right\}$ denotes the anti-commutator. It follows that
	\begin{align}
		\widetilde{V}_{t}:=-\left(1-\phi\right)c_{R}\left(1-\psi\right)c_{R}V_{t}\left(1-\phi\right)=\left(1-\phi\right)c_{R}V_{t}\left(1-\phi\right)c_{R}+\left\{\left(1-\psi\right)c_{R},V_{t}\right\},
	\end{align}
	and by applying $\tr_{\IC^{r}}$ and bringing $c_{R}$ around, we have
	\begin{align}\label{rl1lemeq3}
		2\tr_{\IC^{r}}\widetilde{V}_{t}=\tr_{\IC^{r}}\left\{\left(1-\psi\right)c_{R},V_{t}\right\}=\sum_{i=0}^{d}\tr_{\IC^{r}}\int_{s\in\Delta_{d+1}}e^{-s H_{B}}\shuffle C^{i}\Id s.
	\end{align}
	Now recall that $B$ is a operator function which commutes with some finite rank projection $K_{m}$ from Proposition and Definition \ref{approxpropandef}, for $m\in\IN$ large enough, and that $K_{m}B=B=BK_{m}$. Consequently also $K_{m}T^{B}_{1}=T^{B}_{1}=T^{B}_{1}K_{m}$. Furthermore, with $\widetilde{K}_{m}=\one_{L^2\left(\IR^{d},\IC^{r}\right)}\otimes K_{m}$, we have $\left[H_{B},\widetilde{K}_{m}\right]=B\left[M_{0},\widetilde{K}_{m}\right]+\left[M_{0},\widetilde{K}_{m}\right]B=:K$, where the right hand side is a bounded continuous function with values in the finite rank operators. We have similar by Lemma \ref{dicommlem} for $\tau>0$,
	\begin{align}
		\overline{\left[\widetilde{K}_{m},e^{-\tau H_{B}}\right]}=\int_{0}^{\tau}e^{-\left(\tau-\sigma\right)H_{B}}Ke^{-\sigma H_{B}}\Id\sigma,
	\end{align}
	which shows that
	\begin{align}
		\int_{s\in t\Delta_{d+1}}e^{-s H_{B}}\shuffle C^{i}\Id s=\int_{s\in\Delta_{d+1}}e^{-s H_{B}}\shuffle\widetilde{C}^{i}\Id s+\int_{s\in\Delta_{d+2}}U_{s}\shuffle F^{i}\Id s,
	\end{align}
	where $\widetilde{C}^{i}_{j}=T^{B}_{1}$, $j\neq i$, $\widetilde{C}^{i}_{i}=\widetilde{K}_{m}P$, $U_{s_{j},j}=e^{-s_{j}H_{B}}$, $j\neq i$, $U_{s_{i},i}=\langle X\rangle^{\epsilon}e^{-s_{i}H_{B}}\langle X\rangle^{-\epsilon}$, and finally $F^{i}_{j}=T^{B}_{1}$, $j\notin\left\{i,i+1\right\}$, $F^{i}_{i}=\langle X\rangle^{-\epsilon}K$, $F^{i}_{i+1}=\langle X\rangle^{\epsilon}P$, with some $\epsilon>0$. Since the factor $\langle X\rangle^{\epsilon}e^{-s_{i}H_{B}}\langle X\rangle^{-\epsilon}$ can be regarded as the semi-group generated by $\langle X\rangle^{\epsilon}H_{B}\langle X\rangle^{-\epsilon}$, which satisfies the conditions of the Hille-Yosida theorem (cf. the Neumann series argument of (\ref{boundedlemeq1})), we conclude that the semi-group $\overline{\langle X\rangle^{\epsilon}e^{-s_{i}H_{B}}\langle X\rangle^{-\epsilon}}$ is uniformly bounded in operator norm for all $s_{i}>0$. If we write
	\begin{align}
		\langle M_{0}\rangle^{1+m} e^{-sH_{B}}\shuffle C^{i}\langle M_{0}\rangle^{n}=&\langle M_{0}\rangle^{1+m} e^{-s_{-1}H_{B}}\langle M_{0}\rangle^{-\eta}\left(\langle M_{0}\rangle^{\eta}C^{i}_{0}\langle\di\rangle^{-\beta_{0}}\right)\left(\langle\di\rangle^{\beta_{0}}e^{-s_{0}H_{B}}\right)\nonumber\\
		&\prod_{j=1}^{d}\left(C^{i}_{j}\langle\di\rangle^{-\beta_{j}}\right)\left(\langle\di\rangle^{\beta_{j}}e^{-s_{j}H_{B}}\right)\langle M_{0}\rangle^{n},
	\end{align}
	for suitable $\beta_{j}>0$, $\eta>0$, we may estimate the factors with $C^{i}_{j}\langle\di\rangle^{-\beta_{j}}$ in an appropriate Schatten-von Neumann norm by Corollary \ref{simoninterpolcor}, and $\langle\di\rangle^{\beta_{j}}e^{-s_{j}H_{B}}$ in operator norm, to obtain a $s$-integrable trace-norm of $\langle M_{0}^{1+m}\rangle e^{-sH_{B}}\shuffle C^{i}\langle M_{0}\rangle^{n}$. Careful accounting then leads to the following norm estimate
	\begin{align}
		\int_{s\in t\Delta_{d+1}}\left\|\langle M_{0}\rangle e^{-s H_{B}}\shuffle C^{i}\right\|_{S^1\left(L^2\left(\IR^{d},\IC^{r}\otimes H\right)\right)}\Id s\leq Ct^{\frac{d-1-n-m}{2}+\delta},\ t_{0}\geq t>0,
	\end{align}
	for some $t_{0},\delta>0$, and a constant $C$. For $\int_{s\in\Delta_{d+2}}U_{s}\shuffle F^{i}\Id s$ a similar factorization is possible, and one obtains
	\begin{align}
		\int_{s\in\Delta_{d+2}}\left\|U_{s}\shuffle F^{i}\right\|_{S^1\left(L^2\left(\IR^{d},\IC^{r}\otimes H\right)\right)}\Id s\leq Ct^{\frac{d+1-n-m}{2}-\delta},\ t_{0}\geq t>0,
	\end{align}
	for some $t_{0},\delta>0$, and a constant $C$. Consequently we have by (\ref{rl1lemeq3}),
	\begin{align}
		\left\|\tr_{\IC^{r}}\widetilde{V}_{t}\right\|_{S^1\left(L^2\left(\IR^{d},H\right)\right)}\leq C t^{\frac{d-1-n-m}{2}+\delta},\ t_{0}\geq t>0,
	\end{align}
	for some $t_{0},\delta>0$, and a constant $C$. Now Theorem \ref{bruseethm} and its Remark \ref{bruseerem}, imply for the diagonal of its continuous integral kernel, that
	\begin{align}
		&\left\|\tr_{\IC^{r}}\left(A\left(x\right)\left(1-\phi\left(x\right)\right)\int_{s\in t\Delta_{d}}\left(\theta_{s}\shuffle_{\ast}\left(T_{1}^{B}\right)^{\otimes d}\right)\left(x,x\right)\left(1-\phi\left(x\right)\right)\Id s\right)\right\|_{\langle A_{0}\rangle^{-n}S^1\left(H\right)\langle A_{0}\rangle^{-m}}\nonumber\\
		&\leq Ct^{\frac{d-1-n-m}{2}+\delta},\ t_{0}\geq t>0,
	\end{align}
	for some $t_{0},\delta>0$, and a constant $C$. Compiling all summands and estimates we see that for any $\delta>0$ small enough, we have
	\begin{align}
		\left\|\tr_{\IC^{r}}\ A\left(x\right)r^{d}_{t}\left(x,x\right)\right\|_{\langle A_{0}\rangle^{-n}S^1\left(H\right)\langle A_{0}\rangle^{-m}}\leq h\left(t\right)t^{-\delta},\ t_{0}\geq t>0,
	\end{align}
	for $t_{0}>0$ and some continuous function $h$.
\end{proof}
	
	The following Lemma shows what intuitively seems right, namely that the trace of $\sigma_{s}$ should vanish, due to the presence of a commutator. We note that at some point in the proof we need that $\left(\Delta B\right)$ is bounded, which is the main reason for the introduction of the smoothing procedure in Proposition and Definition \ref{approxpropandef}.
	
	\begin{Lemma}\label{diffkernellem}
		Let $t>0$. Then
		\begin{align}
			\tr_{H}\left(\lim_{r\to\infty}\lim_{\epsilon\searrow 0}\int_{\IR^{d}}\int_{0}^{t}\sigma_{s}^{\epsilon,r}\left(x,x\right)\Id s\ \Id x\right)=0.
		\end{align}
	\end{Lemma}
	
	\begin{proof}
		$\kappa_{t}^{\pm,k}$ is the integral kernel of a self-adjoint operator, and thus
		\begin{align}
			\kappa_{t}^{\pm,k}\left(x,y\right)=\left(\kappa_{t}^{\pm,k}\left(y,x\right)\right)^{\ast},\ x,y\in\IR^{d},
		\end{align}
		which implies
		\begin{align}
			\di_{1}\kappa_{t}^{\pm,k}\left(x,x\right)=\frac{1}{2}\di_{x}\kappa_{t}^{\pm,k}\left(x,x\right)=\left.\frac{1}{2}\left(\di_{x}+\di_{y}\right)\kappa_{t}^{\pm,k}\left(x,y\right)\right|_{y=x},\ x\in\IR^{d}.
		\end{align}
		Analogously we have
		\begin{align}
			\Ii\di_{1}r^{j}_{t}\left(x,x\right)=\left.\frac{1}{2}\left(\Ii\di_{x}+\Ii\di_{y}\right)r^{j}_{t}\left(x,y\right)\right|_{y=x},\ x\in\IR^{d}.
		\end{align}
		If we apply $\tr_{\IC^{r}}$ and integrate by parts, it follows that $\frac{1}{2}\tr_{\IC^{r}}\left(\Ii\di_{1}+\Ii\di_{2}\right)r_{t}^{j}$ is the continuous $S^1\left(H\right)$-valued integral kernel of $\frac{1}{2}\tr_{\IC^{r}}\overline{\left[\Ii\di,R_{t}^{j}\right]}$. We have
		\begin{align}
			\overline{\left[\Ii\di,R_{t}^{j}\right]}&=\sum_{i=0}^{j}\int_{s\in t\Delta_{j}}T_{s}^{i}\shuffle\left(\Ii\di B\right)^{\otimes j}\Id s+\sum_{i=1}^{j}\int_{s\in t\Delta_{j}}e^{-sH_{B}}\shuffle C^{j,i}\Id s,\nonumber\\
			\left(T^{i}_{s}\right)_{k}&:=\begin{cases}
				e^{-s_{k}H_{B}},&k\neq i,\\
				\overline{\left[\Ii\di,e^{-s_{i}H_{B}}\right]},&k=i,
			\end{cases}\ C^{j,i}_{k}:=\begin{cases}
				\Ii\left(\di B\right),&k\neq i,\\
				-\left(\Delta B\right),&k=i.
			\end{cases}
		\end{align}
		By Lemma \ref{dicommlem}, we have
		\begin{align}
			\overline{\left[\Ii\di,e^{-s_{i}H_{B}}\right]}=-\int_{0}^{s_{i}}e^{-\zeta H_{B}}\left(A\Ii\left(\di B\right)+\Ii\overline{\left(\di B\right)A}\right)e^{-\zeta H_{B}}\Id\zeta,
		\end{align}
		which implies
		\begin{align}
			\overline{\left[\Ii\di,R_{t}^{j}\right]}=&\sum_{i=1}^{j+1}\int_{s\in t\Delta_{j+1}}e^{-sH_{B}}\shuffle S^{j+1,i}\Id s+\sum_{i=1}^{j}\int_{s\in t\Delta_{j}}e^{-sH_{B}}\shuffle C^{j,i}\Id s,\nonumber\\
			S^{j+1,i}_{k}:=&\begin{cases}
				\Ii\left(\di B\right),& k\neq i,\\
				-A\Ii\left(\di B\right)-\Ii\overline{\left(\di B\right)A},& k=i.
			\end{cases}
		\end{align}
		Therefore, by the point-wise estimate on $\theta_{t}$, given by Corollary \ref{thetacor}, and Fubini's theorem we have
		\begin{align}\label{diffkernellemeq1}
			\left(\Ii\di_{1}+\Ii\di_{2}\right)r_{t}^{j}=\sum_{i=1}^{j+1}\int_{s\in t\Delta_{j+1}}\theta_{s}\shuffle_{\ast} S^{j+1,i}\Id s+\sum_{i=1}^{j}\int_{s\in t\Delta_{j}}\theta_{s}\shuffle_{\ast}C^{j,i}\Id s,
		\end{align}
		which a priori only holds a.e., however, we note that the right hand side of (\ref{diffkernellemeq1}) is in $C^{0,S^1}_{m,n}\left(\IR^{2d}\right)$ similar to Lemma \ref{rproplem}, with summable norms over $j\in\IN$. We take the sum and apply $\tr_{\IC^{r}}$ to the diagonal, and so we have for $x\in\IR^{d}$,
		\begin{align}
			\tr_{\IC^{r}}\left(\Ii\di_{1}\kappa_{t}^{+,d-1}\left(x,x\right)\right)=&\frac{1}{2}\sum_{j=\frac{d-1}{2}}^{\infty}\sum_{i=0}^{2j+2}\int_{s\in t\Delta_{2j+2}}\tr_{\IC^{r}}\left(\theta_{s}\shuffle_{\ast} S^{2j+2,i}\right)\left(x,x\right)\Id s\nonumber\\
			&+\frac{1}{2}\sum_{j=\frac{d+1}{2}}^{\infty}\sum_{i=1}^{2j+1}\int_{s\in t\Delta_{2j+1}}\tr_{\IC^{r}}\left(\theta_{s}\shuffle_{\ast}C^{2j+1,i}\right)\left(x,x\right)\Id s,
		\end{align}
		Note that we used that $C^{j,i}$ contains $j-1$ factors of Clifford matrices $c^{k}$, and thus the second sum starts with $j=\frac{d+1}{2}$ instead of $\frac{d-1}{2}$. For $n,m\in\IN$ we take the norm in $\langle A_{0}\rangle^{-n}S^1\left(H\right)\langle A_{0}\rangle^{-m}$, and obtain with the pointwise estimate
		$\left\|\theta_{t}\left(x,y\right)\right\|_{B\left(\IC^{r}\otimes\dom\ A_{0}^{n}\right)}\leq e^{ct}q_{t}\left(x-y\right)$, $x,y\in\IR^{d}$,
		\begin{align}\label{diffkernellemeq2}
			&\left\|\tr_{\IC^{r}}\left(\Ii\di_{1}\kappa_{t}^{+,d-1}\left(x,x\right)\right)\right\|_{\langle A_{0}\rangle^{-n}S^1\left(H\right)\langle A_{0}\rangle^{-m}}\leq\frac{e^{ct}}{2}\sum_{j=\frac{d-1}{2}}^{\infty}\sum_{i=0}^{2j+2}\int_{s\in t\Delta_{2j+2}}\left(q_{s}\shuffle_{\ast}f^{2j+2,i}\right)\left(x,x\right)\Id s\nonumber\\
			&+\frac{e^{ct}}{2}\sum_{j=\frac{d+1}{2}}^{\infty}\sum_{i=1}^{2j+1}\int_{s\in t\Delta_{2j+1}}\left(q_{s}\shuffle_{\ast}g^{2j+1,i}\right)\left(x,x\right)\Id s,
		\end{align}
		where $f^{j,i}_{k}:=\left\|S^{j,i}_{k}\right\|_{\left(\dom\ A_{0}^{m}\right)'\otimes \dom\ A_{0}^{n}}$, $g^{j,i}_{k}:=\left\|C^{j,i}_{k}\right\|_{\left(\dom\ A_{0}^{m}\right)'\otimes \dom\ A_{0}^{n}}$, with a uniform constant $c'$, such that
		\begin{align}
			\left\|f^{j,i}_{k}\right\|_{L^{d}\left(\IR^{d}\right)}+\left\|f^{j,i}_{k}\right\|_{L^{\infty}\left(\IR^{d}\right)}+\left\|g^{j,i}_{k}\right\|_{L^{d}\left(\IR^{d}\right)}+\left\|g^{j,i}_{k}\right\|_{L^{\infty}\left(\IR^{d}\right)}\leq c',
		\end{align}
		for all $j\in\IN$, and $i,k\leq j$. The right-hand side of (\ref{diffkernellemeq2}) are by Corollary \ref{simonshufflecor} the diagonal entries of the (continuous) integral kernel of a trace class operator $T$ in $L^2\left(\IR^{d}\right)$, satisfying the trace norm estimate
		\begin{align}
			&\left\|T\right\|_{S^1\left(L^2\left(\IR^{d}\right)\right)}\leq C\frac{e^{ct}}{2}\sum_{j=\frac{d-1}{2}}^{\infty}c'^{2j+2}\sum_{i=0}^{2j+2}\int_{s\in t\Delta_{2j+2}}\prod_{l=1}^{d+1}s_{l}^{-\frac{1}{2}}\Id s\nonumber\\
			&+C\frac{e^{ct}}{2}\sum_{j=\frac{d+1}{2}}^{\infty}c'^{2j+1}\sum_{i=1}^{2j+1}\int_{s\in t\Delta_{2j+1}}\prod_{l=1}^{d+1}s_{l}^{-\frac{1}{2}}\Id s\leq C\frac{e^{ct}}{2}\sum_{j=\frac{d-1}{2}}^{\infty}c'^{2j+2}\left(2j+3\right)\frac{C'}{\left(2j+1-d\right)!}t^{2j+1-d}\nonumber\\
			&+C\frac{e^{ct}}{2}\sum_{j=\frac{d+1}{2}}^{\infty}c'^{2j+1}\left(2j+2\right)\frac{C'}{\left(2j-d\right)!}t^{2j-d}=h\left(t\right),
		\end{align}
		where $h$ is some continuous function. By Theorem \ref{bruseethm} we thus have
		\begin{align}
			\left\|x\mapsto\tr_{\IC^{r}}\left(\Ii\di_{1}\kappa_{t}^{+,d-1}\left(x,x\right)\right)\right\|_{L^1\left(\IR^{d},\langle A_{0}\rangle^{-n}S^1\left(H\right)\langle A_{0}\rangle^{-m}\right)}\leq\left\|T\right\|_{S^1\left(L^2\left(\IR^{d}\right)\right)}\leq h\left(t\right).
		\end{align}
		An analogous argument shows that
		\begin{align}
			\left\|x\mapsto\tr_{\IC^{r}}A\left(x\right)\kappa_{t}^{-,d+2}\left(x,x\right)\right\|_{L^1\left(\IR^{d},\langle A_{0}\rangle^{-n}S^1\left(H\right)\langle A_{0}\rangle^{-m}\right)}\leq h\left(t\right),
		\end{align}
		for some continuous function $h$. Together with Lemma \ref{rl1lem}, we have for $n,m\in\IN_{0}$, with $n+m\leq 1$,
		\begin{align}
			\left\|x\mapsto\tr_{\IC^{r}}A\left(x\right)\kappa_{t}^{-,d}\left(x,x\right)\right\|_{L^1\left(\IR^{d},\langle A_{0}\rangle^{-n}S^1\left(H\right)\langle A_{0}\rangle^{-m}\right)}\leq H\left(t\right)t^{-\delta},
		\end{align}
		for any $\delta>0$, and some continuous function $H$. In total we thus obtain
		\begin{align}\label{diffkernellemeq3}
			\tau_{t}\left(x\right)&:=\tr_{\IC^{r}}\left(\Ii\di_{1}\kappa_{t}^{+,d-1}\left(x,x\right)+A\left(x\right)\kappa_{t}^{-,d}\left(x,x\right)\right),\nonumber\\
			\left\|\tau_{t}\right\|_{L^1\left(\IR^{d},\langle A_{0}\rangle^{-n}S^1\left(H\right)\langle A_{0}\rangle^{-m}\right)}&=O\left(t^{-\delta}\right),\ t\searrow 0,\ n,m\in\IN_{0},\ n+m\leq1,\ \delta>0.
		\end{align}
		Let $\sigma_{t}\left(x\right):=-\overline{\left[A\left(x\right),\tau_{t}\left(x\right)\right]}$, $x\in\IR^{d}$. For $r>0$, $\sigma_{s}^{\epsilon,r}$ possesses an integrable dominant in $S^1\left(H\right)$-norm, by Corollary \ref{kkernelcor}, and converges to $\sigma_{s}^{r}$ pointwise in $S^1\left(H\right)$-norm as $\epsilon\searrow 0$, because $q_{\epsilon}$ is a delta sequence. Therefore by Lebesgue's dominated convergence theorem we have convergence in $S^1\left(H\right)$ of
		\begin{align}
			\lim_{\epsilon\searrow 0}\int_{\IR^{d}}\int_{0}^{t}\sigma_{s}^{\epsilon,r}\left(x,x\right)\Id s\ \Id x=\int_{B_{r}\left(0\right)}\int_{0}^{t}\ \sigma_{s}^{r}\left(x,x\right)\Id s\ \Id x=\int_{\IR^{d}}\one_{B_{r}\left(0\right)}\left(x\right)\int_{0}^{t}\sigma_{s}\left(x\right)\Id s\ \Id x
		\end{align}
		Since $A\langle A_{0}\rangle^{-1}\in L^{\infty}\left(\IR^{d},B\left(H\right)\right)$, we have, by (\ref{diffkernellemeq3}), and dominated convergence, iterated convergence in $S^1\left(H\right)$,
		\begin{align}\label{diffkernellemeq4}
			\lim_{r\to\infty}\lim_{\epsilon\searrow 0}\int_{\IR^{d}}\int_{0}^{t}\sigma_{s}^{\epsilon,r}\left(x,x\right)\Id s=\int_{\IR^{d}}\int_{0}^{t}\sigma_{s}\left(x\right)\Id s\ \Id x.
		\end{align}
		We apply $\tr_{H}$, which we may pull through the integrals by (\ref{diffkernellemeq3}). Then for $s>0$, $x\in\IR^{d}$,
		\begin{align}
			\tr_{H}\sigma_{s}\left(x\right)=&-\tr_{H}\left(A\left(x\right)\tau_{s}\left(x\right)\right)+\tr_{H}\left(\overline{\tau_{s}\left(x\right)A\left(x\right)}\right)=-\tr_{H}\left(\langle A_{0}\rangle\tau_{s}\left(x\right)A\left(x\right)\langle A_{0}\rangle^{-1}\right)\nonumber\\
			&+\tr_{H}\left(\overline{\tau_{s}\left(x\right)A\left(x\right)}\right)=0.
		\end{align}
		Together with (\ref{diffkernellemeq4}), the claim follows.
	\end{proof}
	
	We are prepared to conclude this chapter with a trace formula, in which already the theme of an integral over a sphere with radius $r\to\infty$ appears. Denote by $S_{r}\left(0\right)$ the $\left(d-1\right)$-dimensional sphere of radius $r$ around $0$ in $\IR^{d}$, and for $v\in\IR^{d}$, $\langle c,v\rangle:=\sum_{j=1}^{d}c^{j}v_{j}$.
	
	\begin{Proposition}\label{tracecalcprop}
		Let $t>0$. Then
		\begin{align}\label{tracecalcpropeq0}
			&\tr_{L^2\left(\IR^{d},H\right)}\tr_{\IC^{r}}\left(e^{-tD^{\ast}D}-e^{-tDD^{\ast}}\right)\nonumber\\
			=&\tr_{H}\lim_{r\to\infty}r^{d-1}\left(\int_{0}^{t}\int_{S_{1}\left(0\right)}\tr_{\IC^{r}}\left(\Ii\langle c,y\rangle A\left(ry\right)\kappa_{s}^{+,d-1}\left(ry,ry\right)\right)\Id S_{1}\left(y\right)\ \Id s\right.\nonumber\\
			&\left.+\frac{1}{2}\int_{0}^{t}\int_{S_{1}\left(0\right)}\tr_{\IC^{r}}\left(\p_{r}\kappa_{s}^{-,d}\left(ry,ry\right)\right)\Id S_{1}\left(y\right)\ \Id s\right),
		\end{align}
		where the argument of $\tr_{H}$ of the right hand side is a limit in $S^1\left(H\right)$-norm.
	\end{Proposition}
	
	\begin{proof}
		We start with Corollary \ref{traceapproxcor}, and consider the term of the right hand side of (\ref{traceapproxcoreq1}) without the limit, i.e.
		\begin{align}\label{tracecalcpropeq1}
			T_{t}^{\epsilon,r}:=\int_{0}^{t}&\left(\chi_{r}\tr_{\IC^{r}}\left(\left[\Ii\di,\left(\Ii\di P_{\epsilon} K_{s}^{-,d-2}+P_{\epsilon}AK_{s}^{+,d-1}\right)P_{\epsilon}\right]\right)\chi_{r}\right.\nonumber\\
			&\left.-\chi_{r}p_{\epsilon}\left[A,\tr_{\IC^{r}}\left(AK_{s}^{-,d}+\Ii\di K_{s}^{+,d-1}\right)\right]p_{\epsilon}\chi_{r}\right)\Id s,
		\end{align}
		Now Lemma \ref{kernelwelllem}, and Corollary \ref{kkernelcor} imply that
		\begin{align}\label{tracecalcpropeq2}
			\int_{0}^{t}\left(\rho_{s}^{\epsilon,r}+\sigma_{s}^{\epsilon,r}\right)\Id s
		\end{align}
		is a $S^1\left(H\right)$-valued integral kernel of $T_{t}^{\epsilon,r}$, which is continuous on $B_{r}\left(0\right)\times B_{r}\left(0\right)$. Thus Remark \ref{bruseerem} implies that the $S^1\left(H\right)$-valued integral kernel of $T_{t}^{\epsilon,r}$, which is supported on $B_{r}\left(0\right)\times B_{r}\left(0\right)$, must coincide with (\ref{tracecalcpropeq2}) on its support. Theorem \ref{bruseethm} also implies that $\tr_{\IC^{r}}\left(e^{-tD^{\ast}D}-e^{-tDD^{\ast}}\right)$ possesses an integral kernel $k_{t}$, which is normalized by the condition $h\mapsto\left(x\mapsto k_{t}\left(x,x+h\right)\right)\in C_{b}\left(\IR^{d},L^1\left(\IR^{d},S^1\left(H\right)\right)\right)$, and by Corollary \ref{traceapproxcor} satisfies
		\begin{align}
			\int_{\IR^{d}}k_{t}\left(x,x\right)\Id x=\lim_{\left(\epsilon,r\right)\to\left(0,\infty\right)}\int_{\IR^{d}}\int_{0}^{t}\left(\rho_{s}^{\epsilon,r}+\sigma_{s}^{\epsilon,r}\right)\Id s\ \Id x,
		\end{align}
		where the convergence holds in $S^1\left(H\right)$-norm, by Theorem \ref{bruseethm}. We also obtain by Theorem \ref{bruseethm},
		\begin{align}\label{tracecalcpropeq4}
			\tr_{L^2\left(\IR^{d},H\right)}\tr_{\IC^{r}}\left(e^{-tD^{\ast}D}-e^{-tDD^{\ast}}\right)=\tr_{H}\left(\lim_{r\to\infty}\lim_{\epsilon\searrow 0}\int_{\IR^{d}}\int_{0}^{t}\left(\rho_{s}^{\epsilon,r}+\sigma_{s}^{\epsilon,r}\right)\left(x,x\right)\Id s\ \Id x\right).
		\end{align}
		By Lemma \ref{diffkernellem} we thus have
		\begin{align}\label{tracecalcpropeq3}
			&\tr_{L^2\left(\IR^{d},H\right)}\tr_{\IC^{r}}\left(e^{-tD^{\ast}D}-e^{-tDD^{\ast}}\right)=\tr_{H}\left(\lim_{r\to\infty}\lim_{\epsilon\searrow 0}\int_{\IR^{d}}\int_{0}^{t}\rho_{s}^{\epsilon,r}\left(x,x\right)\Id s\ \Id x\right)\nonumber\\
			&=\tr_{H}\left(\lim_{r\to\infty}\lim_{\epsilon\searrow 0}\int_{0}^{t}\int_{B_{r}\left(0\right)}\tr_{\IC^{r}}\Ii\di_{x}\left(\left(\Ii\di q_{\epsilon}\otimes q_{\epsilon}\right)\ast\kappa_{s}^{-,d-2}+\left(q_{\epsilon}\otimes q_{\epsilon}\right)\ast\left(A\kappa_{s}^{+,d-1}\right)\right)\left(x,x\right)\Id x\ \Id s\right).
		\end{align}
		Since $\left(q_{\epsilon}\otimes q_{\epsilon}\right)\ast\kappa_{s}^{-,d-2}$ is the continuous integral kernel of a self-adjoint operator, we have for $x,y\in\IR^{d}$,
		\begin{align}
			\left(\left(q_{\epsilon}\otimes q_{\epsilon}\right)\ast\kappa_{s}^{-,d-2}\right)\left(x,y\right)=\left(\left(\left(q_{\epsilon}\otimes q_{\epsilon}\right)\ast\kappa_{s}^{-,d-2}\right)\left(y,x\right)\right)^{\ast},
		\end{align}
		and thus for $x\in\IR^{d}$, if we apply the trace $\tr_{\IC^{r}}$, which eliminates the $d-2$ summand,
		\begin{align}\label{tracecalcpropeq8}
			\tr_{\IC^{r}}\Ii\di_{x}\left(\left(\Ii\di q_{\epsilon}\otimes q_{\epsilon}\right)\ast\kappa_{s}^{-,d-2}\right)\left(x,x\right)&=-\frac{1}{2}\tr_{\IC^{r}}\Delta_{x}\left(\left(q_{\epsilon}\otimes q_{\epsilon}\right)\ast\kappa_{s}^{-,d}\right)\left(x,x\right).
		\end{align}
		We insert (\ref{tracecalcpropeq8}) into (\ref{tracecalcpropeq3}), apply the divergence theorem, and choose polar coordinates, to obtain
		\begin{align}\label{tracecalcpropeq6}
			&\tr_{L^2\left(\IR^{d},H\right)}\tr_{\IC^{r}}\left(e^{-tD^{\ast}D}-e^{-tDD^{\ast}}\right)\nonumber\\
			=&\tr_{H}\lim_{r\to\infty}\lim_{\epsilon\searrow 0}\left(r^{d-1}\int_{0}^{t}\int_{S_{1}\left(0\right)}\tr_{\IC^{r}}\left(\Ii\langle c,y\rangle\left(q_{\epsilon}\otimes q_{\epsilon}\right)\ast\left(A\kappa_{s}^{+,d-1}\right)\left(ry,ry\right)\right)\Id S_{1}\left(y\right)\ \Id s\right.\nonumber\\
			&\left.+\frac{1}{2}\int_{0}^{t}\int_{S_{1}\left(0\right)}\tr_{\IC^{r}}\left(\p_{r}\left(q_{\epsilon}\otimes q_{\epsilon}\right)\ast\kappa_{s}^{-,d}\left(ry,ry\right)\right)\Id S_{1}\left(y\right)\ \Id s\right).
		\end{align}
		Fixing $r>0$, by Corollary \ref{kkernelcor}, the above inner integrands of the right hand side of (\ref{tracecalcpropeq6}) possess an integrable dominant in $S^1\left(H\right)$-norm, and converge in $S^1\left(H\right)$-norm point-wise. So by Lebesgue's dominated convergence we obtain
		\begin{align}\label{tracecalcpropeq7}
			&\tr_{L^2\left(\IR^{d},H\right)}\tr_{\IC^{r}}\left(e^{-tD^{\ast}D}-e^{-tDD^{\ast}}\right)\nonumber\\
			=&\tr_{H}\lim_{r\to\infty}\left(r^{d-1}\int_{0}^{t}\int_{S_{1}\left(0\right)}\tr_{\IC^{r}}\left(\Ii\langle c,y\rangle A\left(ry\right)\kappa_{s}^{+,d-1}\left(ry,ry\right)\right)\Id S_{1}\left(y\right)\ \Id s\right.\nonumber\\
			&\left.+\frac{1}{2}\int_{0}^{t}\int_{S_{1}\left(0\right)}\tr_{\IC^{r}}\left(\p_{r}\kappa_{s}^{-,d}\left(ry,ry\right)\right)\Id S_{1}\left(y\right)\ \Id s\right).
		\end{align}
	\end{proof}
	
	\section{The principal trace formula}
	
	In this chapter we exploit that we only need the involved integral kernels in their long range limits $\left|x\right|\to\infty$. The first result shows that we do not need to care about which values the operator family $A$ has near $0$.
	
	\begin{Lemma}\label{ballexciselem}
		For $j\in\IN$ let $M=\left(M_{i}\right)_{i=1}^{j}$ be measurable families of bounded operators in $\IC^{r\times r}\otimes\langle A_{0}\rangle^{-n}B\left(H\right)\langle A_{0}\rangle^{-m}$, with
		\begin{align}
			\left\|M_{i}\left(x\right)\right\|_{L^{\infty}\left(\IR^{d},\IC^{r\times r}\otimes\langle A_{0}\rangle^{-n}B\left(H\right)\langle A_{0}\rangle^{-m}\right)}\leq C,
		\end{align}
		for some constant $C<\infty$. Assume also that there are measurable functions
		\begin{align}
			\left(0,\infty\right)\times\IR^{d}\times\IR^{d}\ni\left(t,x,y\right)\mapsto k_{t}^{i}\left(x,y\right)\in\IC^{r\times r}\otimes B\left(H\right),\ i\in\left\{0,\ldots,j\right\},
		\end{align}
		satisfying for some constant $c<\infty$,
		\begin{align}
			\left\|k_{t}^{i}\left(x,y\right)\right\|_{B\left(\IC^{r}\otimes\dom\ A_{0}^{n}\right)}\leq e^{ct}q_{t}\left(x-y\right),\ x,y\in\IR^{d},\ t>0.
		\end{align}
		Denote for $R>0$, $M':=\left(M_{i}'\right)_{i=1}^{j}$, with $M_{i}'\in\left\{M_{i},M_{i}\one_{B_{\frac{R}{2}}\left(0\right)^{c}}\right\}$, $i\in\left\{1,\ldots,j\right\}$ with at least one $i_{0}\in\left\{1,\ldots,j\right\}$, such that $M_{i_{0}}'=M_{i_{0}}\one_{B_{\frac{R}{2}}\left(0\right)^{c}}$. Then there is a constant $c'$ independent of $j$, such that for $t>0$,
		\begin{align}
			&\int_{S_{R}\left(0\right)}\int_{s\in t\Delta_{j}}\left\|\left(k_{s}\shuffle_{\ast} M-k_{s}\shuffle_{\ast}M'\right)\left(x,x\right)\right\|_{\IC^{r\times r}\otimes\langle A_{0}\rangle^{-n}B\left(H\right)\langle A_{0}\rangle^{-m}}\Id s\ \Id S_{R}\left(x\right)\nonumber\\
			&\leq c'e^{ct}\frac{\left(2Ct\right)^{j}}{j!}R^{-1}.
		\end{align}
	\end{Lemma}
	
	\begin{proof}
		Since $M$, and $M'$ are uniformly bounded in $\left\|\cdot\right\|_{\IC^{r\times r}\otimes\langle A_{0}\rangle^{-n}B\left(H\right)\langle A_{0}\rangle^{-m}}$-norm by some constant $C$, we have for $t>0$, $x\in S_{R}\left(0\right)$,
		\begin{align}\label{ballexciselemeq1}
			&\int_{s\in t\Delta_{j}}\left\|\left(k_{s}\shuffle_{\ast} M-k_{s}\shuffle_{\ast}M'\right)\left(x,x\right)\right\|_{\IC^{r\times r}\otimes\langle A_{0}\rangle^{-n}B\left(H\right)\langle A_{0}\rangle^{-m}}\Id s\nonumber\\
			&\leq e^{ct}C^{j}\left(2^{j}-1\right)\int_{0}^{t}\frac{u^{j-1}}{\left(j-1\right)!}\int_{B_{\frac{R}{2}}\left(0\right)}q_{t-u}\left(x-y\right)q_{u}\left(y-x\right)\Id y\ \Id u
		\end{align}
		Since $\mathrm{dist}\left(x,y\right)\geq\frac{R}{2}$, there exist constants $c_{d}, c_{d}'$ only dependent on $d$, such that
		\begin{align}\label{ballexciselemeq2}
			&\int_{B_{\frac{R}{2}}\left(0\right)}q_{t-u}\left(x-y\right)q_{u}\left(y-x\right)\Id u\leq R^{d}c_{d}\left(\sup_{\sigma>0}\sigma^{-\frac{d}{2}}e^{-\frac{R^2}{16\sigma}}\right)^{2}\stackrel{\lambda=\frac{R^2}{\sigma}}{=}R^{d}c_{d}\left(R^{-d}\sup_{\lambda>0}\lambda^{\frac{d}{2}}e^{-\frac{\lambda}{16}}\right)^{2}\nonumber\\
			=&c_{d}'R^{-d}.
		\end{align}
		Compiling (\ref{ballexciselemeq1}) and (\ref{ballexciselemeq2}), we obtain
		\begin{align}
			&\int_{s\in t\Delta_{j}}\left\|\left(k_{s}\shuffle_{\ast} M-k_{s}\shuffle_{\ast}\left(M\one_{B_{\frac{R}{2}}\left(0\right)^{c}}\right)\right)\left(x,x\right)\right\|_{\IC^{r\times r}\otimes\langle A_{0}\rangle^{-n}B\left(H\right)\langle A_{0}\rangle^{-m}}\Id s\nonumber\\
			&\leq e^{ct}c'_{d}C^{j}\left(2^{j}-1\right)\frac{t^{j}}{j!}R^{-d}.
		\end{align}
		Integration over $S_{R}\left(0\right)$ yields the claim.
	\end{proof}
	
	\begin{Lemma}\label{decaycutofflem}
		Let $j\geq d$, and let $F:=\left(F_{k}\right)_{k=1}^{j}$ be a tuple of measurable $B\left(\IC^{r}\otimes H\right)$-valued functions, which satisfy
		\begin{align}
			\left\|\langle X\rangle F_{k}\right\|_{L^{\infty}\left(\IR^{d},\IC^{r\times r}\otimes\langle A_{0}\rangle^{-n}B\left(H\right)\langle A_{0}\rangle^{-m}\right)}\leq C<\infty.
		\end{align}
		Let there be measurable functions
		\begin{align}
			\left(0,\infty\right)\times\IR^{d}\times\IR^{d}\ni\left(t,x,y\right)\mapsto k_{t}^{i}\left(x,y\right)\in B\left(\IC^{r}\otimes\dom\ A_{0}^{n}\right),\ i\in\left\{0,\ldots,j\right\},
		\end{align}
		satisfying for some constant $c<\infty$,
		\begin{align}
			\left\|k_{t}^{i}\left(x,y\right)\right\|_{B\left(\IC^{r}\otimes\dom\ A_{0}^{n}\right)}\leq e^{ct}q_{t}\left(x-y\right),\ x,y\in\IR^{d},\ t>0.
		\end{align}
		Then, for $t_{0}>0$ there is a constant $C'$ independent of $j$, such that for $0<t\leq t_{0}$, $R>0$,
		\begin{align}
			\int_{S_{R}\left(0\right)}\int_{s\in t\Delta_{j}}\left\|\left(k_{s}\shuffle_{\ast}F\right)\left(x,x\right)\right\|_{\IC^{r\times r}\otimes\langle A_{0}\rangle^{-n}B\left(H\right)\langle A_{0}\rangle^{-m}}\Id s\ \Id x\leq C'e^{ct}\frac{\left(2C\right)^{j}}{j!}t^{j-\frac{d}{2}}R^{-1}.
		\end{align}
	\end{Lemma}
	
	\begin{proof}
		Let $R>0$ and consider first that all $F_{k}$ are supported outside of $B_{\frac{R}{2}}\left(0\right)$. Then
		\begin{align}
			R\left\|F_{k}\left(x\right)\right\|_{\IC^{r\times r}\otimes\langle A_{0}\rangle^{-n}B\left(H\right)\langle A_{0}\rangle^{-m}}\leq 2C,
		\end{align}
		for a.e. $x\in\IR^{d}$. So
		\begin{align}
			&\int_{S_{R}\left(0\right)}\int_{s\in t\Delta_{j}}\left\|\left(k_{s}\shuffle_{\ast}F\right)\left(x,x\right)\right\|_{\IC^{r\times r}\otimes\langle A_{0}\rangle^{-n}B\left(H\right)\langle A_{0}\rangle^{-m}}\Id s\ \Id x\leq e^{ct}\left(4\pi t\right)^{-\frac{d}{2}}\frac{\left(2Ct\right)^{j}}{j!}\vol S_{R}\left(0\right)R^{-j}\nonumber\\
			&\leq C'e^{ct}\frac{\left(2C\right)^{j}}{j!}t^{j-\frac{d}{2}}R^{-1}.
		\end{align}
		The general case follows by Lemma \ref{ballexciselem}.
	\end{proof}

	The next lemma shows that we may delete all higher summands $r^{j}_{t}$, $j\geq d$, which make up the kernels $\kappa^{+,d-1}$ and $\kappa^{-,d}$. In particular, the term in the trace formula of Proposition \ref{tracecalcprop} which involves the radial derivative is also eliminated.
	
	\begin{Lemma}\label{cutofflem}
		For $t>0$,
		\begin{align}
			&\tr_{L^{2}\left(\IR^{d},H\right)}\tr_{\IC^{r}}\left(e^{-tD^{\ast}D}-e^{-tDD^{\ast}}\right)\nonumber\\
			=&\tr_{H}\left(\lim_{r\to\infty}r^{d-1}\int_{0}^{t}\int_{S_{1}\left(0\right)}\tr_{\IC^{r}}\left(\Ii\langle c,y\rangle A\left(ry\right)r_{s}^{d-1}\left(ry,ry\right)\right)\Id S_{1}\left(y\right)\ \Id s\right).
		\end{align}
	\end{Lemma}
	
	\begin{proof}
		Let $y\in S_{1}\left(0\right)$, and denote by $\nabla^{y}$ the directional derivative along $y$. We have for $r>0$,
		\begin{align}
			\p_{r}\kappa_{s}^{-,d}\left(ry,ry\right)=\sum_{j=\frac{d-1}{2}}^{\infty}\left(\left(\nabla_{1}^{y}+\nabla_{2}^{y}\right)r_{s}^{2j+1}\right)\left(ry,ry\right),
		\end{align}
		where the sum converges as continuous functions uniformly in $r,y$.	By integration by parts, $\left(\nabla_{1}^{y}+\nabla_{2}^{y}\right)r_{s}^{j}$ is the continuous integral kernel of the operator $\overline{\left[\nabla^{y},R_{s}^{j}\right]}$, and so, by Lemma \ref{dicommlem}, and similar to the proof of Lemma \ref{diffkernellem}, we have
		\begin{align}
			\left(\nabla_{1}^{y}+\nabla_{2}^{y}\right)r_{s}^{j}=&\sum_{i=1}^{j+1}\int_{u\in s\Delta_{j+1}}\theta_{s}\shuffle_{\ast}S^{j+1,i}\Id s+\sum_{i=1}^{j}\int_{u\in s\Delta_{j}}\theta_{s}\shuffle_{\ast}T^{j,i}\Id s,\nonumber\\
			S^{j+1,i}_{k}:=&\begin{cases}
				\Ii\left(\di B\right),&k\neq i,\\
				-A\Ii\left(\di B\right)-\Ii\overline{\left(\di B\right)A},&k=i,
			\end{cases}\ T^{j,i}_{k}:=\begin{cases}
				\Ii\left(\di B\right),&k\neq i,\\
				-\Ii\left(\nabla^{y}\di B\right),&k=i.
			\end{cases}
		\end{align}
		In particular, for $n,m\in\IN_{0}$, there exists a constant $C$, such that for all $i,j,k\in\IN$,
		\begin{align}
			&\left\|\langle X\rangle S_{k}^{j+1,i}\right\|_{L^{\infty}\left(\IR^{d},\IC^{r\times r}\otimes\langle A_{0}\rangle^{-n}B\left(H\right)\langle A_{0}\rangle^{-m}\right)}+\left\|\langle X\rangle T_{k}^{j,i}\right\|_{L^{\infty}\left(\IR^{d},\IC^{r\times r}\otimes\langle A_{0}\rangle^{-n}B\left(H\right)\langle A_{0}\rangle^{-m}\right)}\nonumber\\
			&\leq C<\infty.
		\end{align}
		Lemma \ref{decaycutofflem} implies that for $t>0$,
		\begin{align}\label{cutofflemeq1}
			&R^{d-1}\int_{0}^{t}\int_{S_{1}\left(0\right)}\left\|\p_{r}\kappa_{s}^{-,d}\left(ry,ry\right)\right\|_{B\left(\IC^{r}\otimes H\right)}\Id S_{1}\left(y\right)\ \Id s\nonumber\\
			&\leq C'R^{-1}\int_{0}^{t}e^{cs}\sum_{j=\frac{d-1}{2}}^{\infty}\left(\left(2j+2\right)\frac{\left(2C\right)^{2j+2}}{\left(2j+2\right)!}s^{2j+2-\frac{d}{2}}+\left(2j+1\right)\frac{\left(2C\right)^{2j+1}}{\left(2j+1\right)!}s^{2j+1-\frac{d}{2}}\right)\Id s\xrightarrow{R\to\infty}0.
		\end{align}
		Similarly, since $A\langle A_{0}\rangle^{-1}$ is uniformly bounded, say by $c'$, we have by Lemma \ref{decaycutofflem},
		\begin{align}\label{cutofflemeq2}
			&\int_{0}^{t}\int_{S_{R}\left(0\right)}\left\|A\left(x\right)\kappa_{s}^{+,d+1}\left(x,x\right)\right\|_{B\left(\IC^{r}\otimes H\right)}\Id S_{R}\left(x\right)\ \Id s\leq c'C'R^{-1}\int_{0}^{t}e^{cs}\sum_{j=\frac{d+1}{2}}^{\infty}\frac{\left(2C\right)^{2j}}{\left(2j\right)!}s^{2j-\frac{d}{2}}\Id s\nonumber\\
			&\xrightarrow{R\to\infty}0.
		\end{align}
		The claim then follows from (\ref{cutofflemeq1}), (\ref{cutofflemeq2}), and Proposition \ref{tracecalcprop}.
	\end{proof}
	
	We are now ready to simplify the right hand side of the principal trace formula to its penultimate form, which has its own raison d'\^{e}tre, in so far as it exhibits clearly a similar form as the classical Callias index formula (i.e. a limit of a $d-1$ form integrated over spheres with radii tending to infinity). However it is not the final form, because we should not forget that $A=A_{n,m,l}$ is still an approximant. Our general conditions on the family $A$ do not suffice to enable us to go to the limit $n,m,l\to\infty$ in the trace formula of the lemma below.
	
	\begin{Lemma}\label{qforthetalem}
		For $t>0$,
		\begin{align}
			\tr_{L^{2}\left(\IR^{d},H\right)}\tr_{\IC^{r}}&\left(e^{-tD^{\ast}D}-e^{-tDD^{\ast}}\right)\nonumber\\
			=\tr_{H}&\left(\lim_{r\to\infty}r^{d-1}\int_{0}^{t}\left(4\pi s\right)^{-\frac{d}{2}}\Ii^{d}\kappa_{c}\int_{S_{1}\left(0\right)}A\left(ry\right)\int_{u\in s\Delta_{d-1}}\sum_{j=1}^{d}\sum_{\alpha\in\left\{1,\ldots,d\right\}^{d-1}}\right.\nonumber\\
			&\left.y^{j}\epsilon_{\left(j,\alpha\right)}e^{-uA\left(ry\right)^{2}}\shuffle\left(\left(\p_{\alpha_{1}}^{E}A\right)\left(ry\right),\ldots,\left(\p_{\alpha_{d-1}}^{E}A\right)\left(ry\right)\right)\Id u\ \Id S_{1}\left(y\right)\ \Id s\right),
		\end{align}
        where $\epsilon_{\left(j,\alpha\right)}$ denotes the Levi-Civita symbol of the tuple $\left(j,\alpha\right)\in\left\{1,\ldots,d\right\}^{d}$.
	\end{Lemma}
	
	\begin{proof}
		For $x_{0}\in\IR^{d}$, denote $\left(M_{x_{0}}f\right)\left(x\right):=A\left(x_{0}\right)f\left(x\right)$, $x\in\IR^{d}$, $f\in\dom\ M_{0}$, and $H_{x_{0}}:=\Delta+M_{x_{0}}^{2}$, which is self-adjoint on $\dom\ H_{0}$. The smooth integral kernel $Q^{x_{0}}_{t}$ of $e^{-tH_{x_{0}}}$, $t>0$, is given by
		\begin{align}
			Q^{x_{0}}_{t}\left(x,y\right)=q_{t}\left(x-y\right)\one_{\IC^{r}}\otimes e^{-tA\left(x_{0}\right)^{2}},\ x,y\in\IR^{d}.
		\end{align}
		Duhamel's principle yields $e^{-tH_{B}}=e^{-tH_{x_{0}}}-\int_{0}^{t}e^{-\left(t-s\right)H_{x_{0}}}\left(A^{2}-M_{x_{0}}^{2}\right)e^{-sH_{B}}\Id s$, which implies for a.e. $x,y\in\IR^{d}$,
		\begin{align}\label{qforthetalemeq1}
			\theta_{t}\left(x,y\right)=Q^{x_{0}}_{t}\left(x,y\right)-\int_{0}^{t}\int_{\IR^{d}}Q^{x_{0}}_{t-s}\left(x,z\right)\left(A\left(z\right)^{2}-A\left(x_{0}\right)^{2}\right)\theta_{s}\left(z,y\right)\Id z\ \Id s.
		\end{align}
		In fact the right hand side is an element of $C_{m,n}^{\infty}\left(\IR^{2d}\right)$, $m,n\in\IN_{0}$, which can be shown by using Lemma \ref{qproplem}, Corollary \ref{thetacor}, and Corollary \ref{thetakernelcor} in a similar way as (\ref{auxeq1}), and (\ref{auxeq2}). So (\ref{qforthetalemeq1}) holds everywhere. We find,
		\begin{align}
			r_{t}^{d-1}\left(x_{0},x_{0}\right)=&\int_{s\in t\Delta_{d-1}}\left(\left(Q_{s_{0}}^{x_{0}},\theta_{s_{1}},\ldots,\theta_{s_{d-1}}\right)\shuffle_{\ast}\left(\Ii\di B\right)^{\otimes\left(d-1\right)}\right)\left(x_{0},x_{0}\right)\Id s\nonumber\\
			-&\int_{s\in t\Delta_{d}}\left(\left(Q_{s_{0}}^{x_{0}},\theta_{s_{1}},\ldots,\theta_{s_{d}}\right)\shuffle_{\ast}\left(A^{2}-A\left(x_{0}\right)^{2},\left(\Ii\di B\right)^{\otimes\left(d-1\right)}\right)\right)\left(x_{0},x_{0}\right)\Id s.
		\end{align}
		By Lemma \ref{ballexciselem}, we have for $n,m\in\IN_{0}$, with $A\left(x\right)\langle A_{0}\rangle^{-1}\leq C'$, $x\in\IR^{d}$, and $B_{x_{0}}:=A^{2}-A\left(x_{0}\right)^{2}$,
		\begin{align}\label{qforthetalemeq2}
			&\limsup_{R\to\infty}\int_{0}^{t}\int_{S_{R}\left(0\right)}\int_{u\in s\Delta_{d}}\nonumber\\
			&\left\|A\left(x_{0}\right)\left(\left(Q_{u_{0}}^{x_{0}},\theta_{\left(u_{1},\ldots,u_{d}\right)}\right)\shuffle_{\ast}\left(B_{x_{0}},\left(\Ii\di B\right)^{\otimes\left(d-1\right)}\right)\right)\left(x_{0},x_{0}\right)\right\|_{\IC^{r\times r}\otimes\langle A_{0}\rangle^{-n}B\left(H\right)\langle A_{0}\rangle^{-m}}\nonumber\\
			&\Id u\ \Id S_{R}\left(x_{0}\right)\ \Id s\nonumber\\
			=&\limsup_{R\to\infty}\int_{0}^{t}\int_{S_{R}\left(0\right)}\int_{u\in s\Delta_{d}}\nonumber\\
			&\left\|A\left(x_{0}\right)\left(\left(Q_{u_{0}}^{x_{0}},\theta_{\left(u_{1},\ldots,u_{d}\right)}\right)\shuffle_{\ast}\left(B_{x_{0}},\left(\one_{B_{\frac{R}{2}}\left(0\right)^{c}}\Ii\left(\di B\right)\right)^{\otimes\left(d-1\right)}\right)\right)\left(x_{0},x_{0}\right)\right\|_{\IC^{r\times r}\otimes\langle A_{0}\rangle^{-n}B\left(H\right)\langle A_{0}\rangle^{-m}}\nonumber\\
			&\Id u\ \Id S_{R}\left(x_{0}\right)\ \Id s\nonumber\\
			&\leq C'C^{d-1}\limsup_{R\to\infty}\int_{0}^{t}e^{cs}\int_{S_{1}\left(0\right)}\int_{0}^{s}\int_{\IR^{d}}q_{s-u}\left(Ry-z\right)\left\|A\left(z\right)^{2}-A\left(Ry\right)^{2}\right\|_{B\left(H,\dom\ A_{0}\right)}\nonumber\\
            &q_{u}\left(z-Ry\right)\Id z\ \frac{u^{d-1}}{\left(d-1\right)!}\Id u\ \Id S_{1}\left(y\right)\ \Id s\nonumber\\
			&=C'C^{d-1}\limsup_{R\to\infty}\int_{0}^{t}e^{cs}\int_{S_{1}\left(0\right)}\int_{0}^{s}\int_{\IR^{d}}q_{s-u}\left(w\right)q_{u}\left(w\right)\left\|A\left(Ry+w\right)^{2}-A\left(Ry\right)^{2}\right\|_{B\left(H,\dom\ A_{0}\right)}\nonumber\\
			&\Id w\ \frac{u^{d-1}}{\left(d-1\right)!}\Id u\ \Id S_{1}\left(y\right)\ \Id s.
		\end{align}
		The integrand of the last line of (\ref{qforthetalemeq2}) has an integrable dominant, since
		\begin{align}
			\left\|A\left(Ry+w\right)^{2}-A\left(Ry\right)^{2}\right\|_{B\left(H,\dom\ A_{0}\right)},
		\end{align}
	 	is uniformly bounded (since $B$ is uniformly bounded with values in $\langle A_{0}\rangle^{-m}B\left(H\right)\langle A_{0}\rangle^{n}$ for any $m,n\in\IN_{0}$), and we may take the limit $R\to\infty$ inside the integrals, which yields $0$. Therefore
		\begin{align}
			&\lim_{r\to\infty}r^{d-1}\int_{0}^{t}\int_{S_{1}\left(0\right)}\tr_{\IC^{r}}\left(\Ii\langle c,y\rangle A\left(ry\right)r_{s}^{d-1}\left(ry,ry\right)\right)\Id S_{1}\left(y\right)\ \Id s\nonumber\\
			=&\lim_{r\to\infty}r^{d-1}\int_{0}^{t}\int_{S_{1}\left(0\right)}\tr_{\IC^{r}}\left(\Ii\langle c,y\rangle A\left(ry\right)\int_{u\in s\Delta_{d-1}}\right.\nonumber\\
			&\left.\left(\left(Q_{u_{0}}^{ry},\theta_{u_{1},\ldots,u_{d-1}}\right)\shuffle_{\ast}\left(\Ii\di B\right)^{\otimes\left(d-1\right)}\right)\left(ry,ry\right)\Id u\right)\Id S_{1}\left(y\right)\ \Id s
		\end{align}
		By iterating the procedure, we obtain
		\begin{align}
			&\lim_{r\to\infty}r^{d-1}\int_{0}^{t}\int_{S_{1}\left(0\right)}\tr_{\IC^{r}}\left(\Ii\langle c,y\rangle A\left(ry\right)r_{s}^{d-1}\left(ry,ry\right)\right)\Id S_{1}\left(y\right)\ \Id s\nonumber\\
			&=\lim_{r\to\infty}r^{d-1}\int_{0}^{t}\int_{S_{1}\left(0\right)}\tr_{\IC^{r}}\left(\Ii\langle c,y\rangle A\left(ry\right)p_{s}\left(ry\right)\right)\Id S_{1}\left(y\right)\ \Id s,
		\end{align}
		where for $t>0$, $x_{0}\in\IR^{d}$ we define,
		\begin{align}
			p_{t}\left(x_{0}\right):=\int_{s\in t\Delta_{d-1}}\left(Q^{x_{0}}_{s}\shuffle_{\ast}\left(\Ii\di B\right)^{\otimes\left(d-1\right)}\right)\left(x_{0},x_{0}\right)\Id s.
		\end{align}
		
		\textit{Step 2:} Denote $\left(N_{x_{0}}f\right)\left(x\right):=\left(\Ii\di B\right)\left(x_{0}\right)f\left(x\right)$, $f\in\dom\ M_{0}$, then
		\begin{align}\label{qforthetalemeq3}
			&\limsup_{R\to\infty}\int_{0}^{t}\int_{S_{R}\left(0\right)}\int_{u\in s\Delta_{d-1}}\left\|A\left(x_{0}\right)Q_{u}^{x_{0}}\shuffle_{\ast}\left(\Ii\left(\di B\right)-N_{x_{0}},\left(\Ii\di B\right)^{\otimes\left(d-2\right)}\right)\left(x_{0},x_{0}\right)\right\|_{\IC^{r\times r}\otimes B\left(H\right)}\nonumber\\
			&\Id u\ \Id S_{R}\left(x_{0}\right)\Id s\nonumber\\
			=&\limsup_{R\to\infty}\int_{0}^{t}\int_{S_{R}\left(0\right)}\int_{u\in s\Delta_{d-1}}\nonumber\\
			&\left\|A\left(x_{0}\right)\left(Q_{u}^{x_{0}}\shuffle_{\ast}\left(\Ii\left(\di B\right)-N_{x_{0}},\left(\one_{B_{\frac{R}{2}}\left(0\right)^{c}}\Ii\left(\di B\right)\right)^{\otimes\left(d-2\right)}\right)\right)\left(x_{0},x_{0}\right)\right\|_{\IC^{r\times r}\otimes B\left(H\right)}\Id u\ \Id S_{R}\left(x_{0}\right)\Id s\nonumber\\
			\leq& C'C^{d-1}\limsup_{R\to\infty}\int_{0}^{t}e^{cs}\int_{S_{1}\left(0\right)}\int_{0}^{s}\int_{\IR^{d}}q_{s-u}\left(Ry-z\right)R\left\|\left(\di B\right)\left(z\right)-\left(\di B\right)\left(Ry\right)\right\|_{B\left(H,\dom\ A_{0}\right)}\nonumber\\
			&q_{u}\left(z-Ry\right)\Id z\ \frac{u^{d-2}}{\left(d-2\right)!} \Id u\ \Id S_{1}\left(y\right)\ \Id s\nonumber\\
			=&C'C^{d-1}\limsup_{R\to\infty}\int_{0}^{t}e^{cs}\int_{S_{1}\left(0\right)}\int_{0}^{s}\int_{\IR^{d}}R\left\|\left(\di B\right)\left(Ry+w\right)-\left(\di B\right)\left(Ry\right)\right\|_{B\left(H,\dom\ A_{0}\right)}\nonumber\\
            &q_{s-u}\left(w\right)q_{u}\left(w\right)\Id w\ \frac{u^{d-2}}{\left(d-2\right)!}\Id u\ \Id S_{1}\left(y\right)\ \Id s.
		\end{align}
		By the same reasoning as for (\ref{qforthetalemeq2}), we may take the limit $R\to\infty$ inside the integrals in (\ref{qforthetalemeq3}), which is $0$. Thus,
		\begin{align}
			&\lim_{r\to\infty}r^{d-1}\int_{0}^{t}\int_{S_{1}\left(0\right)}\tr_{\IC^{r}}\left(\Ii\langle c,y\rangle A\left(ry\right)p_{s}\left(ry\right)\right)\Id S_{1}\left(y\right)\ \Id s\nonumber\\
			=&\lim_{r\to\infty}r^{d-1}\int_{0}^{t}\int_{S_{1}\left(0\right)}\tr_{\IC^{r}}\left(\Ii\langle c,y\rangle A\left(ry\right)\right.\nonumber\\
			&\left.\int_{u\in s\Delta_{d-1}}\left(Q_{u}^{ry}\shuffle_{\ast}\left(N_{ry},\left(\Ii\di B\right)^{\otimes\left(d-2\right)}\right)\right)\left(ry,ry\right)\right)\Id S_{1}\left(y\right)\ \Id s.
		\end{align}
		By iterating the procedure, we obtain
		\begin{align}
			&\lim_{r\to\infty}r^{d-1}\int_{0}^{t}\int_{S_{1}\left(0\right)}\tr_{\IC^{r}}\left(\Ii\langle c,y\rangle A\left(ry\right)p_{s}\left(ry\right)\right)\Id S_{1}\left(y\right)\ \Id s\nonumber\\
			&=\lim_{r\to\infty}r^{d-1}\int_{0}^{t}\int_{S_{1}\left(0\right)}\tr_{\IC^{r}}\left(\Ii\langle c,y\rangle A\left(ry\right)\pi_{s}\left(ry\right)\right)\Id S_{1}\left(y\right)\ \Id s,
		\end{align}
		where for $t>0$, $x_{0}\in\IR^{d}$,
		\begin{align}
			&\pi_{t}\left(x_{0}\right):=\int_{s\in t\Delta_{d-1}}\left(Q_{s}^{x_{0}}\shuffle_{\ast}\left(N_{x_{0}}\right)^{\otimes\left(d-1\right)}\right)\left(x_{0},x_{0}\right)\Id s\nonumber\\
			=&\left(4\pi t\right)^{-\frac{d}{2}}\int_{s\in t\Delta_{d-1}}\left(\one_{\IC^{r}}^{\otimes d}\otimes e^{-sA\left(x_{0}\right)^{2}}\right)\shuffle\left(\Ii\left(\di B\right)\left(x_{0}\right)\right)^{\otimes\left(d-1\right)}\Id s,
		\end{align}
		The claim then follows if we apply $\tr_{\IC^{r}}$ to $\Ii\langle c,y\rangle\pi_{s}$, using Lemma \ref{cliflem}.
	\end{proof}
	
	\subsection{The main result}
	
	Throughout the previous chapters, we used tacitly the approximation family $B_{n,m,l}$ instead of $B$ for notational brevity. Now we are prepared to take the limits $n\to\infty$, $m\to\infty$, $l\to\infty$ to arrive at the principal trace formula, which is our main result, presented in the following Theorem.
    
	If we assume Hypothesis \ref{hyp1}, by Remark \ref{a0rem}, the model operator $A_{0}$ is allowed to be $A\left(x_{0}\right)$, for a.e. $x_{0}\in\IR^{d}$. We fix such a $x_{0}$, and set $A_{0}=A\left(x_{0}\right)$ in the following.

\begin{Theorem}\label{newmainthm}
	Assume Hypothesis \ref{hyp1}. For $\phi\in C^{\infty}_{c}\left(\IR^{d}\right)$ denote $A_{\phi}:=A_{0}+\left(1-\phi\right)\left(A-A_{0}\right)$. Then for $t>0$, and any $\phi\in C^{\infty}_{c}\left(\IR^{d}\right)$,
	\begin{align}\label{newmainthmeq-1}
		&\tr_{L^{2}\left(\IR^{d},H\right)}\tr_{\IC^{r}}\left(e^{-tD^{\ast}D}-e^{-tDD^{\ast}}\right)\nonumber\\
		=&\frac{2}{d}\left(4\pi\right)^{-\frac{d}{2}}\Ii^{d}\kappa_{c}t^{\frac{d}{2}}\int_{s\in\Delta_{d-1}}\int_{\IR^{d}}\tr_{H}\bigwedge_{j=0}^{d-1}\left(\left(\Id A_{\phi}\right)e^{-ts_{j}A_{\phi}^2}\right)\Id s,
	\end{align}
	where $\bigwedge$ denotes the exterior product.
\end{Theorem}

\begin{proof}
	Denote $A_{\ast}:=A_{n,m,l}$, and $D_{\ast}:=D_{n,m,l}$. Let $r=\left|x\right|$. For $x=ry\in\IR^{d}$ with $y\in S_{1}\left(0\right)$, we have $\Id r=\sum_{j=1}^{d}y^{j}\Id x^{j}$, and for $u\in s\Delta_{d-1}$, $s>0$,
	\begin{align}
			\Id r\wedge e^{-u_{0}A_{\ast}^{2}}\bigwedge_{j=1}^{d-1}\left(\Id A_{\ast}\right)e^{-u_{j}A_{\ast}^{2}}=\sum_{j=1}^{d}\sum_{\alpha\in\left\{1,\ldots,d\right\}^{d-1}}y^{j}\epsilon_{\left(j,\alpha\right)}e^{-uA_{\ast}^2}\shuffle\left(\p_{\alpha_{1}}^{E}A_{\ast},\ldots,\p_{\alpha_{d-1}}^{E}A_{\ast}\right)\Id x.
	\end{align}
	If we apply $\iota_{\Id r^{\sharp}}$, the interior multiplication with the radial field, to both sides, and restrict to $S_{R}\left(0\right)$, we obtain
	\begin{align}\label{newmainthmeq0}
		e^{-u_{0}A_{\ast}^{2}}\bigwedge_{j=1}^{d-1}\left(\Id A_{\ast}\right)e^{-u_{j}A_{\ast}^{2}}=\sum_{j=1}^{d}\sum_{\alpha\in\left\{1,\ldots,d\right\}^{d-1}}y^{j}\epsilon_{\left(j,\alpha\right)}e^{-uA_{\ast}^2}\shuffle\left(\p_{\alpha_{1}}^{E}A_{\ast},\ldots,\p_{\alpha_{d-1}}^{E}A_{\ast}\right)\Id S_{R}\left(x\right).
	\end{align}
	We insert (\ref{newmainthmeq0}) into Lemma \ref{qforthetalem}, and change variables in the $u$-integration,
	\begin{align}\label{newmainthmeq1}
		X_{t}:=&\tr_{L^{2}\left(\IR^{d},H\right)}\tr_{\IC^{r}}\left(e^{-tD^{\ast}_{\ast}D_{\ast}}-e^{-tD_{\ast}D^{\ast}_{\ast}}\right)\nonumber\\
		=&\Ii^{d}\left(4\pi\right)^{-\frac{d}{2}}\kappa_{c}\tr_{H}\left(\lim_{R\to\infty}\int_{0}^{t}s^{\frac{d}{2}-1}\int_{S_{R}\left(0\right)}A_{\ast}\int_{u\in\Delta_{d-1}}e^{-su_{0}A_{\ast}^{2}}\bigwedge_{j=1}^{d-1}\left(\Id A_{\ast}\right)e^{-su_{j}A_{\ast}^{2}}\Id u\ \Id s\right).
		\end{align}
		Due to the limit $R\to\infty$, we may replace $A_{\ast}$ by $A_{\phi,\ast}:=A_{0}+\left(1-\phi\right)\left(A_{\ast}-A_{0}\right)$, for any $\phi\in C^{\infty}_{c}\left(\IR^{d}\right)$, in (\ref{newmainthmeq1}). We note that the integrand is Bochner integrable in trace-class, since $\Id A_{\ast}$ is finite rank. We apply Stokes Theorem, and the trace $\tr_{H}$, and get
		\begin{align}\label{newmainthmeq2}
			X_{t}=&\Ii^{d}\left(4\pi\right)^{-\frac{d}{2}}\kappa_{c}\int_{0}^{t}s^{\frac{d}{2}-1}\int_{\IR^{d}}\left(\omega_{s}^{1}+\omega_{s}^{2}\right)\Id s,
		\end{align}
		where
		\begin{align}
			\omega_{s}^{1}&:=\int_{u\in\Delta_{d-1}}\tr_{H}\bigwedge_{j=0}^{d-1}\left(\left(\Id A_{\phi,\ast}\right)e^{-su_{j}A_{\phi,\ast}^2}\right)\Id u,\nonumber\\
			\omega_{s}^{2}&:=-2s\int_{u\in\Delta_{d}} \tr_{H}\left(A_{\phi,\ast}^{2}e^{-su_{0}A_{\phi,\ast}^{2}}\bigwedge_{j=1}^{d}\left(\Id A_{\phi,\ast}\right)e^{-su_{j}A_{\phi,\ast}^{2}}\right)\Id u\nonumber\\
			&=-2s\int_{u\in\Delta_{d-1}}u_{0}\tr_{H}\left(\left(\Id A_{\phi,\ast}\right) e^{-su_{0}A_{\phi,\ast}^2} A_{\phi,\ast}^{2}\bigwedge_{j=1}^{d-1}\left(\Id A_{\phi,\ast}\right)e^{-su_{j}A_{\phi,\ast}^2}\right)\Id u.
		\end{align}
		by cyclically commuting the factor $A_{\phi,\ast}^{2}$ under the trace and changing variables, we obtain for $\omega_{s}^{2}$,
		\begin{align}
			\omega_{s}^{2}=\frac{2}{d}s\p_{s}\omega_{s}^{1}.
		\end{align}
		Integrating by parts in (\ref{newmainthmeq2}), we thus obtain,
		\begin{align}\label{newmainthmeq3}
			X_{t}=&\frac{2}{d}\left(4\pi\right)^{-\frac{d}{2}}t^{\frac{d}{2}}\Ii^{d}\kappa_{c}\int_{\IR^{d}}\omega^{1}_{t}.
		\end{align}
		The argument of the trace of the $n,m,l$-dependent $\omega_{t}^{1}$ possesses iteratively integrable dominants in trace-norm. Indeed, first for fixed $n,m\in\IN$ there exists a constant $C$ independent of $l$ (because $K_{m}$ is finite rank), such that
		\begin{align}\label{newmainthmeq4}
			&\left\|\bigwedge_{j=0}^{d-1}\left(\left(\Id A_{\phi,n,m,l}\right)e^{-su_{j}A_{\phi,n,m,l}^2}\right)\left(x\right)\right\|_{S^{1}\left(H\right)}\leq C\left\|\delta_{l}\left(K_{m}P_{n}BP_{n}K_{m}\right)\right\|_{\mathcal{F}_{\epsilon}}^{d}\langle x\rangle^{-d-\epsilon}\Id x\nonumber\\
            &\stackrel{\text{Lemma \ref{radialconvolutionlem}}}{\leq} C'\left\|K_{m}P_{n}BP_{n}K_{m}\right\|_{\mathcal{F}_{\epsilon}}^{d}\langle x\rangle^{-d-\epsilon}\Id x,
		\end{align}
        which is an integrable dominant on $\IR^{d}\times\Delta_{d-1}$, and thus allows us to take the limit $l\to\infty$ under the integral.
        
        Now fix $n$. Since $B\left(x_{0}\right)=A\left(x_{0}\right)-A_{0}=0$, the operator function $\langle A_{0}\rangle^{-\alpha+\beta}B\langle A_{0}\rangle^{-\beta}$ is Lipschitz with values in $S^{d}\left(H\right)$, for $\alpha,\beta$ as in Hypothesis \ref{hyp1}. Since $\nabla\phi$ is compactly supported it follows that $\left(\nabla\phi\right)P_{n}BP_{n}\in C_{c}\left(\IR^{d},S^{d}\left(H\right)\right)$. Hence there is a constant $C$ independent of $m$, such that
        \begin{align}
            \left\|\nu^{E}\left(\left(1-\phi\right)K_{m}P_{n}BK_{m}P_{n}\right)\left(x\right)\right\|_{S^{d}\left(H\right)}\leq C\langle x\rangle^{-1-\eta},
        \end{align}
        where $\nu$ is a vector field, $\eta\geq 0$, and $\eta=\epsilon$ for $\nu=\p_{R}$. Thus
        \begin{align}\label{newmainthmeq5}
			\left\|\bigwedge_{j=0}^{d-1}\left(\left(\Id A_{\phi,n,m}\right)e^{-su_{j}A_{\phi,n,m}^2}\right)\left(x\right)\right\|_{S^{1}\left(H\right)}\leq C^{d}\langle x\rangle^{-d-\epsilon}\Id x,
		\end{align}
    which is integrable on $\IR^{d}\times\Delta_{d-1}$, which allows us to take the limit $m\to\infty$ under the integral.
    
    For the last limit $n\to\infty$, we partition $\Delta_{d-1}$ via $\Delta_{d-1}=\bigcup_{k=0}^{d-1}M_{k}=\dot{\bigcup}_{k=0}^{d-1}N_{k}$, where $M_{k}:=\left\{u\in\Delta_{d-1},u_{k}\geq\frac{1}{d}\right\}$, $N_{k}:=M_{k}\backslash\left(\bigcup_{l=0}^{k-1}M_{l}\right)$. For $u\in N_{k}$, we have by Hypothesis \ref{hyp1} a constant $C$ independent of $n$, such that for $t>0$,
		\begin{align}
			\left\|\bigwedge_{j=0}^{d-1}\left(\left(\Id A_{\phi,n}\right)e^{-su_{j}A_{\phi,n}^2}\right)\right\|_{S^{1}\left(H\right)}\leq C\prod_{j=0}^{d-1}\tau^{\alpha,\beta_{j},\nu_{j},d,\delta_{j}}\left(\left(1-\phi\right)B\right)\sup_{\lambda\in\IR}\left(\langle\lambda\rangle^{-d\alpha}e^{-tu_{k}\lambda^2}\right)\langle x\rangle^{-d-\epsilon}\Id x,
		\end{align}
		where $\beta_{j}=j\alpha$, $\nu_{j}=\nabla$ for $j\neq k$, $\nu_{k}=\p_{R}$, and $\delta_{j}=1$ for $j\neq k$, and $\delta_{k}=1+\epsilon$. In the estimate we used that $P_{n}$ commutes with $A_{0}$, and, as seen before, that $\langle A_{0}\rangle^{-\alpha+\beta_{j}}\left(\nu_{j}\phi\right)B\langle A_{0}\rangle^{-\beta_{j}}\in C_{c}\left(\IR^{d},S^{d}\left(H\right)\right)$, which implies that the involved semi-norms $\tau^{\bullet}\left(\left(1-\phi\right)B\right)$ are finite. We may interchange the limit $n\to\infty$ with the trace and the integrals integrals in (\ref{newmainthmeq3}) and obtain the claimed formula (\ref{newmainthmeq-1}).
\end{proof}
	
If the one-forms $\Id A$ admit radial limits, we may pass to these limits in the above trace formula.

\begin{Corollary}\label{newcor}
	Assume Hypothesis \ref{hyp1}, and denote $A^{\circ}\left(y\right)\psi:=\lim_{R\to\infty}A\left(Ry\right)\psi$, for $\psi\in\dom\ A_{0}$, and $y\in S_{1}\left(0\right)$. Assume that $A^{\circ}\psi$ is differentiable on $S_{1}\left(0\right)$ for $\psi\in\dom\ A_{0}$, such that for a.e. $y\in S_{1}\left(0\right)$,
	\begin{align}\label{mainlimitcoreq1}
		\lim_{R\to\infty}R\left\|\p_{i}A^{\circ}\left(y\right)\psi-\p_{i}A\left(Ry\right)\psi\right\|_{H}=0.
	\end{align}
	and that for $\beta\in\left[-2N+\alpha,2N+1\right]$, and all unit vector fields $\nu$ in $TS_{1}\left(0\right)$,
	\begin{align}\label{newcoreq0}
		\langle A_{0}\rangle^{-\alpha+\beta}\left(\nu A^{\circ}\right)\langle A_{0}\rangle^{-\beta}\in L^{\infty}\left(S_{1}\left(0\right),S^{d}\left(H\right)\right).
	\end{align}
	Denote for $\rho\in C^{\infty}_{c}\left(\IR\right)$,
	\begin{align}
		\widetilde{A^{\circ}}_{\rho}\left(x\right)&:=A_{0}+\left(1-\rho\left(\left|x\right|\right)\right)\left(A^{\circ}\left(\frac{x}{\left|x\right|}\right)-A_{0}\right).
	\end{align}
	Then for $t>0$, and any $\rho\in C^{\infty}_{c}\left(\IR\right)$, with $\rho\equiv 1$ near $0$,
	\begin{align}
		&\tr_{L^{2}\left(\IR^{d},H\right)}\tr_{\IC^{r}}\left(e^{-tD^{\ast}D}-e^{-tDD^{\ast}}\right)\nonumber\\
		=&\frac{2}{d}\left(4\pi\right)^{-\frac{d}{2}}\Ii^{d}\kappa_{c}t^{\frac{d}{2}}\int_{s\in\Delta_{d-1}}\int_{\IR^{d}}\tr_{H}\bigwedge_{j=0}^{d-1}\left(\left(\Id\widetilde{A^{\circ}}_{\rho}\right)e^{-ts_{j}\widetilde{A^{\circ}}_{\rho}^2}\right)\Id u.
	\end{align}
\end{Corollary}

\begin{proof}
	We first note that $A^{\circ}$ is well-defined. Indeed, for a.e. $y\in S_{1}\left(0\right)$,
	\begin{align}
		A\left(Ry\right)-A\left(0\right)=\int_{0}^{R}\left(\p_{R}A\right)\left(sy\right)\Id s,
	\end{align}
	and by Hypothesis \ref{hyp1} there is a constant $C$, such that for $\beta\in\left[-2N+\alpha,2N+1\right]$,
	\begin{align}
		\left\|\langle A_{0}\rangle^{-\alpha+\beta}\left(\p_{R}A\right)\left(sy\right)\langle A_{0}\rangle^{-\beta}\right\|_{S^{d}\left(H\right)}\leq C\langle s\rangle^{-1-\epsilon},
	\end{align}
	which is integrable on $s\in\left(0,\infty\right)$. Moreover we estimate for $y\in S_{1}\left(0\right)$, and $k\in\left[-2N,2N\right]$ by Hypothesis \ref{hyp1},
	\begin{align}\label{newcoreq1}
		&\left\|\langle A_{0}\rangle^{k}\left(A^{\circ}\left(y\right)-A_{0}\right)\langle A_{0}\rangle^{-k}\langle A_{0}\rangle^{-1}_{z}\right\|_{B\left(H\right)}\leq\sup_{x\in\IR^{d}}\left\|\langle A_{0}\rangle^{k}\left(A\left(x\right)-A_{0}\right)\langle A_{0}\rangle^{-k}\langle A_{0}\rangle^{-1}_{z}\right\|_{B\left(H\right)}\nonumber\\
		&\xrightarrow{z\to\infty}0,
	\end{align}
	which implies that $A^{\circ}\left(y\right)$ is self-adjoint on $\dom\ A_{0}$ and that $\dom\ \left(A^{\circ}\left(y\right)\right)^{k}=\dom\ A_{0}^{k}$, for $k\in\left[0,2N\right]$, and that the graph norms are equivalent, uniformly in $y\in S_{1}\left(0\right)$ by Kato-Rellich. Similarly we introduce for $n,m\in\IN$, $\rho\in C^{\infty}_{c}\left(\IR\right)$,
	\begin{align}
		\widetilde{A^{\circ}}_{n,m,l,\rho}:=A_{0}+\delta_{l}\left(K_{m}P_{n}\left(\widetilde{A^{\circ}}_{\rho}-A_{0}\right)P_{n}K_{m}\right).
	\end{align}
	Analogously to (\ref{newcoreq1}) we conclude that the graph norms of $\widetilde{A^{\circ}}_{n,m,l,\rho}^{k}$ and $A_{0}^{k}$ are equivalent, uniformly in $n,m,l\in\IN$ for $k\in\left[0,2N\right]$. In the proof of Theorem \ref{newmainthm} we replace, due to the limit $R\to\infty$, in (\ref{newmainthmeq1}), the terms $A_{\ast}$ by $\widetilde{A^{\circ}}_{n,m,l,\rho}$ and $\Id A_{\ast}$ by $\Id\widetilde{A^{\circ}}_{n,m,l,\rho}$, which follows from the additional condition (\ref{mainlimitcoreq1}), and dominated convergence after scaling coordinates from $S_{R}\left(0\right)$ to $S_{1}\left(0\right)$. The presence of $K_{m}$ ensures that the limit $R\to\infty$ of the integrand is still in trace-class. We now proceed similarly to the proof of Theorem \ref{newmainthm}, and obtain
	\begin{align}
		&\tr_{L^{2}\left(\IR^{d},H\right)}\tr_{\IC^{r}}\left(e^{-tD^{\ast}D}-e^{-tDD^{\ast}}\right)\nonumber\\
		=&\lim_{n\to\infty}\lim_{m\to\infty}\lim_{l\to\infty}\frac{2}{d}\left(4\pi\right)^{-\frac{d}{2}}\Ii^{d}\kappa_{c}t^{\frac{d}{2}}\int_{s\in\Delta_{d-1}}\int_{\IR^{d}}\tr_{H}\bigwedge_{j=0}^{d-1}\left(\left(\Id\widetilde{A^{\circ}}_{n,m,l,\rho}\right)e^{-ts_{j}\widetilde{A^{\circ}}_{n,m,l,\rho}^2}\right)\Id u.
	\end{align}
	It remains to evaluate the iterated limits $n,m,l\to\infty$. Since the graph norms of the powers of $\widetilde{A^{\circ}}_{n,m,l,\rho}$ are uniformly equivalent to those of $A_{0}$, the limiting procedure is similar to the one in the proof of Theorem \ref{newmainthm}, if we show that the operator family $\widetilde{A^{\circ}}_{\rho}$ satisfies for $\beta\in\left[-2N+\alpha,2N+1\right]$,
	\begin{align}\label{newcoreq2}
		\tau^{\alpha,\beta,\nabla,d,1}\left(\widetilde{A^{\circ}}_{\rho}-A_{0}\right)<&\infty,\nonumber\\
		\exists\epsilon>0:\tau^{\alpha,\beta,\p_{R},d,1+\epsilon}\left(\widetilde{A^{\circ}}_{\rho}-A_{0}\right)<&\infty.
	\end{align}
	Indeed, by (\ref{newcoreq0}) we have a constant $C$, such that
	\begin{align}
		\left\|\langle A_{0}\rangle^{-\alpha+\beta}\left(A^{\circ}\left(\frac{x}{\left|x\right|}\right)-A_{0}\right)\langle A_{0}\rangle^{-\beta}\right\|_{S^{d}\left(H\right)}\leq C.
	\end{align}
	The fact that $\rho'$ is compactly supported, and condition (\ref{newcoreq0}) imply by direct calculation (\ref{newcoreq2}) for any choice of $\epsilon>0$.
\end{proof}

\section{Applications}

	The principal trace formulas displayed in Theorem \ref{newmainthm} and in Corollary \ref{newcor} enable us to establish a condition of the existence, and a formula for the index of $D$. At this stage we introduce the following regularized notion of an index, which might still exist even if $D$ fails to be Fredholm, which is called the (partial) Witten index. Without the presence of the partial trace, its notion was introduced in \cite{GesSim}.
	
	\begin{Definition}
		Let $T$ be a linear, densely defined operator $\IC^{r}\otimes X$ for some separable complex Hilbert space $X$. We say the partial (semi-group regularized) Witten index of $T$ exists, if for some $t_{0}>0$,
		\begin{align}
			\tr_{\IC^{r}}\left(e^{-t T^{\ast}T}-e^{-tTT^{\ast}}\right)\in S^1\left(X\right),\ t\geq t_{0},
		\end{align}
		and if the limit
		\begin{align}
			\ind_{W}T:=\lim_{t\to\infty}\tr_{X}\left(\tr_{\IC^{r}}\left(e^{-t T^{\ast}T}-e^{-tTT^{\ast}}\right)\right),
		\end{align}
		exists. Then $\ind_{W}T$ is called the partial Witten index of $T$.
	\end{Definition}

	\begin{Remark}\label{wittenrem}
		Similar to the proof of \cite{Cal}[Lemma 1], one obtains the analogous result for the partial Witten index. Namely, if for some $t_{0}>0$, the condition
		\begin{align}
			\tr_{\IC^{r}}\left(e^{-t T^{\ast}T}-e^{-tTT^{\ast}}\right)\in S^1\left(X\right),\ t\geq t_{0}
		\end{align}
        holds, and $T$ is additionally Fredholm, then $T$ admits a partial Witten index, $\ind_{W}T$, which coincides with the Fredholm index $\ind T=\dim\ker T-\dim\coker T$.
	\end{Remark}

	It is clear that $\ind_{W}D$ exists, whenever the right hand sides of the principal trace formulas in Theorem \ref{newmainthm} or Corollary \ref{newcor} admit a limit for $t\to\infty$. If $D$ is Fredholm, this is the case by Remark \ref{wittenrem}. However for the interesting case where $D$ is non-Fredholm, we also present an example for which the Witten index exists, and can be calculated directly in terms of its symbol. We start with the Fredholm result first.
	
	\begin{Theorem}\label{fredindexthm}
		Assume Hypothesis \ref{hyp1}. Additionally assume that there exists $R_{0}\geq 0$, such that for $\left|x\right|\geq R_{0}$, we have $0\in\rho\left(A\left(x\right)\right)$. Then $D$ is Fredholm, and the Fredholm index $\ind D$ is given by
		\begin{align}\label{fredindexthmeq1}
			\ind_{W}D=\frac{2}{d}\left(4\pi\right)^{-\frac{d}{2}}\Ii^{d}\kappa_{c}\lim_{t\to\infty}t^{\frac{d}{2}}\int_{s\in\Delta_{d-1}}\int_{\IR^{d}}\tr_{H}\bigwedge_{j=0}^{d-1}\left(\left(\Id A_{\phi}\right)e^{-ts_{j}A_{\phi}^2}\right)\Id s,
		\end{align}
		where $A_{\phi}:=A_{0}+\left(1-\phi\right)\left(A-A_{0}\right)$ for any $\phi\in C^{\infty}_{c}\left(\IR^{d}\right)$, and the formula is independent of the choice of $\phi$.
	\end{Theorem}

	\begin{proof}
		If $D$ is Fredholm, $D$ admits a partial Witten index, according to Remark \ref{wittenrem}, and it equals the right hand side of (\ref{fredindexthmeq1}) by Theorem \ref{newmainthm}. We thus only need to show that $D$ is Fredholm. Since $A\left(x\right)$ is invertible for $\left|x\right|$ large enough, let $C>0$ be a constant such that $A\left(x\right)^{2}\geq C$, for $\left|x\right|\geq R_{0}$. Note that $D^{\ast}D=\Delta+A^{2}-\Ii\left(\di B\right)$ as differential operator. On the domain of $D^{\ast}D$, we thus have some constant $C'$, such that
		\begin{align}\label{fredindexthmeq2}
			D^{\ast}D\geq\Delta+C-\frac{C'}{\langle X\rangle},
		\end{align}
		where we used that $A^2\geq 0$, and $\Ii\left(\di B\right)$ is bounded, and decays like $\left|x\right|^{-1}$ for $\left|x\right|\to\infty$. By Seeley's Theorem (\cite{Cal}[Theorem 1]), the elliptic differential operator $\Delta+C-\frac{C'}{\langle X\rangle}$ is Fredholm. The operator inequality (\ref{fredindexthmeq2}) then implies that $D^{\ast}D$ is also Fredholm. The analogous argument works for $DD^{\ast}$ as well, hence $D$ is Fredholm. 
	\end{proof}
	
	The following examples are presented in more detail in \cite{F2}, here we merely quote the results without the proofs. The basic object is a Dirac-Schr\"odinger operator $D_{V}$ which is non-Fredholm.
	
	\begin{Definition}
		For a potential $V:\IR^{d}\times\IR\rightarrow B_{sa}\left(H\right)$, define the associated \textit{massless $\left(d+1\right)$-Dirac-Schr\"odinger operator} $D_{V}$,
		\begin{align}
			\left(D_{V}f\right)\left(x,y\right):=&i\sum_{j=1}^{d}c^{j}\p_{x^{j}}f\left(x,y\right)+i\p_{y}f\left(x,y\right)+V\left(x,y\right)f\left(x,y\right),\ x\in\IR^{d},\ y\in\IR,\nonumber\\
			f\in& W^{1,2}\left(\IR^{d+1},\IC^{r}\otimes H\right),
		\end{align}
	\end{Definition}
	
	\begin{Example}[\cite{F2}, Theorem 1.1]\label{theexample}
		Let $V\in C_{b}^{2}\left(\IR^{d}\times\IR,B_{sa}\left(H\right)\right)$, where $B_{sa}\left(H\right)$ is equipped with the strong operator topology. Assume furthermore that for $i\in\left\{1,\ldots,d\right\}$ and $k\in\left\{0,1,2\right\}$,
		\begin{align}
			\left\|y\mapsto\p_{x^{i}}\p_{y}^{k}V\left(x,y\right)\right\|_{L^{d}\left(\IR,S^{d}\left(H\right)\right)}&=O\left(\left|x\right|^{-1}\right),\ \left|x\right|\to\infty,\nonumber\\
			\exists\epsilon>0:\ \left\|y\mapsto\p_{R}\p_{y}^{k}V\left(x,y\right)\right\|_{L^{d}\left(\IR,S^{d}\left(H\right)\right)}&=O\left(\left|x\right|^{-1-\epsilon}\right),\ \left|x\right|\to\infty,
		\end{align}
		and that for any $\widehat{x}\in S_{1}\left(0\right)$, $x_{0}\in\IR^{d}$, and $\gamma\in\IN_{0}^{d}$ with $\left|\gamma\right|\leq 1$, we have
		\begin{align}
			\lim_{R\to\infty}R^{\left|\gamma\right|}\left\|\left(\p^{\gamma}_{x}V\right)\left(x_{0}+R\widehat{x},\cdot\right)-\left(\p^{\gamma}_{x}V\right)\left(R\widehat{x},\cdot\right)\right\|_{L^{\infty}\left(\IR,B\left(H\right)\right)}=0.
		\end{align}
		Then
		\begin{align}\label{thethmeq1}
			\ind_{W}D_{V}=\frac{1}{2\pi}\left(4\pi\right)^{-\frac{d-1}{2}}\frac{\left(\frac{d-1}{2}\right)!}{d!}\kappa_{c}\int_{\IR^{d}}\tr_{H}\left(\left(U^{V}\right)^{-1}\Id U^{V}\right)^{\wedge d},
		\end{align}
		where $\wedge d$ is the $d$-fold exterior power,
		\begin{align}
			\kappa_{c}:=\tr_{\IC^{r}}\left(c^{1}\cdot\ldots\cdot c^{d}\right),
		\end{align}
		and the unitary $U^{V}$ is given by $U^{V}\left(x\right):=\lim_{n\to\infty}U^{V\left(x,\cdot\right)}\left(n,-n\right)$, $x\in\IR^{d}$, where $U^{A}\left(y_{1},y_{2}\right)$, $y_{1},y_{2}\in\IR$, for a given $A:\IR\rightarrow B_{sa}\left(H\right)$ is the (unique) evolutions system of the equation
		\begin{align}
			u'\left(y\right)=\Ii A\left(y\right) u\left(y\right),\ y\in\IR,
		\end{align}
		i.e. for $y,y_{1},y_{2},y_{3}\in\IR$,
		\begin{align}
			\p_{y_{1}}U^{A}\left(y_{1},y_{2}\right)=&\Ii A\left(y_{1}\right)U^{A}\left(y_{1},y_{2}\right),\\
			\p_{y_{2}}U^{A}\left(y_{1},y_{2}\right)=&-\Ii U^{A}\left(y_{1},y_{2}\right)A\left(y_{2}\right),\\
			U^{A}\left(y_{1},y_{2}\right)U^{A}\left(y_{2},y_{3}\right)=&U^{A}\left(y_{1},y_{3}\right),\\
			U^{A}\left(y,y\right)=&1_{H}.
		\end{align}
		
		With the minimal choice $r=2^{\frac{d-1}{2}}$, and sign convention\footnote{This corresponds to the choice $c^{j}=-\Ii\sigma^{j}$, where $\sigma^{j}$ are classical Pauli matrices.} $\kappa_{c}=\left(-\Ii\right)^{d}\left(2\Ii\right)^{\frac{d-1}{2}}$, we have
		\begin{align}\label{ththmeq2}
			\ind_{W}D_{V}=\left(2\pi\Ii\right)^{-\frac{d+1}{2}}\frac{\left(\frac{d-1}{2}\right)!}{d!}\int_{\IR^{d}}\tr_{H}\left(\left(U^{V}\right)^{-1}\Id U^{V}\right)^{\wedge d}.
		\end{align}
	\end{Example}
	
	This formula generalizes an index formula found in \cite{BolGesGroSchSim}, \cite{CGGLPSZ}, and \cite{CGLPSZ} for the case $d=1$ to $d\geq 3$ odd. For a special choice of the potential $V$, the index formula becomes even more concrete. We shall close the paper with this example.
	
		\begin{Definition}\label{potentialexdef}
		Let $F:\IR^{d}\rightarrow\IR^{d}$ be a bounded smooth function with bounded derivatives, such that
		\begin{align}
			\p_{x^{i}}F\left(x\right)&=O\left(\left|x\right|^{-1}\right),\ \left|x\right|\to\infty,\nonumber\\
			\exists\epsilon>0:\ \p_{R}F\left(x\right)&=O\left(\left|x\right|^{-1-\epsilon}\right),\ \left|x\right|\to\infty,
		\end{align}
		and that for any $\widehat{x}\in S_{1}\left(0\right)$, $x_{0}\in\IR^{d}$, and $\gamma\in\IN_{0}^{d}$ with $\left|\gamma\right|\leq 1$,
		\begin{align}
			\lim_{R\to\infty}R^{\left|\gamma\right|}\left|\left(\p^{\gamma}F\right)\left(x_{0}+R\widehat{x}\right)-\left(\p^{\gamma}F\right)\left(R\widehat{x}\right)\right|=0.
		\end{align}
		Set $V\left(x,y\right):=\Ii\sum_{j=1}^{d}c^{j}F_{j}\left(x\right)\phi\left(y\right)$, $x\in\IR^{d}$, $y\in\IR$, for some function $\phi\in C_{c}^{\infty}\left(\IR\right)$ with $\int_{\IR}\phi\left(y\right)\Id y=1$.
	\end{Definition}
	
	\begin{Example}[\cite{F2}, Proposition 1.2]\label{exampleminor}
		With $V$ as in Definition \ref{potentialexdef}, with the minimal choices $\mathrm{rank}\left(c^{i}\right)=2^{\frac{d-1}{2}}$ and sign convention $\kappa_{c}=\left(-\Ii\right)^{d}\left(2\Ii\right)^{\frac{d-1}{2}}$, the partial Witten index of $D_{V}$ exists, and
		\begin{align}
			\ind_{W}D_{V}=\left(2\pi\right)^{-\frac{d+1}{2}}\frac{\left(\frac{d-1}{2}\right)!}{d}\int_{\IR^{d}}\frac{\left(\left|F\right|+d-1\right)\left(\cos\left(2\left|F\right|\right)-1\right)^{\frac{d-1}{2}}}{\left|F\right|^{d}}\det\mathrm{D}F\ \Id x,
		\end{align}
		where $\mathrm{D}F$ is the Jacobian of $F$. Note that the integrand extends continuously to $F=0$.
	\end{Example}
	
	The above formula allows one to show that $\ind_{W}$ may assume any real number for this class of operators $D_{V}$ if one chooses appropriate diffeomorphisms of $\IR^{d}$ onto balls centred around $0$ for the function $F$ (cf. \cite{F2}[Corollary 4.6]).
	
	\section{Acknowledgements}
	
	I would like to thank Elmar Schrohe, Matthias Lesch, Lennart Ronge, Robert Fulsche and Philipp Schmitt for many fruitful discussion concerned with this paper.

\end{document}